%% file: Tannakian_Gerbes.tex
\documentclass[12pt]{amsart}
\usepackage{mathrsfs}
\usepackage{mathabx}
\usepackage{enumerate}
\usepackage{bbm}
\usepackage{mathrsfs, mathtools}
\usepackage{stmaryrd}
\usepackage{amsmath,amssymb,amsfonts}
\usepackage[all,cmtip]{xy}
\usepackage{a4wide}
\usepackage[mathscr]{eucal}
\usepackage[toc,page]{appendix}

\input{packages_and_functions.for_preamble}

\usepackage[colorinlistoftodos]{todonotes}

\theoremstyle{plain}
\newtheorem{thm}{Theorem}[section]
\newtheorem{lem}[thm]{Lemma}
\newtheorem{cor}[thm]{Corollary}
\newtheorem{prop}[thm]{Proposition}

\theoremstyle{definition}

\newtheorem{defn}[thm]{Definition} 
\newtheorem{ex}[thm]{Example}
\newtheorem{ax}[thm]{Axioms}

\newtheorem{rmk}[thm]{Remark}

\numberwithin{thm}{section}
\newcounter{x}

\newcommand{\red}{{\rm red}}

\newcommand{\EF}{{\rm EFin}}
\newcommand{\Et}{{\text{\rm \'Et}}}

\newcommand{\Fib}{{\rm Fib}}

\newcommand{\Ess}{{\rm EFin}}

\newcommand{\Strat}{{\rm Str}}
\newcommand{\Hom}{{\rm Hom}}

\newcommand{\Spec}{{\rm Spec \,}}
\newcommand{\Crys}{{\rm Crys}}
\newcommand{\Dmod}{{\rm Dmod}}
\newcommand{\Aff}{{\rm Aff}}

\newcommand{\Gal}{{\rm Gal}}
\newcommand{\Fdiv}{{\rm Fdiv}}
\newcommand{\Ker}{{\rm Ker}}

\newcommand{\sC}{{\mathcal C}}
\newcommand{\sD}{{\mathcal D}}

\newcommand{\sF}{{\mathcal F}}
\newcommand{\sG}{{\mathcal G}}

\newcommand{\sO}{{\mathcal O}}

\newcommand{\sT}{{\mathcal T}}

\newcommand{\sX}{{\mathcal X}}
\newcommand{\sY}{{\mathcal Y}}
\newcommand{\sZ}{{\mathcal Z}}
\newcommand{\A}{{\mathbb A}}

\newcommand{\C}{{\mathbb C}}

\newcommand{\E}{{\mathbb E}}
\newcommand{\F}{{\mathbb F}}

\newcommand{\N}{{\mathbb N}}

\newcommand{\Q}{{\mathbb Q}}

\newcommand{\Z}{{\mathbb Z}}

\newcommand{\Vect}{\text{\sf Vect}}
\newcommand{\Rep}{\text{\sf Rep}}
\newcommand{\id}{{\rm id\hspace{.1ex}}}

\renewcommand{\et}{\textup{\'et}}
\newcommand{\str}{\textup{str}}
\newcommand{\iinf}{\textup{inf}}
\newcommand{\perf}{\textup{perf}}
\newcommand{\NN}{\textup{N}}
\newcommand{\LL}{\textup{L}}
 \renewcommand{\todo}{\phantom}

\theoremstyle{plain}
\newtheorem{thmI}{Theorem}
\newtheorem{thmII}{Theorem}

\title{Algebraic and Nori fundamental gerbes}
\begin{document}

\author{Fabio Tonini,  Lei Zhang }

\address{ Fabio Tonini\\
    Freie Universit\"at Berlin\\
    FB Mathematik und Informatik\\
    Arnimallee 3\\ Zimmer 112A\\
    14195 Berlin\\ Deutschland }
\email{tonini@zedat.fu-berlin.de}
 \address{ Lei Zhang\\
    Freie Universit\"at Berlin\\
    FB Mathematik und Informatik\\
    Arnimallee 3\\ Zimmer 112A\\
    14195 Berlin\\ Deutschland }
\email{l.zhang@fu-berlin.de}
\thanks{This work was supported by the European Research Council (ERC) Advanced Grant 0419744101 and the Einstein Foundation}
\date{\today}

\input{packages_and_functions.tex}


\begin{abstract}
In this paper we extend the generalized  algebraic fundamental group  constructed in \cite{EH} to general fibered categories using the language of gerbes. As an application we obtain a Tannakian interpretation for the Nori fundamental gerbe defined in \cite{BV} for non smooth non pseudo-proper algebraic stacks. 
\end{abstract}

\setcounter{section}{0}
\maketitle

\section*{Introduction}
Let $k$ be a field and let $X$ be a smooth and connected  scheme  over  $k$ with a rational point $x\in X(k)$. The algebraic fundamental group of $(X,x)$, denoted by $\pi^{\rm alg}(X,x)$ is the affine group scheme over $k$ associated with the $k$-Tannakian category $\Dmod(X/k)$ of $\sO_X$-coherent $D_{X/k}$-modules neutralized by the pullback along $x\colon \Spec k\arr X$. If $k$ is algebraically closed then the profinite quotient of $\pi^{\rm alg}(X,x)$ is $\pi_1^{\text{\'et}}(X,x)$,  Grothendieck's \'etale fundamental group developed in \cite{SGA1}. 

On the other hand if $X$ is a connected and  reduced scheme over $k$ with a rational point $x\in X(k)$, Nori defined in \cite{Nori} a profinite fundamental group scheme $\pi^\NN(X,x)$ over $k$ which classifies torsors over $X$ by finite group schemes of $k$ with a trivialization over $x$. If $k$ is algebraically closed then its pro\'etale quotient is again $\pi_1^{\text{\'et}}(X,x)$, so that if $X$ is smooth we have maps
  \[\label{diagram introduction}
  \begin{tikzpicture}[xscale=3.0,yscale=-1.4]
    \node (A0_0) at (0, 0) {$\pi^{\rm alg,\infty}(X,x)$};
    \node (A0_1) at (1, 0) {$\pi^\NN(X,x)$};
    \node (A1_0) at (0, 1) {$\pi^{\rm alg}(X,x)$};
    \node (A1_1) at (1, 1) {$\pi_1^{\text{\'et}}(X,x)$};
    \path (A0_0) edge [->>,dashed]node [auto] {$\scriptstyle{c}$} (A0_1);
    \path (A1_0) edge [->>]node [auto] {$\scriptstyle{a}$} (A1_1);
    \path (A0_1) edge [->>]node [auto] {$\scriptstyle{b}$} (A1_1);
    \path (A0_0) edge [->>,dashed]node [auto] {$\scriptstyle{d}$} (A1_0);
  \end{tikzpicture}
  \]
where $a$ is the profinite quotient and $b$ is the pro-\'etale quotient (and thus an isomorphism if $\car k=0$). If $\car k>0$, in \cite{EH} Esnault and Hogadi completed this diagram with dashed arrows from an affine group schemes $\pi^{\rm alg,\infty}(X,x)$ associated with a Tannakian category denoted by $\textup{Strat}(X,\infty)$, with $c$ a profinite quotient and $d$ a quotient.

In this paper we would like to generalize the above picture to certain fibered categories over a field $k$ which may not possess a rational point, and this applies in particular to algebraic stacks which are not necessarily smooth.  To achieve this we will  use the language of gerbes instead of that of affine group schemes,  just as how Borne and Vistoli generalized the Nori fundamental group scheme to  fibered categories in \cite{BV}.

For smooth schemes $X$ there are several equivalent descriptions of the category   of $\sO_X$ coherent $D_X$-modules,  for instance the category $\Crys(X)$ of crystals on the infinitesimal site of $X$, or the category $\Strat(X)$ of stratified bundles, or, in positive characteristic, the category $\Fdiv(X)$ of $F$-divided sheaves (see \cite[Prop. 2.11, pp. 2.13]{BO} and \cite[Thm. 1.3, pp. 4]{Gie}). 

Let $\stX$ be a quasi-compact, quasi-separated and connected category fibered in groupoids over $k$ (see \ref{connected and reduced fibered categories} and the section   \emph{Notations and Conventions} for the meaning of those adjectives).
In order to define an algebraic fundamental gerbe in general, we are going to define $k$-linear   monoidal categories $\Crys(\stX)$, $\Strat(\stX)$, and, in positive characteristic, $\Fdiv(\stX)$, and discuss when those are Tannakian categories. More precisely we will define the \emph{big infinitesimal site} $\stX_{\iinf}$ of $\stX$, the \emph{big stratified site} $\stX_{\str}$ of $\stX$ and the direct limit $\stX^{(\infty,k)}$ of relative Frobenius of $\stX$. These are fibered categories over $k$ equipped with a morphism from $\stX$. The categories $\Crys(\stX)$, $\Strat(\stX)$ and $\Fdiv(\stX)$ are then defined as $\Vect(\stX_{\iinf})$, $\Vect(\stX_{\str})$ and $\Vect(\stX^{(\infty,k)})$ (see \ref{definition of stratifications and crystals} and \ref{definition of Fdivided sheaves}), where $\Vect(-)$ denotes the category of vector bundles (see the section \emph{Notations and Conventions} for its definition).
Since those categories are not equivalent in general when $\stX$ is not smooth, we develop an axiomatic language which allow to treat all of them together. The advantage of this language is that all functors involved will be expressed as pullback of vector bundles along certain maps, making proof easier and more conceptual. 

Let $\stX\arr\stX_\shT$ be a  morphism of fibered categories over $k$, and let $\shT(\stX)=\Vect(\stX_\shT)$. We will list four axioms A,B,C and D on the given morphism $\stX\arr \stX_\shT$ or, to simplify the exposition, on $\shT(\stX)$ (see \ref{axioms}) which imply nice ``Tannakian'' properties of $\shT(\stX)$. Denote by $L_0$ the endomorphisms of the unit object of $\shT(\stX)$, that is $L_0=\Hl^0(\odi{\stX_\shT})$ and, if $\shC$ is a $k$-Tannakian category, denote by $\Pi_\shC$ the associated affine gerbe over $k$. For instance A and B imply that $L_0$ is a field, that $\shT(\stX)$ is an $L_0$-Tannakian category and, moreover, that $\Pi_{\shT(\stX)}$ has the following universal property: there is an $L_0$-map $\stX_\shT\arr \Pi_{\shT(\stX)}$ which is universal among $L_0$-morphisms from $\stX_\shT$ to an affine gerbe over $L_0$ (see \ref{when T(X) is Tannakian}). 

The first main application of our axiomatic language is the following: 
\begin{thmI}\label{algebraic fundamental group}[\ref{etale part and geometric connectdness}, \ref{main thm for str and inf}, \ref{main thm for FDiv}]
Assume that $\sX$ is geometrically connected over $k$ and either $\Hl^0(\odi\stX)=k$ or there exists a field extension $L/k$ separably generated up to a finite extension (see \ref{separably generated}) such that $\stX(L)\neq \emptyset$.
\begin{enumerate}
\item If $\sX$ admits an fpqc covering $U\to \sX$ from a Noetherian scheme $U$ defined over the perfection $k^{\perf}$ of $k$ then $\Strat(\stX)$ satisfies axioms A,B and C and it is a $k$-Tannakian category.
\item If $\sX$ is an algebraic stack locally of finite type over $k$, then $\Crys(\sX)$ satisfies axioms A,B and C and it is a $k$-Tannakian category.
\item ($\car k>0$) If $\sX$ admits an fpqc covering $U\to \sX$ from a Noetherian scheme $U$ whose residue fields are separable up to a finite extension over $k$ (see \ref{separably generated}) then $\Fdiv(\stX)$ satisfies axioms A,B,C and D and it is a pro-smooth banded (see \ref{pro stuff for gerbes 2}) $k$-Tannakian category. 
\end{enumerate}
\end{thmI}
In any of the above situations, taking the gerbe associated with the corresponding Tannakian category, one has a notion of \textit{algebraic fundamental gerbe} for $\stX/k$. Notice moreover that all conditions are satisfied in Theorem \ref{algebraic fundamental group} if $\stX$ is a geometrically connected algebraic stack of finite type over $k$. In this last situation, in an unpublished result B. Bhatt proved that the three categories $\Crys(\stX)$, $\Strat(\stX)$ and, in positive characteristic, $\Fdiv(\stX)$ are all equivalent. This means that the three candidates for algebraic fundamental gerbe coincide for algebraic stacks of finite type over $k$.

The fact that $\Fdiv(\stX)$ is pro-smooth banded has already been observed by dos Santos in \cite[Theorem 11]{dS}, under the assumption that $k$ is algebraically closed and $\stX$ is a connected, locally Noetherian and regular scheme (see \ref{prosmoothness when k is perfect}).

Once we have a notion of an algebraic fundamental gerbe we must compare it with the relative analogous of the Grothendieck's \'etale fundamental group, namely the Nori \'etale fundamental gerbe $\Pi^{\NN,\et}_{\stX/k}$ of $\stX/k$ (see \ref{definition of Nori gerbes}) which exists if and only if $\stX$ is geometrically connected over $k$ (see \ref{existence of Nori etale gerbe}).  If $\shT(\stX)$ satisfies axioms A,B and C then $\stX$ is geometrically connected over $L_0$ and $\Pi^{\NN,\et}_{\stX/L_0}$ is the pro-\'etale quotient of $\Pi_{\shT(\stX)}$. If moreover $\shT(\stX)$ satisfies axiom D, one can use the profinite quotient instead (see \ref{when T(X) is Tannakian}).
In the hypothesis of Theorem \ref{algebraic fundamental group} we have that $\Pi^{\NN,\et}_{\stX/k}$ is the pro-\'etale quotient of $\Pi_{\Strat(\stX)}$ and $\Pi_{\Crys(\stX)}$ in situations $(1)$ and $(2)$ respectively, it is the profinite quotient of $\Pi_{\Fdiv(\stX)}$ in situation $(3)$.

The fibered category $\stX$ admits a Nori fundamental gerbe $\Pi^\NN_{\stX/k}$ over $k$ if and only if it is inflexible over $k$ (see  \cite[Definition 5.3 and Theorem 5.7]{BV}) and in this case $\Pi^{\NN,\et}_{\stX/k}$ is the pro-\'etale quotient of $\Pi^\NN_{\stX/k}$. We give a new concrete geometric interpretation of the notion of inflexibility: If $\stX$ is reduced (see \ref{connected and reduced fibered categories}) then $\stX$ is inflexible if and only if $k$ is integrally closed in $\Hl^0(\odi\stX)$ (see \ref{inflexibility characterization}).

Assume $\stX$ reduced from now on. In characteristic $0$ Nori fundamental gerbe and Nori \'etale fundamental gerbe coincide, so let's assume $\car k=p>0$. The same procedure used by Esnault and Hogadi in \cite{EH} allows us to construct a category $\shT_\infty(\stX)$ starting from the functor $\shT(\stX)\arr\Vect(\stX)$ and the pullback of Frobenius on those categories (see \ref{sequence of EH}). In particular are defined categories $\Crys_\infty(\stX)$, $\Strat_\infty(\stX)$ and $\Fdiv_\infty(\stX)$.
If $\shT(\stX)$ satisfies axioms A and B then $\shT_\infty(\stX)$ is an $L_\infty$-Tannakian category, where $L_\infty$ is the purely inseparable closure of $L_0$ inside $\Hl^0(\odi\stX)$ and thus, $\stX$ is also a category fibred over $L_{\infty}$. If $\shT(\stX)$ also satisfies axiom C, then $\stX$ is inflexible over $L_\infty$ and we have a diagram
  \[
  \begin{tikzpicture}[xscale=3.0,yscale=-1.4]
    \node (A0_0) at (0, 0) {$\Pi_{\shT_\infty(\stX)}$};
    \node (A0_1) at (1, 0) {$\Pi^\NN_{\stX/L_\infty}$};
    \node (A1_0) at (0, 1) {$\Pi_{\shT(\stX)}$};
    \node (A1_1) at (1, 1) {$\Pi^{\NN,\et}_{\stX/L_0}$};
    \path (A0_0) edge [->]node [auto] {$\scriptstyle{c}$} (A0_1);
    \path (A1_0) edge [->]node [auto] {$\scriptstyle{a}$} (A1_1);
    \path (A0_1) edge [->]node [auto] {$\scriptstyle{b}$} (A1_1);
    \path (A0_0) edge [->]node [auto] {$\scriptstyle{d}$} (A1_0);
  \end{tikzpicture}
  \]
where $c$ is a profinite quotient of $L_\infty$-gerbes (see \ref{main theorem formal II}). In particular $\Rep(\Pi^\NN_{\stX/L_\infty})\simeq \EF(\shT_\infty(\stX))$, where $\EF(-)$ denote the full subcategory of essentially finite objects (see \cite[Def 7.7]{BV}). Via Theorem \ref{algebraic fundamental group} we obtain the following Tannakian interpretation of the Nori fundamental gerbe, which extends the Tannakian interpretation in \cite[Theorem 7.9]{BV} to non pseudo-proper fibered categories.

\begin{thmII} In the hypothesis of Theorem \ref{algebraic fundamental group} assume moreover $\sX$ reduced and inflexible.  In situation (1) (resp. (2), (3)) of \ref{algebraic fundamental group} we have a canonical equivalence of $k$-Tannakian categories:
\[
\Rep_k(\Pi_{\sX/k}^\NN) \simeq \EF(\Strat_\infty(\stX))\text{ (resp. } \EF(\Crys_\infty(\stX)) \comma \EF(\Fdiv_{\infty}(\sX)))
\]
\end{thmII}

If $\stX$ is inflexible over $k$ and we apply the axiomatic theory to $\stX\arr \stX_\shT=\Pi^{\NN,\et}_{\stX/k}$ we obtain $\Rep \Pi^\NN_{\stX/k}\simeq \shT_\infty(\stX)$. In particular $\Rep(\Pi^\NN_{\stX/k})$ can be reconstructed from the map $\Rep(\Pi^{\NN,\et}_{\stX/k})\arr \Vect(\stX)$ and the Frobenius pullback of those categories (see \ref{full nori from etale nori}).

Finally we study the infinitesimal part of $\Pi^\NN_{\stX/k}$, that is its pro-local quotient $\Pi^{\NN,\LL}_{\stX/k}$ (see \ref{pro stuff for gerbes 2}), and give a concrete description of its representations in terms of vector bundles on $\stX$: applying the axiomatic theory to $\stX\arr \stX_\shT=\Spec k$  we have $\Rep \Pi^{\NN,\LL}_{\stX/k}\simeq \shT_\infty(\stX)$ (see \ref{main theorem local gerbes}).

One of the main ingredient in the proofs of our results regarding the Nori gerbes is the use of a generalized version of Tannaka's duality that can be applied,  not only to gerbes, but also to finite stacks. This version of Tannakian duality is discussed \S \ref{general Tannaka duality} in a great generality.

In \cite{TZ2}, which is based on the results of this paper, we gave an alternative and more geometric description of essentially finite $F$-divided sheaves. 

We outline the content of this paper. In the first section we describe a generalization of classical Tannaka's duality, while in the second and third section we collect some useful results that will be used through all the paper. In section four we introduce different notions of Nori fundamental gerbes and discuss their existence. Section five contains the formalism and general results of the paper, while in section six we determine appropriate conditions under which $\Strat(\stX)$, $\Crys(\stX)$ and $\Fdiv(\stX)$ satisfy the axiom of section five. In the last section we study the pro-local Nori fundamental gerbe. In the two appendices we study limit of categories and general results about affine gerbes respectively.

\section*{Notations and Conventions}\label{notation}
Given a ring $R$ we denote by $\Aff/R$ the category of affine $R$-schemes or, equivalently, the opposite of the category of $R$-algebras.

If $\stZ$ and $\stY$ are categories over a given category $\shC$, by a map $\stZ\arr \stY$ we always mean a base preserving functor. Similarly given maps $F,G\colon \stZ\arr \stY$ a natural transformation $\gamma\colon F\arr G$ will always be a base preserving natural transformation, that is for all $z\in \stZ$ over an object $c\in \shC$, the map $\gamma_z \colon F(z)\arr G(z)$ lies over $\id_c$. If $\stY$ is a fibered category we will denote by $\Hom^c_\shC(\stZ,\stY)$ the category of base preserving functors $\stZ\arr\stY$ which send all arrows to Cartesian arrows and the maps are the base preserving natural transformations.
If $\stY$ is a category fibered in groupoids then $\Hom^c_\shC(\stZ,\stY)$ is the category of all base preserving functors and we will simply denote it by $\Hom_\shC(\stZ,\stY)$.
If $\shC=\Aff/R$, where $R$ is a base ring, we will simply write $\Hom^c_R$ or $\Hom^c$ if the base ring is clear from the context. If $\stZ$ is a category over $\Aff/R$ the categories
\[
\Vect(\stZ)\subseteq \QCoh_{\textup{fp}}(\stZ)\subseteq \QCoh(\stZ)
\]
are defined as $\Hom^c_R(\stZ,\Vect)\subseteq \Hom^c_R(\stZ,\QCoh_{\textup{fp}})\subseteq \Hom^c_R(\stZ,\QCoh)$, where $\Vect\subseteq \QCoh_{\textup{fp}}\subseteq \QCoh$ are the
fiber categories (not in groupoids) over $\Aff/R$ of locally free sheaves of finite rank, quasi-coherent sheaves of finite presentation and quasi-coherent sheaves respectively. 
The categories $\Vect(\stZ)$, $\QCoh_{\textup{fp}}(\stZ)$ and $\QCoh(\stZ)$ are $R$-linear and monoidal categories. 
We say that a sequence of maps $\shF\arr \shF\arr \shF''$ in $\QCoh(\stZ)$ is \emph{pointwise exact} if for all $\xi\in\stZ$ over $\Spec A$ the sequence of $A$-modules $\shF'(\xi)\arr\shF(\xi)\arr\shF''(\xi)$ is exact.
Notice that in $\QCoh_{\textup{fp}}(\stZ)$ and $\QCoh(\stZ)$ all maps have a cokernel (defined pointwise). If $\stZ=\Spec B$ is affine we will simply write $\Vect(B)$, $\QCoh_{\textup{fp}}(B)$ and $\QCoh(B)$. The writing $\xi\in\stZ(A)$ means that $\xi$ is an object of $\stZ$ over $\Spec A$, and if $\shF\in\QCoh(\stZ)$, we will denote by $\shF_\xi\in\QCoh(A)$ the evaluation of $\shF$ in $\xi$.

If $f\colon \stY\arr\stZ$ is a base preserving map of categories over $\Aff/R$ then we have functors
\[
f^*\colon \Vect(\stZ)\arr \Vect(\stY)\comma f^*\colon \QCoh_{\textup{fp}}(\stZ)\arr \QCoh_{\textup{fp}}(\stY)\comma f^*\colon \QCoh(\stZ)\arr \QCoh(\stY)
\]
obtained simply by composing with $f$ and they are $R$-linear and monoidal.

An fpqc covering $\stX\arr \stY$ between categories fibered in groupoids is a functor representable by fpqc covering of algebraic spaces. A fibered category is called quasi-compact if it is fibered in groupoids and it admits an fpqc covering from an affine scheme. Let $\stX$ and $\stY$ be categories fibered in groupoids. A map $f\colon \stX\arr \stY$ is quasi-compact if $\stX\times_\stY A$ is quasi-compact for all maps $\Spec A\arr \stY$, it is quasi-separated if its diagonal is quasi-compact. The category $\stX$ is called quasi-separated is $\stX\arr \Spec \Z$ is quasi-separated, which implies that all maps $\stX\arr \stY$ are quasi-separated if $\stY$ has affine diagonal. If $f\colon \stX\arr\stY$ is quasi-compact and quasi-separated and $\stX$ and $\stY$ admit an fpqc covering from a scheme (resp. affine map between categories fibered in groupoids) then $f^*\colon \QCoh(\stY)\arr \QCoh(\stX)$ has a right adjoint $f_*\colon \QCoh(\stY)\arr \QCoh(\stX)$ which is compatible with flat base changes of $\stY$ (resp. any base change of $\stY$) (see \cite[Prop 1.5 and Prop 1.7]{Ton}).

Given a category fibered in groupoids $\stX$ over $\Aff/\F_p$ we define the absolute Frobenius $F_\stX$ of $\stX$ as
 \[
 F_\stX \colon \stX\arr \stX \comma \stX(A) \ni \xi\longmapsto F_A^*\xi\in \stX(A)
 \]
 where $F_A\colon \Spec A\arr \Spec A$ is the absolute Frobenius of $A$. The Frobenius is $\F_p$-linear, natural in $\stX$ and coincides with the usual Frobenius when $\stX$ is a scheme.
 If $\stX$ is defined over a field $k$ of characteristic $p$ we define $\stX^{(i,k)}=\stX\times_k k$, where $k\arr k$ is the $i$-th power of the absolute Frobenius of $k$, and we regard it as category over $k$ using the second projection. For simplicity when $k$ is clear from the context we will use just $-^{(i)}$ dropping the $k$. Notice that $(\stX^{(i)})^{(j)}$ is canonically equivalent to $\stX^{(i+j)}$. The $i$-th relative Frobenius of $\stX$ is the $k$-linear map $\stX\arr \stX^{(i)}$ that, composed with the projection $\stX^{(i)}\arr \stX$, is the Frobenius $F_\stX^i$. Notice that applying $-^{(j)}$ to the $i$-th Frobenius of $\stX$ one obtains the $i$-th Frobenius of $\stX^{(j)}$ and the composition of $1$-th Frobenius 
 \[
\stX\arr \stX^{(1)}\arr \cdots \arr \stX^{(i)}
 \]
 is the $i$-th Frobenius of $\stX$. When $X=\Spec A$ we will also set $A^{(i)}=A\otimes_k k$, where $k\arr k$ is the $i$-th power of the absolute Frobenius of $k$, so that $X^{(i)}=\Spec A^{(i)}$. 
 
 All monoidal categories and functors considered will be symmetric unless specified otherwise.

 \section*{Acknowledgement}
 We would like to thank B. Bhatt, H. Esnault, M. Olsson, M. Romagny and A. Vistoli for helpful conversations and suggestions received. We would also like to thank the referee for pointing out a mistake in an earlier version of this paper.
 
\section{Tannaka's reconstruction and recognition}\label{general Tannaka duality}

\begin{defn}\label{recognition and reconstruction}
A \emph{pseudo-abelian} category is an additive category $\shC$ endowed with a collection $J_\shC$ of sequences of the form $c'\arr c \arr c''$, where all objects and maps are in $\shC$. A linear functor $\Phi\colon\shC\arr \shD$ of pseudo-abelian categories is called \emph{exact} if it maps a sequence of $J_\shC$ to a sequence isomorphic to one of $J_\shD$.

Let $R$ be a ring. If $\stX$ is a category over $\Aff/R$ then $\Vect(\stX)$ will be considered as pseudo-abelian with the collection of maps $\shF'\arr\shF\arr \shF''$ such that
\[
0 \arr \shF'\arr\shF\arr \shF''\arr 0
\]
is pointwise exact.
If $\shC$ is abelian it is also pseudo-abelian if endowed with its short exact sequences.
If $\shC$ and $\shD$ are $R$-linear, monoidal and pseudo-abelian categories we denote by $\Hom_{\otimes,R}(\shC,\shD)$ the category whose objects are $R$-linear, exact and monoidal functors and whose arrows are natural monoidal isomorphisms. Notice that if $f\colon \stY\arr\stZ$ is any base preserving map of categories over $\Aff/R$ then $f^*\in \Hom_{\otimes,R}(\Vect(\stZ),\Vect(\stY))$.

Let $\shC$ be a pseudo-abelian monoidal $R$-linear category. The expression
\[
\Pi_\shC(A/R)=\Hom_{\otimes,R}(\shC,\Vect(A))
\]
defines a stack in groupoids for the fpqc topology over $R$. There is a functor
\[
\shC \arr \Vect(\Pi_\shC)\comma c \longmapsto ( \Pi_\shC(A)\ni \xi \longmapsto \xi(c)\in \Vect(A))
\]
which is $R$-linear, monoidal and exact. This induces a natural functor
\[
\Hom_R(\stZ,\Pi_\shC)\arr \Hom_{\otimes,R}(\shC,\Vect(\stZ))
\]
for all categories $\stZ$ over $\Aff/R$, which is easily seen to be an equivalence.

We say that $\shC$ satisfies \emph{Tannakian recognition} if the functor $\Phi\colon \shC\arr \Vect(\Pi_\shC)$ is an equivalence and for all sequences $\chi \colon c'\arr c\arr c''$ we have $\Phi(\chi)$ is exact if and only if $\chi\in J_\shC$ (equivalently $\Phi$ has an $R$-linear, monoidal and exact quasi-inverse).

If $\stY$ is a category over $\Aff/R$ there is a base preserving functor $\stY\arr \Pi_{\Vect(\stY)}$, namely
\[
\eta\in\stY(A)\longmapsto (\Vect(\stY)\ni \Phi \longmapsto \Phi(\eta) \in \Vect(A))
\]
We say that a category fibered in groupoids $\stY$ satisfies \emph{Tannakian reconstruction} if the functor $\stY\arr \Pi_{\Vect(\stY)}$ is an equivalence, or, equivalently, the pullback
\[
\Hom^c_R(\stZ,\stY)\arr \Hom_{\otimes,R}(\Vect(\stY),\Vect(\stZ))\comma f\longmapsto f^*
\]
is an equivalence for all categories $\stZ$ over $\Aff/R$ (just apply $\Hom^c_R(\stZ,-)$ to the map $\stY\arr\Pi_{\Vect(\stY)}$).
\end{defn}

\begin{rmk}
 If $\shC$ satisfies Tannakian recognition then $\Pi_\shC$ satisfies Tannakian reconstruction and if $\stY$ satisfies Tannakian reconstruction then $\Vect(\stY)$ satisfies Tannakian recognition.
 Notice also that those conditions do not depend on the base ring $R$. Indeed $\Vect(-)$ is insensible to the base ring and if $\shC$ is a pseudo-abelian monoidal $R$-linear category then
\[
\Pi_\shC\arr \Aff/R\arr \Aff/\Z
\]
coincides with $\Pi_\shC$ where $\shC$ is thought as a $\Z$-linear category.
\end{rmk}

\begin{defn}\label{sheaves generating categories}
Let $\stZ$ be a category fibered in groupoids and $\shD\subseteq \QCoh(\stZ)$ be a full subcategory. We say that $\shD$ generates $\QCoh(\stZ)$ if any object of $\QCoh(\stZ)$ is a quotient of an arbitrary direct sum of objects of $\shD$.
 We say that $\stZ$ has the resolution property if $\Vect(\stZ)$ generates $\QCoh(\stZ)$.
\end{defn}

\begin{thm}\cite[Cor 5.4]{Ton}
 If $\stZ$ is a quasi-compact  stack for the fpqc topology over a ring $R$ with quasi-affine diagonal and the resolution property then it satisfies Tannakian reconstruction.
\end{thm}

\begin{ex}
 let $k$ be a field. Classical Tannaka's duality implies that: if $\shC$ is a $k$-Tannakian category then it satisfies Tannakian recognition and $\Pi_\shC$ is an affine gerbe (gerbes with affine diagonal) over $k$. Conversely if $\Pi$ is an affine gerbe over $k$ then it satisfies Tannakian reconstruction and $\Vect(\Pi)$ is a $k$-Tannakian category. More precisely $\Pi$ has the resolution property (see \cite[Cor 3.9, pp. 132]{De3}).
\end{ex}

\begin{lem}\label{generation for affine and finite maps}
 Let $f\colon\stX\arr \stY$ be a map of categories fibered in groupoids over $R$. If $\shD\subseteq \QCoh(\stX)$ generates $\QCoh(\stX)$ and $f$ is finite, faithfully flat and finitely presented then $f_*\shD=\{f_*\E\st \E\in \shD\}$ generates $\QCoh(\stY)$. If $\overline\shD\subseteq \QCoh(\stY)$ generates $\QCoh(\stY)$ and $f$ is affine then $f^*\overline \shD=\{f^*\shH\st \shH\in \overline\shD\}$ generates $\QCoh(\stX)$.
\end{lem}
\begin{proof}
 In the second case, if $\shF\in\QCoh(\stX)$ then there is a surjective map $\bigoplus_j \shH_j\arr f_*\shF$ with $\shH_j\in \overline \shD$ and therefore a surjective map $\bigoplus_j f^*\shH_j\arr f^*f_*\shF$. Since $f$ is affine the map  $f^*f_*\shF\arr \shF$ is surjective.
 
 Let's consider the first statement. Let $\shG\in \QCoh(\stY)$ and set $\shG_\stX=\shG\otimes_{\odi \stY} \Homsh_\stY(f_*\odi \stX,\odi \stY)$. The map $\odi\stY\arr f_*\odi\stX$, which is locally split injective, induces a surjective map $$\Homsh_\stY(f_*\odi\stX,\odi\stY)\arr \odi\stY$$ and therefore a surjective map $\shG_\stX\arr \shG$. The sheaf $\shG_\stX$ is an $f_*\odi \stX$-module, so there exists $\shG'\in \QCoh(\stX)$ such that $f_*\shG'\simeq \shG_\stX$. Thus, taking a surjection $\bigoplus_j \E_j\arr \shG'$ with $\E_j'\in \shD$ and using that $f_*$ is exact we get the result.
\end{proof}

\begin{cor}\label{finite stacks are reconstructible}
 Let $\Gamma$ be a finite stack over a field $k$ (see \ref{finite and etale stacks}). Then there exists $\E\in \Vect(\Gamma)$ which generates $\QCoh(\Gamma)$. In particular $\Gamma$ has the resolution property and it satisfies Tannakian reconstruction.
\end{cor}
\begin{proof}
Apply \ref{generation for affine and finite maps} to a finite atlas $f\colon U\arr \Gamma$ with $U$ finite $k$-scheme and $\shD=\{\odi U\}$.
\end{proof}

\section{\'Etale part and geometric connectedness}

Through this section we consider given a field $k$.

\begin{defn}
 Given a $k$-algebra $A$ we set
 \[
 A_{\et,k}=\{a\in A\st \exists \text{ a separable polynomial } f\in k[x] \text{ s.t. } f(a)=0\}
 \]
 Alternatively $A_{\et,k}$ is the union of all $k$-subalgebras of $A$ which are finite and \'etale over $k$. When the base field is clear from the context we will simply write $A_{\et}$.
\end{defn}

\begin{rmk}\label{rmk about relative frobenius}
Let $A$ be a $k$-algebra of characteristic $p$. 
 The $i$-th relative Frobenius of $A$ is given by
 \[
 f_i\colon A^{(i)}=A\otimes_k k\arr A\comma a\otimes \lambda \longmapsto a^{p^{i}}\lambda
 \]
 If $x=\sum_{1\leq j\leq n}a_j\otimes b_j\in A^{(i)}$ with $a_i\in A$ and $b_j\in k$, then $x^{p^i}=\sum_{1\leq j\leq n}a_j^{p^i}\otimes b_j^{p^i}=\sum_{1\leq j\leq n}a_j^{p^i}b_j\otimes 1=f_i(x)\otimes 1$ for all $x\in A^{(i)}$. In particular $\Ker f_i=\{x\in A^{(i)}\st x^{p^i}=0\}$.
 
 Moreover the map $(A^{(1)})_{\et}\arr A_{\et}$ is an isomorphism. Indeed denote by $B$ the image of $A^{(1)}\arr A$. Since $A^{(1)}\arr B$ is surjective with nilpotent kernel the map  $(A^{(1)})_{\et}\arr B_{\et}$ is an isomorphism. Since $B$ contains all $p$-powers of $A$, we see that $A_{\et}=B_{\et}$.
\end{rmk}

\begin{lem}\label{frob}
Let $A$ be a finite $k$-algebra of characteristic $p$. There exists $n\in\N$ such that the image of the relative Frobenius $A^{(n)}\arr A$ is an \'etale $k$-algebra. In particular the residue fields of $A^{(n)}$ are separable over $k$.
\end{lem}
\begin{proof}
We can assume that $A$ is local with residue field $L$. Consider $n\in \N$ such that $p^n\geq \dim_k A=\dim_k A^{(n)}$. In particular the $p^n$-power of the maximal ideal of $A^{(n)}$ is zero. Taking into account \ref{rmk about relative frobenius} we see that the image of $A^{(n)}\arr A$ is the residue field of $A^{(n)}$, which also coincides with the residue field of $L^{(n)}$. If $K$ is the maximal separable extension of $k$ inside $L$ we have that $x^{p^n}\in K$ for all $x\in L$. By \ref{rmk about relative frobenius} we see that the image of $L^{(n)}\arr L$ is contained in $K$ and thus is separable over $k$.
\end{proof}

\begin{defn}\label{connected and reduced fibered categories}
 If $\stY$ and $\stZ$ are categories fibered in groupoids we define $\stY\sqcup \stZ$ as the category fibered in groupoids whose objects over an affine scheme $U$ are tuples $(U',U'',\xi,\eta)$ where $U'$, $U''$ are open subsets of $U$ such that $U=U'\sqcup U''$, $\xi\in\stY(U')$ and $\eta\in \stZ(U'')$.
 
 We say that a category fibered in groupoids $\stX$ is connected if $\Hl^0(\odi \stX)$ has no non trivial idempotents. We say it is reduced if any map $U\arr \stX$ from a scheme factors through a reduced scheme fpqc locally in $U$. We say that a morphism of categories fibered in groupoids $f\colon \stX\arr \stY$ is geometrically connected (resp. reduced) if for all geometric points $\Spec L\arr \stY$ the fiber $\stX\times_{\stY} L$ is connected (resp. reduced).
\end{defn}

\begin{rmk}
 If $\stX$ is a stack in groupoids for the Zariski topology and $\stY,\stZ$ are open substacks then one can always define a map $\stY\sqcup \stZ\arr \stX$. In this situation $\stX$ is connected if and only if it cannot be written as a disjoint union of non-empty open substacks.
 
 If $\stX$ is a reduced category fibered in groupoids then $\Hl^0(\odi\stX)$ is a reduced ring. Indeed if $\lambda\in\Hl^0(\odi\stX)$ then one can define the vanishing substack $\stY\arr\stX$ of $\lambda$, so that $\stY\arr\stX$ is a closed immersion which is also nilpotent if $\lambda$ is so. Let's prove that $\stY=\stX$, that is that if $U\arr\stX$ is a map from a scheme then $U\times_\stX\stY\arr U$ is an isomorphism. By fpqc descent and the definition of reduceness we reduce the problem to the case when $U$ is reduced, where the result is clear.
 
 If $\stX$ is an algebraic stack then the notion of reduceness just defined and the classical one coincides.
\end{rmk}

\begin{lem}\label{etale part and geometric connectdness}
 Let $\stX$ be a quasi-compact and quasi-separated fibered category over $k$. Then
 \begin{enumerate}
  \item for all field extensions $L/k$ we have
  \[
  \Hl^0(\odi \stX)_{\et,k}\otimes_k L \simeq \Hl^0(\odi{\stX\times_k L})_{\et,L}
  \]
  \item the map $\stX\arr \Spec \Hl^0(\odi \stX)_{\et,k}$ is geometrically connected;
  \item the fiber category $\stX$ is geometrically connected over $k$ if and only if $\Hl^0(\odi \stX)_{\et,k}=k$.
 \end{enumerate}
\end{lem}
\begin{proof}
 It is clear that $2)\then 3)$. Write $A=\Hl^0(\odi \stX)$ and notice that if $k\subseteq B\subseteq A_{\et}$ and $C$ is any $B$-algebra then
 \[
 \Hl^0(\odi{\stX\times_B C})\simeq A\otimes_B C
 \]
 This follows from the fact that $\stX\arr \Spec B$ is quasi-compact and quasi-separated, so that the notion of push-forward of quasi-coherent sheaves is well defined, and the fact that $B$ is a Von Neumann regular ring, that is all $B$-modules are flat or, equivalently, all finitely generated ideals are generated by an idempotent: indeed $B$ is a filtered direct limits of its $k$-\'etale and finite subalgebras, which are easily seen to be Von Neumann regular rings. This shows that we can assume $\stX=\Spec A$ and work only with algebras.
 
 Let's prove $1)$. We have an inclusion $A_{\et,k}\otimes_k L\subseteq (A\otimes_k L)_{\et,L}$.  Given an element $u\in (A\otimes_k L)$ separable over $L$ we must show that $u\in A_{\et,k}\otimes_k L$. Since $u$ can be written with finitely many elements of $A$ and $L$ and the same is true for the separable equation it satisfies, we can assume that $A/k$ is of finite type and $L/k$ is finitely generated.  Moreover the result holds for the extension $L/k$ if it holds for all subsequent sub-extensions in a finite filtration $k=k_0 \subseteq k_1\subseteq \dots \subseteq k_l=L$, or if it holds for $L'/k$, where $L\subseteq L'$, because of the inclusion
 \[
 (A_{\et,k}\otimes_k L) \otimes_L L' \subseteq (A\otimes_kL)_{\et,L}\otimes_L L'\subseteq (A\otimes_kL')_{\et,L'}
 \]
 In conclusion the problem can be split in the following cases: $L/k$ is finite and Galois; $L=k$ and $k\arr L$ is the Frobenius; $k$ and $L$ are algebraically closed: first assume $L$ algebraically closed, then assume $L/k$ algebraic using the splitting $k\subseteq \overline k \subseteq L$, then assume $L/k$ finite and, finally, split in separable and  purely inseparable extensions which are subextensions of a Galois extension and of a sequence of Frobenius extension respectively.
 
 Assume first that $L/k$ is finite and Galois with group $G=\Gal(L/k)$. The subalgebra $(A\otimes_k L)_{\et,L}$ of $A\otimes_k L$ is invariant by the action of $G$ and therefore, by Galois descent, we have
 \[
 (A\otimes_k L)_{\et,L}\simeq (A\otimes_k L)_{\et,L}^G \otimes_k L
 \]
 Since $(A\otimes_k L)_{\et,L}$ is etale over $L$ and therefore over $k$ and $(A\otimes_k L)_{\et,L}^G=(A\otimes_k L)_{\et,L}\cap A$ we obtain the result.
 
 Assume now that $L=k$ and $k\arr L$ is the Frobenius. We have a commutative diagram of $k$-linear maps
  \[
  \begin{tikzpicture}[xscale=3.2,yscale=-1.2]
    \node (A0_0) at (0, 0) {$A_{\et,k}\otimes_k L$};
    \node (A0_1) at (1, 0) {$(A\otimes_k L)_{\et,L}$};
    \node (A0_2) at (2, 0) {$A\otimes_k L$};
    \node (A1_0) at (0, 1) {$A_{\et,k}$};
    \node (A1_2) at (2, 1) {$A$};
    \path (A0_1) edge [->]node [auto] {$\scriptstyle{\mu}$} (A1_2);
    \path (A0_0) edge [->]node [auto] {$\scriptstyle{\delta}$} (A0_1);
    \path (A1_0) edge [->]node [auto] {$\scriptstyle{}$} (A1_2);
    \path (A0_2) edge [->]node [auto] {$\scriptstyle{\gamma_A}$} (A1_2);
    \path (A0_0) edge [->]node [auto] {$\scriptstyle{\gamma_{A_{\et,k}}}$} (A1_0);
    \path (A0_1) edge [->]node [auto] {$\scriptstyle{}$} (A0_2);
  \end{tikzpicture}
  \]
  where the $\gamma_*$ are the relative Frobenius and $A$ and $A_{\et,k}$ has to be thought of as $L$-algebras in the bottom row.
We must show that $\delta$ is an isomorphism. By \ref{rmk about relative frobenius} $\mu$ and $\gamma_{A_{\et,k}}$ are injective because $(A\otimes_k L)_{\et,L}$ is reduced. It also follows that $\gamma_{A_{\et,k}}$ is an isomorphism because it is an $L$-linear injective map between two $L$-vector spaces of the same dimension. Since $A_{\et,k}$ is the maximum $L$-\'etale subalgebra of $A$ it follows that $\delta$ is an isomorphism.

Assume now that $k$ and $L$ are algebraically closed. Notice that in this situation $A$ is connected if and only if $A_{\et,k}=k$. Decomposing $A$ into connected components we can assume that $A$ is connected. Since $k$ and $L$ are algebraically closed it follows that also $A\otimes_k L$ is connected and therefore that $(A\otimes_k L)_{\et,L}=L$.

Let's now prove $2)$. Let $\alpha\colon A_{\et,k}\arr L$ be a geometric point. We must prove that $(A\otimes_{A_{\et,k}}L)_{\et,L}=L$. Let $J$ be the kernel of $\alpha$ and $F$ be its image, which is easily seen to be a field. Thanks to $1)$ it is sufficient to prove that $(A\otimes_{A_{\et,k}}F)_{\et,F}=(A/JA)_{\et,F}$ is just $F$. Let $a\in A$ be such that its quotient lies in $(A/JA)_{\et,F}$. Lifting also a separable equation satisfied by $a$ mod $J$ to $A_{\et,k}$, we can again assume that $A$ is of finite type over $k$ and, moreover, that it is connected. In this case $A_{\et,k}$ is just a field, thus equal to $F$ and the result is obvious.
\end{proof}

\section{Some results on finite stacks}

Let $k$ be a field. In this section we collect some results about finite stacks that will be used later. For many other properties look at \cite[Section 4]{BV}.

\begin{defn}\label{finite and etale stacks}
 A finite (resp. finite \'etale) stack $\Gamma$ over a field $k$ is a stack in the fppf topology on $\Aff/k$ which has a finite (resp. finite \'etale) and faithfully flat morphism $U\arr \Gamma$ from a finite (resp. finite \'etale) $k$-scheme $U$. Equivalently $\Gamma$ is the quotient of a flat groupoid of finite (resp. finite \'etale) $k$-schemes.
\end{defn}


Here is a non-trivial application of the Tannaka's duality discussed in \S \ref{general Tannaka duality} which generalize \cite[Prop 4.3]{BV}.
\begin{prop}\label{finite reduced stacks are gerbes}
 If $\Gamma$ is a finite and reduced stack over $k$ then $\Gamma\arr \Spec \Hl^0(\odi\Gamma)$ is a gerbe.
\end{prop}
\begin{proof}
 We can assume $\Gamma$ connected, so that $L=\Hl^0(\odi\Gamma)$ is a field. Set $\shC=\Vect(\Gamma)$. Since $\Gamma$ is Tannakian reconstructible by \ref{finite stacks are reconstructible}, the functor $\Gamma\arr \Pi_\shC$, which is an $L$-map, is an equivalence. By \cite[Lemma 7.15]{BV} we have $\shC=\QCoh_{\textup{fp}}(\Gamma)$, which easily implies that $\shC$ is an $L$-Tannakian category and therefore $\Pi_\shC$ is a gerbe over $L$.  
\end{proof}

 \begin{lem}\label{approximating finite stacks}
 If $L/k$ is an algebraic extension of fields and $\Gamma$ is a finite stack over $L$ then there exists a finite subextension $F/k$, a finite stack $\Delta$ over $F$ with an isomorphism $\Gamma\simeq \Delta\times_F L$.
\end{lem}
\begin{proof}
 The stack $\Gamma$ is the quotient of a groupoid $s,t\colon R\rightrightarrows U$, where $R$, $U$ are spectra of finite $L$-algebras and $s,t$ are faithfully flat.  Since everything is of finite presentation, we can descend the groupoid $R\rightrightarrows U$ to a finite sub-extension $F/k$, thus also $\Gamma$.
\end{proof}

\begin{lem}\label{constructing the etale part of a finite stack}
 Let $R\rightrightarrows U$ be a flat groupoid with $R$ and $U$ finite over $k$. Then $(R\times_{s,t,U} R)_{\et}=R_{\et}\times_{s,t,U_{\et}} R_{\et}$, the maps defining the groupoid $R\rightrightarrows U$ yields a structure of groupoid on $R_{\et}\rightrightarrows U_{\et}$ with a map from $R\rightrightarrows U$. Moreover if the residue fields of $R$ and $U$ are separable over $k$, the same holds for $(-)_{\red}$ in place of $(-)_{\et}$ and the resulting groupoids are the same, where $(-)_{\red}$ is the functor which takes, for any scheme $X$, its reduced closed subscheme structure.
\end{lem}
 \begin{proof}
  Using \ref{frob} and \ref{rmk about relative frobenius}, we can Frobenius twist the original groupoid until $R$ and $U$ has separable residue fields, that is their reduced structures are \'etale. In this case $R_{\red}\arr R \arr R_{\et}$ is an isomorphism and similarly for $U$. The result follows by expressing a groupoid in terms of commutative and Cartesian diagrams and using the following fact: if $V,W,Z$ are finite $k$-schemes whose reduced structures are \'etale and $V,W\arr Z$ are maps then
  \[
  (V\times_Z W)_\red = V_\red\times_{Z_\red} W_\red = V_\et \times_{Z_\et} W_\et = (V\times_Z W)_\et
  \]
  The above equalities follows because a product of \'etale schemes is \'etale and thus reduced.
 \end{proof}

 \begin{defn} \label{etale part}
  Let $\Gamma$ be a finite stack over $k$ and let $U\arr \Gamma$ be a finite atlas where $U$ is affine. We define $\Gamma_{\et,k}$ as the quotient of the groupoid constructed in \ref{constructing the etale part of a finite stack} with respect to the groupoid $R=U\times_{\Gamma}U\rightrightarrows U$. When $k$ is clear from the context we will drop the $-_k$. By \ref{independent of the chart} below this notion does not depend on the choice of the finite atlas.
 \end{defn}

\begin{lem} \label{independent of the chart}
Let $\Gamma/k$ be a finite stack and $E/k$ be a finite and \'etale stack. Then the functor $\Hom_k(\Gamma_{\et},E)\arr \Hom_k(\Gamma,E)$ is an equivalence. Moreover for all $j\in \N$ the map $\Gamma_{\et}\arr (\Gamma^{(j)})_{\et}$ is an equivalence and for $j\gg 0$ the functor $\Gamma^{(j)}\arr (\Gamma^{(j)})_{\et}$ has a section. In particular for $j\gg 0$ the relative Frobenius $\Gamma\arr \Gamma^{(j)}$ factors through $\Gamma_{\et}$.
\end{lem}
%

\begin{proof}
The second part follows from \ref{frob} and \ref{rmk about relative frobenius}. For the first part is enough to show that, if $U$ is a finite $k$-scheme, then $E(U_{\et})\arr E(U)$ is an equivalence. Since $U\arr U_{\et}$ is finite, flat and geometrically connected by \ref{etale part and geometric connectdness}, it follows that the diagonal $U\arr R=U\times_{U_{\et}} U$ is a nilpotent closed immersion, so that $E(R)\arr E(U)$ and the two maps $E(U)\rightrightarrows E(R)$ induced by the projections $R\rightrightarrows U$ are equivalences. Computing $E(U_{\et})$ on the flat groupoid $R\rightrightarrows U$ we get the result. \end{proof}

\begin{rmk}\label{base change of etale for finite gerbes}
 If $\Gamma$ is a finite stack over $k$ and $L/k$ is a field extension then, by \ref{etale part and geometric connectdness} and the definition of $\Gamma_{\et,k}$, we have $\Gamma_{\et,k}\times_k L\simeq (\Gamma\times_k L)_{\et,L}$.
\end{rmk}

\begin{rmk}\label{etale part and purely inseparable extensions}
 If $\Gamma$ is a finite stack over $F$ and $F/k$ is a finite and purely inseparable field extension then the natural morphism $\Gamma_{\et,F}\arr \Gamma_{\et,k}\times_k F\cong(\Gamma\times_kF)_{\et,F}$ is an equivalence. Indeed $\Gamma\to \Gamma\times_kF$ is the base change of the diagonal of $\Spec(F)$ by $\Gamma_{\et,F}$, thus it is a nilpotent thickening, so it induces an equivalence on the \'etale quotients. 
\end{rmk}

\begin{defn}\label{finite and local stacks}
 A finite stack $\Gamma$ over $k$ is called \emph{local} if $\Gamma_{\et,k}=\Spec k$.
\end{defn}

\begin{rmk}\label{closed of local is local}
 A closed substack of a finite and local stack is always local. Indeed if $\Delta$ is a closed substack of a finite and local stack $\Gamma$ then, since $\Gamma$ is connected and thus topologically a point, the map $\Delta\arr \Gamma$ is a nilpotent closed immersion: using the definition of the \'etale part from a presentation follows that $\Delta_{\et}=\Gamma_{\et}$. 
\end{rmk}

%


\section{Nori fundamental gerbes}

\begin{defn}\label{definition of Nori gerbes}\cite[Section 5]{BV}
 If $\stZ$ is a category over $\Aff/k$ the Nori fundamental gerbe (resp. \'etale Nori fundamental gerbe, local Nori fundamental gerbe) of $\stZ/k$ is a profinite (resp. pro-\'etale, pro-local) gerbe $\Pi$ over $k$ together with a map $\stZ\arr \Pi$ such that for all finite (resp. finite and \'etale, finite and local) stacks $\Gamma$ over $k$ the pullback functor
 \[
 \Hom_k(\Pi,\Gamma)\arr \Hom_k(\stZ,\Gamma)
 \]
 is an equivalence. If this gerbe exists it is unique up to a unique isomorphism and it will be denoted by $\Pi^\NN_{\stZ/k}$ (resp. $\Pi^{\NN,\et}_{\stZ/k}$, $\Pi^{\NN,\LL}_{\stZ/k}$) or by dropping the $/k$ if it is clear from the context.
\end{defn}

\begin{rmk}\label{nori gerbes for gerbes}
 If $\stZ$ is a category fibered in groupoids over $k$ a Nori gerbe exists over $k$ if and only if $\stZ$ is inflexible over $k$, that is all maps from $\stZ$ to a finite stack over $k$ factors through an affine gerbe over $k$ (see \cite[Definition 5.3 and Theorem 5.7]{BV}). This is the case if $\stZ$ is an affine gerbe over $k$. Moreover if $\stZ$ is inflexible also the \'etale and local Nori gerbe exist, $\Pi^{\NN,\et}_\stZ = (\Pi^\NN_\stZ)_{\et}$ and $\Pi^{\NN,\LL}_\stZ = (\Pi^\NN_\stZ)_{\LL}$ (see \ref{pro stuff for gerbes 2} and  \ref{closed of local is local}).
\end{rmk}

The following result, although not stated elsewhere, is known by experts.
\begin{prop}\label{existence of Nori etale gerbe}
 Let $\stZ$ be a quasi-compact and quasi-separated fibered category. Then $\stZ$ admits a Nori \'etale fundamental gerbe if and only if $\stZ$ is geometrically connected over $k$.
\end{prop}
\begin{proof}
 Assume a Nori \'etale fundamental gerbe exists. If $k\subseteq A\subseteq \Hl^0(\odi \stZ)$ with $A/k$ \'etale, then by definition $\stZ\arr \Spec A$ factors through $\Pi^{\NN,\et}_\stZ$. Since $\Hl^0(\odi{\Pi^{\NN,\et}_\stZ})=k$, the factorization tells us that $A\arr \Hl^0(\odi\stZ)$ factors through $k$. Thus $\Hl^0(\odi\stZ)_{\et}=k$ and $\stZ$ is geometrically connected by \ref{etale part and geometric connectdness}.
 
 Assume now $\stZ$ geometrically connected. The proof of the existence of $\Pi^{\NN,\et}_\stZ$ follows the same proof given in \cite[Proof of Theorem 5.7]{BV}. In our case $I$ is the $2$-category of Nori reduced (see \cite[Definition 5.10]{BV}) maps $\stZ\arr \Delta$ where $\Delta$ is an \'etale gerbe. The only thing that must be checked is that if $\stZ\arrdi f \Gamma$ is a map to a finite and \'etale stack then there exists a factorization $\stZ\arrdi \alpha \Delta \arr \Gamma$ where $\Delta$ is an \'etale gerbe and $\alpha$ is Nori reduced. Consider $\Delta'=\Spec (\odi\Gamma/\shI)$ where $\shI=\Ker(\odi\Gamma\arr f_*\odi\stZ)$. The stack $\Delta'$ is finite, \'etale and $\Hl^0(\odi{\Delta'})$ is \'etale over $k$ and contained in $\Hl^0(\odi\stZ)$, thus equal to $k$. So $\Delta'/k$ is a gerbe thanks to \ref{finite reduced stacks are gerbes}. The map $\stZ\arr \Delta'$ factors through a Nori reduced map $\stZ\arrdi \alpha\Delta$, where $\Delta$ is a finite gerbe, and $\Delta\arr \Delta'$ is faithful. It follows that $\Delta$ is \'etale because faithfulness means that the map on stabilizers is injective. 
\end{proof}

The following result generalize \cite[Proposition 5.5]{BV}
\begin{thm}\label{inflexibility characterization}
 Let $\stZ$ be a reduced, quasi-compact and quasi-separated fibered category. Then $\stZ$ is inflexible if and only if $k$ is integrally closed inside $\Hl^0(\odi\stZ)$.
\end{thm}
\begin{proof}
 The only if part is \cite[Prop 5.4, a)]{BV}. For the if part consider a map $f\colon \stZ\arr \Gamma$ where $\Gamma$ is a finite stack. If $\shI=\Ker(\odi\Gamma\arr f_*\odi\stZ)$ then $f$ factors through $\Spec(\odi\Gamma/\shI)$, so that we can assume $\odi\Gamma\arr f_*\odi\stZ$ injective. So $\Gamma$ is reduced, finite and $\Hl^0(\odi\Gamma)$ is a subalgebra of $\Hl^0(\odi\stZ)$ finite over $k$, thus equal to $k$ by our assumption. By \ref{finite reduced stacks are gerbes} it follows that $\Gamma$ is a finite gerbe.
\end{proof}

\section{Formalism for algebraic and Nori fundamental gerbes}

Let $k$ be a field and consider two categories $\stX$ and $\stX_\shT$ over $\Aff/k$ together with a base preserving functor $\pi_{\sT}:\stX\arr \stX_\shT$.

\begin{defn}\label{the T construction}
Set $\shT_k(\stX)=\Vect(\stX_\shT)$, which is a pseudo-abelian, rigid, monoidal and $k$-linear category. Moreover the functor $\pi_\shT^*\colon \shT_k(\stX)\arr \Vect(\stX)$ is $k$-linear, monoidal and exact. More generally if $\stY$ is a fibered category over $\Aff/k$ we have a natural functor
\[
\Hom_k^c(\stX_\shT,\stY)\arr \Hom^c_k(\stX,\stY)
\]
By \ref{recognition and reconstruction} $\Pi_{\shT_k(\stX)}$ comes equipped with a $k$-map $\stX_\shT \arr \Pi_{\shT_k(\stX)}$ inducing $\id \colon \shT_k(\stX)\arr \Vect(\stX_\shT)$.
We will drop the $-_k$ when $k$ is clear from the context.
\end{defn}

We consider categories over $k$ instead of just fibered categories over $k$ in order to apply this theory also to categories $\stX$ like small sites of algebraic stacks.

We now introduce a list of axioms that will ensure nice Tannakian properties of $\shT(\stX)$. In what follows by a finite (\'etale) stack over a ring $R$ we mean a stack which is an fppf quotient of an fppf groupoid of finite (\'etale), faithfully flat and finitely presented $R$-schemes.

\begin{ax}\label{axioms}
 Set $L=\Hl^0(\odi {\stX_\shT})=\End_{\shT(X)}(1_{\shT(\stX)})$ and consider:
 \begin{enumerate}
  \item[A:] $\shT(\stX)=\QCoh_{\textup{fp}}(\stX_\shT)$;
  \item[B:] the functor $\shT(\sX)\arr \Vect(\stX)$ is faithful;
  \item[C:] for all finite and \'etale stacks $\Gamma$ over $L$ the following functor is an equivalence
 \[
 \Hom_L(\stX_\shT,\Gamma)\arr \Hom_L(\stX,\Gamma)
 \]
  \item[D:] all $L$-maps from $\stX_\shT$ to a finite gerbe over $L$ factors through a finite and \'etale gerbe over $L$.
 \end{enumerate}
\end{ax}

\begin{rmk}
 If $\car k = 0$ and $L=\Hl^0(\odi {\stX_\shT})$ is a field then axiom D is automatic, because all finite gerbes are also \'etale.
\end{rmk}

\begin{lem}\label{TX abelian}
 Assume axiom $A$. Then $\shT(\sX)$ is a $k$-linear, abelian, monoidal and rigid category and the exact sequences are pointwise exact.
\end{lem}
\begin{proof}
We already know that $\shT(\stX)$ is $k$-linear, rigid and monoidal. In the category $\QCoh_{\textup{fp}}(\stX_\shT)$ cokernel can be taken pointwise. The result then follows because if $\alpha\colon\shF\arr\shG$ is a map of locally free sheaves over $\Spec(R)$ whose cokernel is locally free, then $\Ker(\alpha)$ is locally free and the formation of the kernel commutes with arbitrary base change.
\end{proof}

\begin{rmk}
 If $\shC$ is a $k$-linear and monoidal category and $R=\End_\shC(1_\shC)$ then $\shC$ has a natural structure of $R$-linear category: if $\lambda\in R$ and $\phi\colon x\arr y$ is a morphism in $\shC$ we define
 \[
 \lambda\phi\colon x\simeq x\otimes 1_\shC \arrdi{\phi\otimes \lambda}y\otimes 1_\shC\simeq y
 \]
\end{rmk}

\begin{lem}\label{when C is Tannakian}
 Let $\shC$ be a $k$-linear, rigid, abelian and monoidal category and let $F\colon \shC\arr \Vect(\stZ)$, where $\stZ$ is a non-empty category over $\Aff/k$, be a $k$-linear, exact and monoidal functor. If $\stZ$ is connected and $F$ is faithful then $\End_{\shC}(1_\shC)$ is a field. If $L=\End_{\shC}(1_\shC)$ is a field then $\shC$, with its natural $L$-linear structure, is an $L$-Tannakian category and $F$ is faithful. In particular $\shC$ is Tannakian recognizable and $\Pi_\shC$ is an affine gerbe over $L$.
\end{lem}
\begin{proof}
Assume $\stZ$ connected, $F$ faithful and set $R=\End_{\shC}(1_\shC)$. Let's show that it is a field proving that if $\alpha\in R$ is non zero then it is invertible in $R$. Since $\shC$ is abelian consider the exact sequence
 \[
 0\arr\shK\arr 1_\shC\arrdi \alpha 1_\shC \arr\shQ\arr 0
 \]
 Since $F$ is exact, $F(\alpha)$ is an element of $\End(\odi \stZ)=\Hl^0(\odi\stZ)$ whose kernel and cokernel are locally free. Thus for all $\xi\in\stZ$, we get a natural decomposition $\Spec ((\odi \stZ)_\xi)=U_\xi\sqcup V_\xi$ where $U_\xi$, $V_\xi$ are the opens where $F(\alpha)_\xi$ is invertible and $0$ respectively. This determines an idempotent $e\in \Hl^0(\odi\stZ)$. Since this ring is connected by hypothesis then $e=0$ or $e=1$, that is one of the following situations occur: $F(\alpha)=0$ so that $\alpha=0$ since $F$ is faithful; $F(\alpha)$ is an isomorphism, so that $F(\shK)=F(\shQ)=0$ and, again by faithfulness of $F$, $\shK=\shQ=0$, which implies that $\alpha$ is an isomorphism. Thus $R=L$ is a field.
 
 Assume now  $L=\End_{\shC}(1_\shC)$ is a field.  Since $\stZ$ is non empty there exists $\xi\in\stZ(A)$ for some $k$-algebra $A$ and the functor
  \[
 G:\shC\arrdi F\Vect(\stZ)\arrdi{-_\xi}\Vect(A)
 \]
  is $k$-linear, exact and monoidal. The $L$-linear structure on $\shC$ induces an $L$-algebra structure on $A$ such that $G$ is $L$-linear. From \cite[1.9, pp. 114]{De3} it follows that $\shC$ is an $L$-Tannakian category and $G$ is faithful. In particular also $F$ is faithful.
\end{proof}

As a consequence of \ref{TX abelian} and \ref{when C is Tannakian} we obtain:
\begin{prop}\label{when TX is Tannakian}
 Assume axiom A and set $L=\Hl^0(\odi{\stX_\shT})$. If $\stX$ is connected and axiom $B$ holds then $L$ is a field. If $L$ is a field then axiom B holds, so that $\shT(\stX)$ is an $L$-Tannakian category, $\Pi_{\shT(\stX)}$ is an affine gerbe over $L$ and $\stX\arr \stX_\shT \arr \Pi_{\shT(\stX)}$ can be considered as $L$-maps.
\end{prop}

\begin{thm}\label{when T(X) is Tannakian}
 Assume axiom A and that $L=\Hl^0(\odi{\stX_\shT})$ is a field (for instance if B holds and $\stX$ is connected), so that $\shT(\stX)$ is an $L$-Tannakian category by \ref{when TX is Tannakian}, where $L=\Hl^0(\odi{\stX_\shT})$. If $R\arr L$ is a map of rings and $\Gamma$ is any stack in groupoids over $R$ satisfying Tannakian reconstruction, then the functor
\[
 \Hom_R(\Pi_{\shT(\stX)},\Gamma)\arr \Hom_R(\stX_\shT,\Gamma)
\]
is an equivalence. In particular $\stX_\shT\arr \Pi_{\shT(\stX)}$ is universal among $L$-maps from $\stX_\shT$ to an affine gerbe over $L$.

If axiom C also holds then $\stX\arr (\Pi_{\shT(\stX)})_{\et}$ is the \'etale Nori fundamental gerbe of $\stX$ over $L$, so that $\Rep (\Pi^{\NN,\et}_{\stX/L}) \simeq \Et(\shT(\stX))$ (see \ref{pro stuff for gerbes 2}).

If both axioms C and D also holds, then $\widehat \Pi_{\shT(\stX)}=(\Pi_{\shT(\stX)})_{\et}$, so that $\Rep (\Pi^{\NN,\et}_{\stX/L}) \simeq \EF(\shT(\stX))$ (see \ref{pro stuff for gerbes 2}).
\end{thm}
\begin{proof}
Since $\shT(\stX)$ and $\Gamma$ are Tannakian recognizable and reconstructible respectively, we have equivalences
\[
\Hom_R(\Pi_{\shT(\stX)},\Gamma)\simeq \Hom_{\otimes,R}(\Vect(\Gamma),\shT(\stX))\simeq \Hom_R(\stX_\shT,\Gamma)
\]
The above map is easily seen to coincide with the map induced by $\stX_{\shT}\arr \Pi_{\shT(\stX)}$. Since affine gerbes satisfies Tannakian reconstruction we get the universality of $\stX_\shT\arr \Pi_{\shT(\stX)}$.

Assume now C. Since finite stacks are Tannakian reconstructible by \ref{finite stacks are reconstructible}, for all finite and \'etale stacks $\Gamma$ the maps $\stX\arr\stX_\shT\arr  \Pi_{\shT(\stX)}\arr (\Pi_{\shT(\stX)})_{\et}$ induces equivalences
\[
\Hom_L((\Pi_{\shT(\stX)})_{\et},\Gamma)\simeq \Hom_L(\Pi_{\shT(\stX)},\Gamma)\simeq \Hom_L(\stX_\shT,\Gamma)\simeq \Hom_L(\stX,\Gamma) 
\]
as desired, where the first equivalence follows because $(\Pi_{\shT(\stX)})_{\et}$ is the Nori \'etale quotient of $\Pi_{\shT(\stX)}$ (see \ref{pro stuff for gerbes 2}). Finally axiom D tells exactly that a morphism from $\Pi_{\shT(\stX)}$ to a finite stack factors through a finite and \'etale gerbe, which implies the result.
\end{proof}

\begin{rmk}
 Under the hypothesis of \ref{when T(X) is Tannakian} the $L$-gerbe $\Pi_{\shT(\stX)}$ together with the functor $\stX_\shT\arr \Pi_{\shT(\stX)}$ is the $\catC$-fundamental gerbe of $\stX_\shT$ over $L$, where $\catC$ is the class of all affine group schemes, in the sense of \cite[Definition 5.6]{BV2}. In particular $\stX_\shT\arr \widehat \Pi_{\shT(\stX)}=\Pi_{\EF(\shT(\stX))}$ is the Nori fundamental gerbe of $\stX_\shT$ over $L$ and $\stX_\shT$ is inflexible if it is a fibered category.  However, such a fundamental gerbe does not exist in general as is explained in \cite[Theorem 5.7, pp. 13]{BV}. Moreover, $\stX_\shT$ is in general not a nice fibered category  in the cases which we consider in this paper (stratifications, crystals or $F$-divided sheaves): it is unclear if $\stX_\shT$ admits an fpqc covering from a scheme. So we prefer not to apply the general theory of fundamental gerbes on $\stX_\shT$ but just use it as a parameter space. 
\end{rmk}

From now on we assume that $k$ has positive characteristic $p$. If $\stZ$ is any category over $\Aff/k$ we define the Frobenius pullback
\[
F^*\colon \Vect(\stZ)\arr \Vect(\stZ)
\]
applying the pullback of the absolute Frobenius pointwise. The functor $F^*$ is $\F_p$-linear, exact and monoidal.


\begin{defn}\label{sequence of EH}
 Given $i\in \N$ we define $\shT_i(\stX)$ as the category of tuples $(\shF,\shG,\lambda)$ where $\shF\in \Vect(\stX)$, $\shG\in \shT(\sX)$ and $\lambda \colon F^{i*}\shF\arr \shG_{|\stX}$ is an isomorphism. A morphism from $(\shF,\shG,\lambda)$ to $(\shF',\shG',\lambda')$ is a pair of morphisms $\phi:\sF\to\sF'$ and $\varphi:\sG\to\sG'$ which are compatible with $\lambda$ and $\lambda'$ in an obvious way. The category $\shT_i(\stX)$ is $\F_p$-linear, monoidal and rigid. We endow $\shT_i(\stX)$ with a $k$-structure via
 \[
 k\arr \End_{\shT_i(\stX)}(\odi \stX,\odi{\stX_\shT},\id_{\odi \stX})\comma  a \longmapsto (a,a^{p^i})
 \]
 Finally we regard $\shT_i(\stX)$ as a pseudo-abelian category with the distinguished set of sequences which are exact pointwise. The forgetful functor $\shT_i(\stX)\arr \Vect(\stX)$ is $k$-linear, monoidal and exact.
 
 There is a $k$-linear, monoidal and exact functor
 \[
 \shT_i(\stX)\arr \shT_{i+1}(\stX)\comma (\shF,\shG,\lambda)\longmapsto (\shF,F^*\shG,F^*\lambda)
 \]
 We define $\shT_\infty(\stX)$ as the direct limit of the categories $\shT_i(\stX)$. The category $\shT_\infty(\stX)$ is $k$-linear, monoidal and rigid.
\end{defn}


\begin{rmk}\label{universal property at level i}
Given a category fibered in groupoids $\Gamma$ over $\F_p$ we denote by $\Hom(\stX,\stX_\shT,i,\Gamma)$ the category of $\F_p$-linear $2$-commutative diagrams
  \[
  \begin{tikzpicture}[xscale=2.1,yscale=-1.2]
    \node (A0_0) at (0, 0) {$\stX$};
    \node (A0_1) at (1, 0) {$\stX_\shT$};
    \node (A1_0) at (0, 1) {$\Gamma$};
    \node (A1_1) at (1, 1) {$\Gamma$};
    \path (A0_0) edge [->]node [auto] {$\scriptstyle{\pi}$} (A0_1);
    \path (A0_0) edge [->]node [auto] {$\scriptstyle{f}$} (A1_0);
    \path (A0_1) edge [->]node [auto] {$\scriptstyle{g}$} (A1_1);
    \path (A1_0) edge [->]node [auto] {$\scriptstyle{F_\Gamma}$} (A1_1);
  \end{tikzpicture}
  \]
where $F_\Gamma$ is the absolute Frobenius. Pulling back along $f$ and $g$ one obtains a functor $\Phi_i^\Gamma \colon \Hom(\stX,\stX_\shT,i,\Gamma)\arr \Hom_{\otimes,\F_p}(\Vect(\Gamma),\shT_i(\stX))$ which is an equivalence if $\Gamma$ is Tannakian reconstructible. On the other hand using the universal property of $\Pi_*$ in \ref{recognition and reconstruction} and the definition of $\shT_i(\stX)$ 
 we see that $\Pi_{\shT_i(\stX)}$ comes equipped with a $2$-commutative diagram $\chi\in \Hom(\stX,\stX_\shT,i,\Pi_{\shT_i(\stX)})$ such that $\shT_i(\stX)\arr \Vect(\Pi_{\shT_i(\stX)})\arrdi{J_i} \shT_i(\stX)$ is the identity, where $J_i=\Phi_i^{\Pi_{\shT_i(\stX)}}(\chi)$. Composing with $\chi$ we obtain a functor
 \[
 \Hom_{\F_p}(\Pi_{\shT_i(\stX)},\Gamma)\arr \Hom(\stX,\stX_\shT,i,\Gamma)
 \]
 which is an equivalence if $\Gamma$ and $\shT_i(\stX)$ are Tannakian reconstructible and recognizable respectively.
 \end{rmk}

\begin{lem}\label{Frobenius for TiX}
There is a 2-commutative diagram
  \[
  \begin{tikzpicture}[xscale=2.7,yscale=-1.2]
    \node (A0_0) at (0, 0) {$\Vect(\Pi_{\shT_i(\stX)})$};
    \node (A0_1) at (1, 0) {$\shT_i(\stX)$};
    \node (A1_0) at (0, 1) {$\Vect(\Pi_{\shT_i(\stX)})$};
    \node (A1_1) at (1, 1) {$\shT_i(\stX)$};
    \path (A0_0) edge [->]node [auto] {$\scriptstyle{J_i}$} (A0_1);
    \path (A1_0) edge [->]node [auto] {$\scriptstyle{J_i}$} (A1_1);
    \path (A0_1) edge [->]node [auto] {$\scriptstyle{F_\shT}$} (A1_1);
    \path (A0_0) edge [->]node [auto] {$\scriptstyle{F^*}$} (A1_0);
  \end{tikzpicture}
  \]
where $F_{\shT}(\sF,\sG,\lambda)=(F_{\sX}^*\sF,F_{\sX_{\sT}}^*\sG, F_{\sX}^*\lambda)$, $F$ is the absolute Frobenius of $\Pi_{\sT_i(\sX)}$ and $J_i$ is defined in \ref{universal property at level i}.
Moreover $\sF_{\sT}^i$ factors as $\sT_i(\sX)\xrightarrow{\alpha}\sT_0(\sX)\xrightarrow{\beta}\sT_i(\sX)$, where $\alpha$ is the projection $(\sF,\sG,\lambda)\mapsto \sG$ and $\beta$ is the transition morphism defined in \ref{sequence of EH}.
\end{lem}
\begin{proof}
 The commutativity of the first diagram follows from the naturality of Frobenius pullbacks and the definition of $J_i$. The second claim follows from the formula $$(F_{\sX}^{i^*}\sF,F_{\sX}^{i^*}\sG,F_{\sX}^{i^*}\lambda)\xrightarrow{(\lambda,\id)}(\sG|_{\sX},F_{\sX}^{i^*}\sG,\id)$$
\end{proof}

\begin{thm}\label{main theorem formal II}
Assume axiom A, that $L=L_0=\Hl^0(\odi{\stX_\shT})$ is a field and the following property:
 \[
 \forall \shF\in \QCoh_{\textup{fp}}(\stX)\comma \text{if } F^*\shF\in \Vect(\stX)\text{ then } \shF\in \Vect(\stX)
 \]
 Then for all $j\in \N\cup\{\infty\}$ the ring $L_j=\End_{\shT_j(\stX)}(1_{\shT_j(\stX)})$ is a field, $\shT_j(\stX)$ is an $L_j$-Tannakian category, the functors
 \[
\begin{tikzpicture}[xscale=3.0,yscale=-1.2]
  \node (A0_0) at (0, 0) {$\shT_j(\stX)$};
  \node (A0_1) at (1, 0) {$\shT_{j+1}(\stX)$};
  \node (A0_2) at (2, 0) {$\shT_\infty(\stX)$};
  \node (A1_1) at (1, 1) {$\Vect(\stX)$};
  \path (A0_0) edge [->]node [auto] {$\scriptstyle{}$} (A0_1);
  \path (A0_1) edge [->]node [auto] {$\scriptstyle{}$} (A1_1);
  \path (A0_1) edge [->]node [auto] {$\scriptstyle{}$} (A0_2);
  \path (A0_2) edge [->]node [auto] {$\scriptstyle{}$} (A1_1);
  \path (A0_0) edge [->]node [auto] {$\scriptstyle{}$} (A1_1);
\end{tikzpicture}
\]
are faithful, monoidal and exact, $\Pi_{\shT_j(\stX)}$ is an affine gerbe over $L_j$ and the functor $\shT_j(\stX)\arr \Vect(\stX)$ induces a map $\stX\arr \Pi_{\shT_j(\stX)}$, so that $\stX$ is a category over $L_{\infty}$. Moreover
\[
L_\infty=\{x\in \Hl^0(\odi\stX)\st\exists i\in\N\text{ such that }x^{p^i}\in L_0\}
\]
is purely inseparable over $L_0$,
\[
\EF(\shT_i(\stX))=\{ (\shF,\shG,\lambda)\in \shT_i(\stX)\st \shG\in \EF(\shT_0(\stX))\}\comma \EF(\shT_\infty(\stX))\simeq \varinjlim_i \EF(\shT_i(\stX))
\]
and $\stX\arr (\Pi_{\shT_\infty(\stX)})_{\LL} $ is the pro-local Nori fundamental gerbe of $\stX$ over $L_\infty$.

 If we also assume axiom C then $\stX\arr \widehat \Pi_{\shT_\infty(\stX)}$  is the Nori fundamental gerbe of $\stX$ over $L_\infty$, so that $\Rep (\Pi^\NN_{\stX/L_\infty}) \simeq \EF(\shT_\infty(\stX))$, where all the notations here are in \ref{pro stuff for gerbes}.
\end{thm}
\begin{proof}
Notice that $\shT(\stX)=\shT_0(\stX)$ is $L_0$-Tannakian and $\shT(\stX)\arr \Vect(\stX)$ is faithful thanks to \ref{when TX is Tannakian}. 
 Let's show that the category $\shT_i(\stX)$ is abelian. If $(\shF,\shG,\lambda)\arrdi {(\alpha,\beta)} (\shF',\shG',\lambda')$ is a map in $\shT_i(\stX)$, then there is an induced isomorphism $\delta\colon F^{i*}(\Coker \alpha) \arr (\Coker \beta)_{|\stX}$, which implies that $\Coker \alpha\in \Vect(\stX)$ and that $(\Coker \alpha, \Coker \beta, \delta)\in \shT_i(\stX)$ is a cokernel. In this situation also kernels can be taken pointwise so that we obtain a kernel for $(\alpha,\beta)$.
 The map $\shT_i(\stX)\arr \Vect(\stX)$ is faithful because if $(\alpha,\beta)$ is a map as above with $\alpha=0$, then $\beta_{|\stX}=0$, which implies $\beta=0$ because $\shT(\stX)\arr \Vect(\stX)$ is faithful. This implies that all the functors in the statements are faithful and that $\shT_\infty(\stX)$ is an abelian category. In particular for $i\in \N$ the functor $\shT_i(\stX)\arr \Vect(\stX)$ induces an isomorphism
 \[
 L_i=\{ (x,y) \st x\in \Hl^0(\odi\stX)\comma y\in L_0\comma x^{p^i}=y\}\arr\{x\in \Hl^0(\sO_{\stX})\st x^{p^i}\in L_0\}
 \]
 Notice that $\Hl^0(\odi\stX)$ is reduced: if $u\in \Hl^0(\odi\stX)$ with $u^n=0$, then for $i$ large $F^{i^*}(\odi\stX/u\odi\stX)\simeq\odi\stX/u^{p^i}\odi\stX\simeq \odi\stX$, and this implies that $(\odi\stX/u\odi\stX)\in\Vect(\stX)$ which is possible only if $u=0$. In particular it follows that $L_i$ is a field. Moreover $L_\infty$ is the union of the $L_i$, which implies that it is a field and that the description in the statement holds. By \ref{when C is Tannakian} we conclude that the categories $\shT_i(\stX)$ and $\shT_\infty(\stX)$ are $L_i$-Tannakian and $L_\infty$-Tannakian respectively.
 
 Let's consider now the equality about essentially finite objects of $\shT_i(\stX)$ in the statement. The projection $\shT_i(\stX)\arrdi\alpha \shT_0(\stX)$ is $\Z$-linear, exact and monoidal. This gives the inclusion $\subseteq$. For the converse let $\chi=(\shF,\shG,\lambda)\in \shT_i(\stX)$ such that $\shG\in \EF(\shT_0(\stX))$ and denote by $\Gamma$ the monodromy gerbe of $\chi$, which is an $L_i$-gerbe of finite type such that $\Rep (\Gamma)=\langle \chi \rangle \subseteq \shT_i(\stX)$ (see \ref{pro stuff for gerbes}). We have to show that $\Gamma$ is finite. Using \ref{Frobenius for TiX} and its notation we have a $2$-commutative diagram
   \[
  \begin{tikzpicture}[xscale=2.5,yscale=-1.2]
    \node (A0_0) at (0, 0) {$\Rep (\Gamma)$};
    \node (A0_1) at (1, 0) {$\shT_i(\stX)$};
    \node (A1_0) at (0, 1) {$\Rep (\Gamma)$};
    \node (A1_1) at (1, 1) {$\shT_i(\stX)$};
    \path (A0_0) edge [->]node [auto] {$\scriptstyle{}$} (A0_1);
    \path (A0_0) edge [->]node [auto] {$\scriptstyle{F^{i*}}$} (A1_0);
    \path (A0_1) edge [->]node [auto] {$\scriptstyle{F_{\sT}^i}$} (A1_1);
    \path (A1_0) edge [->]node [auto] {$\scriptstyle{}$} (A1_1);
  \end{tikzpicture}
  \]
and, moreover, $F_\shT^i(\chi)$ is essentially finite. Since $\Rep (\Gamma)$ is a sub Tannakian category of $\shT_i(\stX)$ it follows that $F^{i*}\chi$ is essentially finite in $\Rep (\Gamma)$ and, since $\Rep (\Gamma)=\langle\chi\rangle$, it follows that the $i$-th absolute and therefore relative Frobenius of $\Gamma$ factors through a finite gerbe. Such a factorization continues to hold if we base change to $\overline{L_i}$, so that $\Gamma\times_{L_i}\overline{L_i}\simeq \Bi G$, where $G$ is an affine group of finite type over $\overline{L_i}$ whose relative Frobenius $G\arr G^{(i)}$ factors through a finite group scheme. Since the relative Frobenius is topologically surjective, we conclude that $G$ is a finite group scheme as desired.

Let's now prove the isomorphism between $\EF(\shT_\infty(\stX))$ and the limit in the statement. Let $\chi\in \shT_\infty(\stX)$ and $\chi_i\in \shT_i(\stX)$ mapping to $\chi$ for some $i$. If $\chi$ is finite then clearly $\chi_i$ will be finite up to replace $i$. If $\chi$ is instead a kernel of a map between finite objects, then those objects and this map will be image of a map $u$ of finite objects in some $\shT_j(\stX)$. The kernel of $u$ is then a essentially finite objects of $\shT_j(\stX)$ mapping to $\chi$.

 We now consider the claims about Nori gerbes. Let $\Phi$ be a finite stack over $L_\infty$ and consider the map 
 \[
 \Hom_{L_\infty}(\Pi_{\shT_\infty(\stX)},\Phi)\arr \Hom_{L_\infty}(\stX,\Phi)
 \]
 We have to prove that this is an equivalence if $\Phi$ is local and an equivalence in general when axiom C holds.  We can moreover assume that $L_0=k$. Using \ref{approximating finite stacks}, we can find a finite extension $F/k$, a finite stack $\Gamma$ over $F$ with an isomorphism $\Phi\simeq \Gamma\times_F L_{\infty}$. The above map then becomes
 \[
 \Psi_{\Gamma,F}\colon \Hom_F(\Pi_{\shT_\infty(\stX)},\Gamma)\arr \Hom_F(\stX,\Gamma)
 \]
 Notice that $F/k$ is a finite purely inseparable extension and thus $\Gamma/k$ is finite. Moreover if $\Phi$ is local then $\Gamma/k$ is also local thanks to \ref{base change of etale for finite gerbes} and \ref{etale part and purely inseparable extensions}. 

Thus if we know that $\Psi_{\Gamma,k}$ and $\Psi_{\Spec(F),k}$ are equivalences we can conclude that $\Psi_{\Gamma,F}$ is an equivalence. This shows that we can assume $F=k$. Set also $\Psi=\Psi_{\Gamma,k}$.

 We are going to use that finite stacks satisfies Tannakian reconstruction by \ref{finite stacks are reconstructible}. Moreover the map $\Hom_k(\stX_\shT,\Gamma_{\et})\arr \Hom_k(\stX,\Gamma_{\et})$ is an equivalence if $\Gamma$ is local (that is $\Gamma_{\et}=\Spec k$) or in general if axiom C holds. Thus we can assume it is an equivalence.  
 
 $\Psi$ \emph{essentially surjective}.
Let $\stX\arrdi a \Gamma$ be a $k$-map and consider the factorization $\Gamma\arr \Gamma_{\et}\arr \Gamma^{(j)}$ of \ref{independent of the chart}. We can extend the map $\stX\arr\Gamma_\et$ to $\stX_\shT$ obtaining a $2$-commutative diagram
  \[
  \begin{tikzpicture}[xscale=1.5,yscale=-1.2]
    \node (A0_0) at (0, 0) {$\stX$};
    \node (A0_1) at (1, 0) {$\stX_\shT$};
    \node (A1_0) at (0, 1) {$\Gamma$};
    \node (A1_1) at (1, 1) {$\Gamma_{\et}$};
    \node (A1_2) at (2, 1) {$\Gamma^{(j)}$};
    \path (A0_0) edge [->]node [auto] {$\scriptstyle{}$} (A0_1);
    \path (A0_0) edge [->]node [auto] {$\scriptstyle{a}$} (A1_0);
    \path (A0_1) edge [->]node [auto] {$\scriptstyle{}$} (A1_1);
    \path (A1_0) edge [->]node [auto] {$\scriptstyle{}$} (A1_1);
    \path (A1_1) edge [->]node [auto] {$\scriptstyle{}$} (A1_2);
  \end{tikzpicture}
  \]
and therefore, by \ref{universal property at level i}, a map $e\colon \Pi_{\shT_j(\stX)}\arr \Gamma$ inducing $a\colon \stX\arr \Gamma$. The map $e$ is automatically $k$-linear because $a$ is so and $\shT_j(\stX)\arr \Vect(\stX)$ is faithful.

$\Psi$ \emph{fully faithful}. We are going to show that a map $\Pi_{\shT_\infty(\stX)}\arr \Gamma$ factors through a map $\Pi_{\shT_i(\stX)}\arr \Gamma$. Before doing that we show how to conclude that $\Psi$ is fully faithful. Let $\alpha,\beta\colon \Pi_{\shT_\infty(\stX)}\arr \Gamma$ be two maps and $\delta\colon \alpha_{|\stX}\arr \beta_{|\stX}$ be an isomorphism of functors $\stX\arr \Gamma$. The uniqueness of an extension is easy, because $\alpha,\beta$ correspond to maps $\Vect(\Gamma)\arr \shT_{\infty}(\stX)$ and $\shT_{\infty}(\stX)\arr \Vect(\stX)$ is faithful. We can assume that both $\alpha,\beta$ factors through $\Pi_{\shT_i(\stX)}$, so that, by \ref{universal property at level i}, they correspond to $2$-commutative diagrams
  \[
  \begin{tikzpicture}[xscale=1.9,yscale=-1.2]
    \node (A0_0) at (0, 0) {$\stX$};
    \node (A0_1) at (1, 0) {$\stX_\shT$};
    \node (A1_0) at (0, 1) {$\Gamma$};
    \node (A1_1) at (1, 1) {$\Gamma^{(i)}$};
    \path (A0_0) edge [->]node [auto] {$\scriptstyle{}$} (A0_1);
    \path (A0_0) edge [->,bend left=15]node [auto,swap] {$\scriptstyle{\alpha_{|\stX}}$} (A1_0);
    \path (A0_0) edge [->,bend right=15]node [auto] {$\scriptstyle{\beta_{|\stX}}$} (A1_0);
    \path (A0_1) edge [->,bend right=15]node [auto] {$\scriptstyle{v}$} (A1_1);
    \path (A0_1) edge [->,bend left=15]node [auto,swap] {$\scriptstyle{u}$} (A1_1);
    \path (A1_0) edge [->]node [auto] {$\scriptstyle{}$} (A1_1);
  \end{tikzpicture}
  \]
where $u,v$ are $k$-linear. By \ref{independent of the chart} there exists $j>i$ and a factorization $\Gamma^{(i)}\arr \Gamma_{\et}^{(i)}=\Gamma_{\et} \arr \Gamma^{(j)}$. Replacing $i$ by $j$ we can assume that $u,v$ factor through $\Gamma_\et \arr \Gamma^{(i)}$.
Since $\Hom_k(\stX_\shT,\Gamma_{\et})\simeq \Hom_k(\stX,\Gamma_{\et})$ we can lift the isomorphism $\delta\colon \alpha_{|\stX}\arr \beta_{|\stX}$ to an isomorphism $u\arr v$ as required.

It remains to show that a $k$-linear, monoidal and exact map $F\colon \Vect(\Gamma)\arr \shT_\infty(\stX)$ factors through some $\shT_i(\stX)$. 
We will use a slight modification of \cite[Prop 3.8]{BV} and its proof. Pick $S=\Spec K\arr \Pi_{\shT_{\infty}(\sX)}$, where $K$ is a field, an object corresponding to $\xi\colon \shT_\infty(\stX)\arr \Vect(K)$ and set $R_j=S\times_{\Pi_{\shT_j(\stX)}}S$. Given a $K$-scheme $T$ an object of $R_\infty(T)$ is a triple $(u,v,\gamma)$ where $u,v\colon T\arr S$ and $\gamma\colon u^*\circ\xi \arr v^*\circ \xi$ is a monoidal isomorphism. Similarly, using the functor $\chi=\chi_{\Vect(T)}$ of \ref{hom from the limit}, an object of $(\varprojlim_j R_j)(T)$ is a triple $(u,v,\tilde \gamma)$ where $u,v\colon T\arr S$ and $\tilde\gamma\colon \chi(u^*\circ\xi) \arr \chi(v^*\circ \xi)$ is an isomorphism given by monoidal natural transformations. From this we can deduce that $R_\infty\simeq \varprojlim_j R_j$. Since a map from a scheme to a gerbe is an fpqc covering, the map $\Pi_{\shT_\infty(\stX)}\arr \Gamma$ is given by an object $z\in \Gamma(S)$ with an identification of the two projections in $\Gamma(R_{\infty})\simeq \varinjlim_i \Gamma(R_j)$ satisfying the cocycle condition. Here we use that $\Gamma$ is finitely presented. This identification lies in some $\Gamma(R_j)$. Up to replace this $j$, we can also assume that this identification satisfies the cocycle condition, which yields the desired factorization.\end{proof}

\begin{rmk}
 If $\stX$ is a reduced category fibered in groupoids over $k$ and $\shF\in\QCoh_{\textup{fp}}(\stX)$ then $F^*\shF\in \Vect(\stX)$ implies that $\shF\in\Vect(\stX)$. Indeed let $\phi\colon V\arr\stX$ be a map from a scheme. We must show that $\phi^*\shF$ is a vector bundle. Since $\stX$ is reduced, by fpqc descent we can assume that $\phi$ factors through a reduced scheme. This allow to assume that $\stX$ is a reduced scheme and also that $\stX=\Spec R$, where $R$ is a local ring, so that $F^*\shF$ is free of some rank $r$. Since the Frobenius is an homeomorphism, it follows that for all $p\in \stX$ we have $\dim_{k(p)}\shF\otimes k(p)=r$. Nakayama's lemma gives a surjective morphism $\phi\colon R^r\arr \shF$. If $v\in \Ker \phi$, since $\phi$ is an isomorphism on each minimal prime ideal of $R$, it follows that all entries of $v$ are nilpotent and thus $v=0$.
\end{rmk}

\begin{thm}\label{full nori from etale nori}
 Let $\stZ$ be a reduced and inflexible category fibered in groupoids over $k$ and denote by $\pi\colon \stZ\arr \Pi^{\NN,\et}_{\stZ/k}$ the structure morphism. Denote also by $\shC_i$ the monoidal and additive category of triples $(\E,V,\lambda)$ where $\E\in \Vect(\stZ)$, $V\in \Rep (\Pi^{\NN,\et}_{\stZ/k})$ and $\lambda\colon F^{i*}\E \arr \pi^*V$ is an isomorphism and regard $\shC_i$ as a $k$-linear category via $k\arr \End(\odi\stZ,\odi{\Pi^{\NN,\et}_{\stZ/k}},1)$, $x\mapsto (x,x^{p^i})$. By pulling back along the Frobenius of $\Pi^{\NN,\et}_{\stZ/k}$ we obtain $k$-linear monoidal functors $\shC_i\arr \shC_{i+1}$. Then the $\shC_i$ are $k$-Tannakian categories and there is an equivalence of $k$-Tannakian categories $\varinjlim_i \shC_i\simeq \Rep \Pi^\NN_{\stZ/k}$, where the structure morphism $\Rep\Pi^\NN_{\stZ/k}\arr\Vect(\stZ)$ corresponds to the forgetful functor.
\end{thm}
\begin{proof}
 Consider $\stZ=\stX\arr \stX_\shT=\Pi^{\NN,\et}_{\stZ/k}$. It is easy to see that this map satisfies axioms $A,B,C$ and $D$ and $L=\Hl^0(\odi{\Pi^{\NN,\et}_{\stZ/k}})=k$. By \cite[Prop 5.4, a)]{BV} it follows that $k$ is integrally closed in $\Hl^0(\odi\stZ)$ and the result then follows from \ref{main theorem formal II}: we have $k=L_0=L_\infty$, $\shC_i=\shT_i(\stZ)$ and $\EF(\shT_0(\stX))=\shT_0(\stX)$ implies $\EF(\shT_\infty(\stX))=\shT_\infty(\stX)$.
\end{proof}

\section{Stratification, crystal and Frobenius divided structures}

In this section we apply the result of the previous section to find  explicit morphisms $\pi_{\sT}:\stX\arr \stX_\shT$ for which the general theory works properly. In the next sections when  we  talk about axioms we will always refer to the list of axioms \ref{axioms}.

We start by introducing some geometric notions that will be used in the whole section.
\begin{defn} \label{separably generated}
 A field extension $L/k$ is called separably generated (resp. separable) up to a finite extension if there exists an intermediate extension $k\subseteq F\subseteq L$ such that $L/F$ is finite and $F/k$ is separable (resp. separably generated) (see \cite[\href{http://stacks.math.columbia.edu/tag/030I}{030I}]{SP}). 
\end{defn}

For instance any finitely generated field extension of $k$ is separably generated up to a finite extension.

\begin{rmk}\label{separable extension base change by purely inseparable is a field}
 If $L/k$ is a separable extension and $E/k$ is an algebraic and purely inseparable extension then $L\otimes_k E$ is a field. Indeed $\Spec(L\otimes_k E)\arr \Spec L$ is an homeomorphism and, by  \cite[\href{http://stacks.math.columbia.edu/tag/030W}{030W}]{SP}, is reduced.
\end{rmk}

\begin{defn}\label{defn: adically separated schemes}
 Let $X$ be a scheme. A point $p\in X$ is called adically separated if the local ring $(\odi{X,p},m_p)$ is $m_p$-adically separated, that is $\bigcap_n m_p^n=0$.
\end{defn}

\begin{rmk}
 We introduced this notion instead of considering just Noetherian rings because, when studying $F$-divided sheaves, we have to consider Frobenius twists $X^{(i)}$ of a scheme $X$, which may be not Noetherian even though $X$ is so. For example, if $k$ is a field whose absolute Frobenius is not finite and $L=k$ with $k$-structure given by the Frobenius $k\arr L$, then $L^{(1,k)}=L\otimes_k L$ is not Noetherian. The $p$-power of any element in the kernel of the multiplication map $\delta\colon L\otimes_k L \arr L$ is zero because
 \[
 (\sum_i a_i\otimes b_i)^p=\sum_i a_i^p\otimes b_i^p = 1\otimes (\sum_i a_i^pb_i^p) = 1\otimes \delta(\sum_i a_i\otimes b_i)^p \text{ for all } a_i,b_i\in L
 \]
In particular $L\otimes_k L$ is a local $L$-algebra with residue field $L$. If $L\otimes_k L$ was Noetherian, then the maximal ideal would be nilpotent and, because the residue field is a finite extension of $L$, $L\otimes_k L$ would be a finite $L$-algebra. Thus $L$ would be a finite extension of $k$, contrary to our assumption. 
\end{rmk}

Instead adically separatedness is maintained by Frobenius twists under some mild hypothesis:

\begin{lem}\label{Bon Ring}
Let $(R,m)$ be an $m$-adically separated local ring defined over a field $k$ of positive characteristic and whose residue field is separable up to a finite extension over $k$. Then, for all $i\in N$, $R^{(i)}$ is a local ring separated for the topology of its maximal ideal and its residue field is separable up to a finite extension over $k$.
\end{lem}
\begin{proof}
Notice that $R^{(1)}$ is a local ring because the Frobenius of $k$ is purely inseparable.
Denote by $F$ the residue field of $R$. By hypothesis there exists $k\subseteq E\subseteq F$ such that $E/k$ is separable and $F/E$ is finite. By \ref{separable extension base change by purely inseparable is a field} $E^{(1)}$ is a separable field extension of $k$. Since $F^{(1)}/E^{(1)}$ is finite, we see that also the residue field of $R^{(1)}$ is obtained as a separable extension followed by a finite one. In particular we can assume $i=1$. 

 Denote by $m_1$ the maximal ideal of $R^{(1)}$. We can assume $m$ nilpotent: since the image of $\cap_n m_1^n$ in $(R/m^lR)^{(1)}$ is zero for all $l$ we have
 \[
 \bigcap_n m_1^n \subseteq \bigcap_l (m^l\otimes_k k) = (\bigcap m^l) \otimes_k k = 0 
 \]
 By \cite[\href{http://stacks.math.columbia.edu/tag/0320}{0320}]{SP} the extension $E/k$ is formally smooth. Thus there is a lifting $E\subseteq R$. In particular
 \[
 R^{(1,k)}=R\otimes_k k \simeq R\otimes_E E^{(1,k)}
 \]
 Since $E^{(1,k)}$ is a field the relative Frobenius $E^{(1,k)}\arr E$ is injective and, applying $R\otimes_E -$, we see that also $R^{(1,k)}\arr R^{(1,E)}$ is injective. This allows us to reduce the problem to the case that $F/k$ is a finite extension. In particular $F^{(1)}$ is Notherian. Thus a power of $m_1$ lies in the kernel of $R^{(1)}\arr F^{(1)}$, which is $m\otimes_k k$. Since this last ideal is nilpotent, we get that $m_1$ is nilpotent too.
\end{proof}

\begin{lem}\label{m-adically sepated rings and free modules}
Let $(R,m)$ be an $m$-adically separated ring and $M$ be a finitely generated $R$-module. Then $M$ is free if and only if $M/m^nM$ is a free $R/m^n$-module for all $n\in \N$.
\end{lem}
\begin{proof}
We have to prove $\Longleftarrow$. Lifting a basis of $M/mM$ we can define a surjective morphism $\phi\colon R^l \arr M$, which will be an isomorphism after tensoring by $R/m^nR$ by hypothesis. So if $v\in \Ker \phi$, it becomes $0$ on all the quotients $(R/m^nR)^l$ and therefore $v\in (\cap_n m^n)^l=0$ as desired.
\end{proof}

\subsection{Stratifications and crystals}
\begin{defn}\label{definition of stratifications and crystals}
 Let $\pi\colon \stX\arr \Aff/k$ be a category over $\Aff/k$. We define the big infinitesimal site $\stX_{\iinf/k}$ of $\stX$ as the category of pairs $(\xi,j)$ where $\xi\in \stX$ and $j\colon \pi(\xi)\arr T$, where $T$ is an affine $k$-scheme, is a nilpotent closed immersion.
A morphism $(\xi,\pi(\xi)\arrdi j T)\arr (\xi',\pi(\xi')\arrdi{j'} T')$ is a pair $(\alpha,\beta)$, where $\alpha\colon \xi\arr\xi'$ and $\beta\colon T\arr T'$ are such that the following diagram is commutative
   \[
  \begin{tikzpicture}[xscale=1.5,yscale=-1.2]
    \node (A0_0) at (0, 0) {$\pi(\xi)$};
    \node (A0_1) at (1, 0) {$T$};
    \node (A1_0) at (0, 1) {$\pi(\xi')$};
    \node (A1_1) at (1, 1) {$T'$};
    \path (A0_0) edge [->]node [auto] {$\scriptstyle{j}$} (A0_1);
    \path (A0_0) edge [->]node [auto,swap] {$\scriptstyle{\pi(\alpha)}$} (A1_0);
    \path (A0_1) edge [->]node [auto] {$\scriptstyle{\beta}$} (A1_1);
    \path (A1_0) edge [->]node [auto] {$\scriptstyle{j'}$} (A1_1);
  \end{tikzpicture}
  \]
An object $(\xi,\pi(\xi)\arrdi j T)\in\stX_{\iinf/k}$ is called \emph{extendable} if there exists a map $\xi\arr \eta$ in $\stX$ such that $\pi(\xi)\arr \pi(\eta)$ factors through $\pi(\xi)\arr T$. If $\stX$ is a fibered category this simply means that $\xi\colon \pi(\xi)\arr \stX$ extends along $\pi(\xi)\arr T$.
We define the big stratified site $\stX_{\str/k}$ of $\stX$ as the full sub category of extendable objects of $\stX_{\iinf/k}$. We will consider $\stX_{\str/k}$ and $\stX_{\iinf/k}$ as categories over $k$ via the association $(\xi,j:\pi(\xi)\arr T)\longmapsto T$. Notice that there is a canonical map $\stX\arr\stX_{\str/k}\subseteq \stX_{\iinf/k}$ of categories over $\Aff/k$  given by $\xi\longmapsto (\xi,\id_{\pi(\xi)})$. If $\stX$ is a fibered category over $k$ then also $\stX_\str$ and $\stX_\iinf$ are fibered categories.

Let $\stY$ be a fibered category over $k$. Following notations and definitions from \ref{the T construction} we define the following objects:
\begin{itemize}
 \item if $\stX\arr \stX_\shT=\stX_{\str/k}$ then $\shT_k$ will be replaced by $\Strat_k$ and an object of $\Strat_k(\stX,\stY)=\Hom_k^c(\stX_{\str/k},\stY)$ will be called a stratified map, while an object of $\Strat_k(\stX)=\Vect(\stX_{\str/k})$ a stratified sheaf on $\stX$.
 \item if $\stX\arr \stX_\shT=\stX_{\iinf/k}$ then $\shT_k$ will be replaced by $\Crys_k$ and an object of $\Crys_k(\stX,\stY)=\Hom_k^c(\stX_{\iinf/k},\stY)$ will be called a crystal map, while an object of $\Crys_k(\stX)=\Vect(\stX_{\iinf/k})$ a crystal of sheaves on $\stX$.
\end{itemize}
 When $k$ is clear from the context it will be omitted.
\end{defn}

If $\stZ$ is a scheme and $\stX$ is the category of open subsets of $\stZ$ then $\Crys(\stX)$ is, by construction, the usual category of crystals of sheaves. Although we don't prove it here, it is possible to show that the restriction $\Crys(\stZ)\arr \Crys(\stX)$ is an equivalence. In this paper we prefer to consider the big site $\Aff/\stZ$ instead of the small Zariski site to extends the theory to algebraic stacks and fibered categories.

The main result of this section is the following Theorem:
\begin{thm}\label{main thm for str and inf}
Let $\stZ$ be a category fibered in groupoids over $k$. Then:
\begin{enumerate}
 \item axiom C holds for $\stZ\arr \stZ_{\str}$ and $\stZ\arr \stZ_{\iinf}$;
 \item axiom A implies axiom B for $\stZ\arr \stZ_{\str}$ and $\stZ\arr \stZ_{\iinf}$;
 \item axioms A and B hold for $\stZ\arr \stZ_{\str}$ if $\stZ$ admits an fpqc covering $U\arr \stZ$ where $U$ is a scheme over $k^{\perf}$ such that all its nonempty closed subsets contain an adically separated point (see \ref{defn: adically separated schemes}); if moreover $\stZ$ is connected and there exists a map $\Spec L\arr \stZ$ where $L/k$ is a field extension which is separably generated up to a finite extension (see \ref{separably generated}) then $\Hl^0(\odi{\stZ_{\str/k}})=\Hl^0(\odi\stZ)_{\et,k}$;
 \item axiom A and B holds for $\stZ\arr\stZ_{\iinf}$ if $\stZ$ is an algebraic stack locally of finite type over $k$; moreover in this case $\Hl^0(\odi{\stZ_{\iinf/k}})=\Hl^0(\odi\stZ)_{\et,k}$.
\end{enumerate}
\end{thm}

\begin{proof}[Proof of Theorem \ref{main thm for str and inf},1).]
This follows by definition, because a map to something \'etale extends uniquely along a nilpotent closed immersion. 
\end{proof}

\begin{rmk}
 If $\stY$ is a fibered category there are functors $\Crys_k(\stX,\stY),\Strat_k(\stX,\stY) \arr \Hom_k^c(\stX,\stY)$ and there is a map $\Crys_k(\stX,\stY)\arr \Strat_k(\stX,\stY)$ over $\Hom_k^c(\stX,\stY)$. Moreover if $\stX$ is defined over a field extension $L$ of $k$ there is a forgetful functor $\stX_{\iinf/L}\arr \stX_{\iinf/k}$ maintaining the stratified sites and inducing maps
 \[
 \Crys_k(\stX,\stY)\arr \Crys_L(\stX,\stY\times_k L)\comma \Strat_k(\stX,\stY)\arr \Strat_L(\stX,\stY\times_k L)
 \]
\end{rmk}

\begin{lem}\label{Str and Crys along nilpotent closed immersions}
Let $i\colon \stX\arr\stX'$ be a nilpotent closed immersion of categories fibered in groupoids over $k$ and $\stY$ be a fibered category over $k$. Then the restriction $\Crys(\stX',\stY)\arr \Crys(\stX,\sY)$ is an equivalence. If $i$ admits a retraction then also $\Strat(\stX',\stY)\arr \Strat(\stX,\sY)$ is an equivalence.
\end{lem}
\begin{proof}
We will consider only the stratified case since the crystal one is completely analogous.
 There is a restriction functor $\psi\colon \stX_{\str}\arr \stX'_{\str}$ obtained by composing with $i\colon\stX\arr\stX'$. Using the pullback along $i$ we also get a morphism $\phi\colon\stX'_{\str}\arr\stX_{\str}$ (the extendability condition is preserved). It is easy to define base-preserving natural transformations $\phi\circ \psi \arr \id$ and $\psi\circ\phi \arr \id$. Since stratified maps sends all arrows to Cartesian arrows, it follows that $\Strat(\stX,\stY)\arr \Strat(\stX',\sY)$ obtained by composing with $\phi$ is a quasi-inverse of the map in the statement.
\end{proof}

Usually stratified sheaves for a scheme are defined using higher diagonals. In the following proposition we show that our definition is equivalent to the classical one. Let us recall the definition of higher diagonals.

\begin{defn}
 Let $S$ be a base scheme and $X$ be an $S$-scheme. The $n$-th diagonal of $X$ over $S$ at level $r\in\N$, denoted $P^n_{X/S}(r)$ is defined as follows: pick an open $U\subseteq X^{\times_S (r+1)}$ containing the diagonal as a closed subscheme with ideal sheaf $\shI$ and set $P^n_{X/S}(r)=\Spec(\odi U/\shI^{n+1})$. 
\end{defn}

\begin{prop}\label{classical definition of stratifications}
 Let $X$ be a $k$-scheme and $\stY$ be a Zariski stack over $\Aff/k$. The category $\Strat(X,\stY)$ is canonically equivalent to the category $\hat\Strat(X,\stY)$ whose objects are tuples $(\eta,\sigma_n)_{n\in \N}$ where $\eta\in\stY(X)$ and $(\sigma_n)_{n\in\N}$ is a compatible system of isomorphisms between the two pullbacks of $\eta$ to $\stY(P^n_{X/k})$ satisfying the cocycle condition on $\stY(P^n_{X/k}(2))$, while the morphisms are maps in $\stY(X)$ compatible with the $\sigma_n$.
 \end{prop}
 \begin{proof}
 Let $\shF\in \Strat(X,\sY)$ be an object and consider the maps 
  $$\xymatrix{X\ar[r]^-{j_n}&P^n_{X/k}(2)\ar@/^1pc/[rr]^{p_{12}}\ar[rr]^{p_{23}}\ar@/_1pc/[rr]^-{p_{13}}&& P_{X/k}^n\ar@/^/[rr]^{p_1}\ar@/_/[rr]_{p_2}&&X&}$$  
where $p_{i}$ and $p_{ij}$ are the projections. Since all the maps $X\arr P^n_{X/k}(r)$ are nilpotent closed immersions with a retraction for all $n,r\in \N$, by \ref{Str and Crys along nilpotent closed immersions} we see that applying $\Strat(-,\stY)$ to the above sequence of maps we get a sequence of equivalences. This easily yields compatible maps $\sigma_n\colon p_2^*\shF\arrdi\simeq p_1^*\shF$ in $\Strat(P^n_{X/k},\stY)$ satisfying the cocycle condition in $\Strat(P^n_{X/k}(2),\stY)$. Applying the natural functor $\Strat(-,\stY)\arr \Hom(-,\stY)\simeq \stY(-)$ we obtain an object of $\hat\Strat(X,\stY)$. The association just defined  extends to a functor $\Strat(X,\sY)\to \hat{\Strat}(X,\sY)$.

A quasi-inverse can be defined as follows. For all $\chi=(U\arr T)\in X_{\str}$, where $U$ is an $X$-scheme, choose an extension $g_\chi\colon T\arr X$. Given $(\eta,\sigma_n)_{n\in\N}\in\hat\Strat(X,\stY)$ define $\Phi_X((\eta,\sigma_n)_{n\in\N})=\shF\in \Strat(X,\stY)$ as follows. For $\chi\in X_{\str}$ set $\shF(\chi)=g_\chi^*\eta$. Given a map $\psi\colon \chi\arr\chi'$ over $T\arrdi\alpha T'$ we have to specify a Cartesian arrow $\shF(\psi)\colon g_\chi^*\eta\arr g_{\chi'}^*\eta$ over $\alpha$.
By construction $U\arr T\arrdi{(g_{\chi'}\alpha,g_\chi)}X\times X$ factors through the diagonal. Since $U\arr T$ is nilpotent, we get a factorization of $(g_{\chi'}\alpha,g_\chi)\colon T\arrdi\beta P^n_{X/k}\subseteq X\times X$ for some $n\in \N$. The map $\shF(\psi)$ is $g_\chi^*\eta \simeq \beta^*\pr_2^*\eta \arrdi{\beta^*\sigma_{n}} \beta^*\pr_1^*\eta \simeq \alpha^*g_{\chi'}^*\eta \arr g_{\chi'}^*\eta$. The compatibility among the $\sigma_j$ tell us that $\shF(\psi)$ does not depend on the choice of $n$, while the cocycle condition and a similar argument show that $\shF$ is indeed a functor.

One can show that the two functors are quasi-inverse of each other.
%
 \end{proof}

\begin{lem}\label{key lemma for reducing to a retraction}
 Let $R$ be a $k$-algebra, $I$ be a nilpotent ideal and $L/k$ be a field extension with a $k$-map $L\arr R/I$. Then there exists an fpqc covering $R\arr R'$ and an isomorphism $R'/IR' \simeq R/I \otimes_L L^{\perf}$, where $L^{\perf}$ is the perfect completion of $L$.
 \end{lem}
\begin{proof}
A proof is required only if $p=\car k>0$. We show how to construct the ring $R'$ when $L^{\perf}$ is replaced by $L^{1/p}$. A simple induction on $\N$ will then give the desired algebra.
 By Zorn's lemma there exists a maximal subset $S\subseteq L-L^p$ such that
 \[
 \text{for all finite } T\subseteq S \text{ the map } L_T=L[X_t]_{t\in T}/(X_t^p-t)\arr L^{1/p}\text{ is injective}
 \]
 Moreover it is easy to show that $L^{1/p}\simeq \lim_T L_T$. For all $t\in S$ let $\hat t \in R$ be a lifting and set
 \[
 R_T=R[X_t]_{t\in T}/(X_t^p-\hat t) \text{ for } T\subseteq S\text{ finite and } R'=\lim_T R_T
 \]
 It is now easy to prove that $R'/IR' \simeq R/I \otimes_L L^{1/p}$ as required.
\end{proof}


\begin{proof}[Proof of Theorem \ref{main thm for str and inf},2).]
 Let $\shC=\Strat(\stZ)$ or $\shC=\Crys(\stZ)$, By \ref{TX abelian} the category $\shC$ is abelian, thus one has to show that if $\shF\in \shC$ and $\shF_{|\stZ}=0$ then $\shF=0$. This follows because if $j\colon U\arr T$ is a nilpotent closed immersion and $\E\in \Vect(T)$ is such that $j^*\E=0$ then $\E=0$.
\end{proof}

\begin{proof}[Proof of Theorem \ref{main thm for str and inf},3), first sentence.]
 Let $\shF\in \Strat(\stZ,\QCoh_{\textup{fp}})$. Since the objects of $\stZ_{\str}$ are extendable by definition, it is enough to show that $\shF_{|\stZ}\in \QCoh_{\textup{fp}}(\stZ)$ is locally free in oder to conclude that $\shF\in \Strat(\stZ)$. Using the existence of an atlas as in the statement, fpqc descent and \ref{m-adically sepated rings and free modules} we can reduce the problem to the case $\stZ=\Spec R$, where $(R,m)$ is a local ring defined over $k^{\perf}$ and with $m$ nilpotent. Using the restriction $\Strat_k(\stZ)\arr\Strat_{k^{\perf}}(\stZ)$ we can also assume $k$-perfect. By \ref{key lemma for reducing to a retraction} applied when $I$ the maximal ideal of $R$ and $L$ is its residue field we can assume that $L$ is perfect. 
 Since an extension of perfect fields is formally smooth (see \cite[\href{http://stacks.math.columbia.edu/tag/031U}{031U}]{SP}), we can assume that the nilpotent closed immersion $\Spec L\arr \Spec R$ has a retraction $\sigma\colon \Spec R\arr \Spec L$. Thanks to \ref{Str and Crys along nilpotent closed immersions}, there exists $\shG\in \Strat(\Spec L,\QCoh_{\textup{fp}})$ restricting to $\shF$ along the retraction $\sigma$. In particular
 \[
 \shF(\id_{R}, \Spec R\arrdi \id \Spec R)\simeq \shG(\sigma,\Spec R\arrdi \id \Spec R)\simeq \sigma^*\shG(\id_{L}, \Spec L\arrdi \id \Spec L)
 \]
is free as required.
\end{proof}

\begin{ex}\label{example of bad stratification}
 If we don't assume that the scheme $U$ in \ref{main thm for str and inf}, 2) is defined over $k^{\perf}$ then the conclusion is false, even if $U$ is the spectrum of an Artinian ring.
 Consider $k=\F_p(z)$, $L=k^{\perf}$ and $A=L[x]/(x^2)$. We regard $A$ as a $k$-algebra via the morphism $\lambda\colon k\arr A$ mapping $z$ to $z-x$. We are going to construct an object $\shF\in \Strat_k(A,\QCoh_{\textup{fp}})$ which is not a vector bundle. Write $x=z-\lambda(z)$ and let $y_n\in L$ such that $y_n^{p^n}=z$. If $J$ is the ideal of the diagonal in $A\otimes_k A$ we have
 \[
 x\otimes 1 - 1\otimes x = z\otimes 1 - 1\otimes z = (y_n\otimes 1 - 1 \otimes y_n)^{p^n}\in J^{p^n}
 \]
 By \ref{classical definition of stratifications} we can conclude that $A\arrdi x A$ is a map in $\Strat_k(A)$, where $A$ is the trivial object. Since in $\Strat_k(A,\QCoh)$ we can take cokernels pointwise, we can conclude that $A/x$ has a stratification, even though is not locally free.
 
 This also show that $\Strat_k(A)\arr \Strat_k(L)$ is not an equivalence even though $\Spec L\arr \Spec A$ is a nilpotent closed immersion. To see this we prove that $\Strat_k(L)\simeq \Vect(L)$. Indeed for $l\in \N$ let $J_l=\Ker(L^{\otimes_k l}\arrdi {\mu_l} L)$. If $x\in J_l$, since $L^{\otimes_k l}$ is perfect, there exists $y\in L^{\otimes_k l}$ such that $y^p=x$. Since $\mu_l(x)=\mu_l(y)^p$ we see that $y\in J_l$, so that $J_l^2=J_l$. Thus all higher diagonals of $\Spec L$ over $k$ are trivial, which implies the result. 
\end{ex}

\begin{lem}\label{from the big to the small site}
If $V$ is an affine scheme over $k$ and $\chi\in V_{\iinf}$ there exists $\chi'=(\id_V,V\arr T')\in V_{\iinf}$ and a map $\chi\arr \chi'$. If $V$ is of finite type over $k$ we can furthermore assume that $T'$ is of finite type too. 
\end{lem}
\begin{proof}
Set $V=\Spec A$, and consider a ring $B$ with a nilpotent ideal $I$ and a map $\phi\colon A\arr B/I$. Set $\pi\colon B\arr B/I$ the projection and write $A=k[\underline x]/J$, where $\underline x=(x_s)_{s\in S}$ is a set of variables. For all $s\in S$ choose $b_s\in B$ such that $\pi(b_s)=\phi(x_s)$ and denote $\psi\colon k[\underline x]\arr B$ the map such that $\psi(x_s)=b_s$. Since $\psi(J)\subseteq I$ and $I$ is nilpotent, there exists $n\in \N$ such that $J^n\subseteq \Ker(\psi)$. We can therefore choose $T'=\Spec (k[\underline x]/J^n)$. If $V$ is of finite type then $S$ can be chosen finite and therefore also the last claim holds.
\end{proof}

\begin{proof}[Proof of Theorem \ref{main thm for str and inf},4), first sentence.]
Let $\shF\in \Crys(\stZ,\QCoh_{\textup{fp}})$ and $(\xi,V\arrdi j T)\in \stZ_{\iinf}$. We must show that $\shF(\xi,j)\in\QCoh_{\textup{fp}}(T)$ is locally free. Let $U\arr \stZ$ be a smooth atlas. There are Cartesian diagrams
  \[
  \begin{tikzpicture}[xscale=1.5,yscale=-1.2]
    \node (A0_1) at (1, 0) {$T'$};
    \node (A0_3) at (3, 0) {$T$};
    \node (A1_0) at (0, 1) {$V'$};
    \node (A1_1) at (1, 1) {$V_U$};
    \node (A1_2) at (2, 1) {$V$};
    \node (A2_1) at (1, 2) {$U$};
    \node (A2_2) at (2, 2) {$\stZ$};
    \path (A1_2) edge [->]node [auto] {$\scriptstyle{j}$} (A0_3);
    \path (A1_0) edge [->]node [auto] {$\scriptstyle{}$} (A1_1);
    \path (A1_1) edge [->]node [auto] {$\scriptstyle{}$} (A1_2);
    \path (A1_2) edge [->]node [auto] {$\scriptstyle{}$} (A2_2);
    \path (A1_1) edge [->]node [auto] {$\scriptstyle{}$} (A2_1);
    \path (A1_0) edge [->]node [auto] {$\scriptstyle{j'}$} (A0_1);
    \path (A2_1) edge [->]node [auto] {$\scriptstyle{}$} (A2_2);
    \path (A0_1) edge [->]node [auto] {$\scriptstyle{}$} (A0_3);
  \end{tikzpicture}
  \]
where the map $T'\arr T$ is \'etale and surjective. The above diagram is obtained using that the smooth surjective map $V_U\arr V$ has sections in the \'etale topology and that the \'etale map $V'\arr V$ always extends along a nilpotent closed immersion by \cite[Expos\'e VIII, Thm 1.1]{SGA4}. By descent we can assume $\stZ=U$ and, since the problem is Zariski local, that $U$ is affine. By \ref{from the big to the small site} we can further assume that $V=U$, $\xi=\id$ and that $T$ is of finite type over $k$.
Using \ref{Str and Crys along nilpotent closed immersions}, there exists $\shG\in\Crys(T)$ restricting to our $\shF\in\Crys(U)$. In particular
\[
\shF(\id_U,U \arrdi j T)\simeq \shG(j,U \arrdi j T)=\shG(\id_T,T\arrdi \id T)=(\shG_{|T_{\str}})(\id_T,T\arrdi \id T)
\]
This sheaf is locally free because $T\times_k k^{\perf}\arr T$ is an fpqc atlas, $T\times_k k^{\perf}$ is Noetherian and therefore, by \ref{main thm for str and inf},2), $\shG_{|T_{\str}}\in \Strat(T)$.
\end{proof}

\begin{ex}
 Let $k=\F_p(z)$ and $L=k^{\perf}$. We are going to construct a $\shF\in \Crys_k(L,\QCoh_{\textup{fp}})$ such that $\shF\notin \Crys_k(L)$. More precisely we construct a non zero $a\colon \odi{(\Spec L)_{\iinf}}\arr \odi{(\Spec L)_{\iinf}}$ which is $0$ in $\Strat_k(L)=\Vect(L)$ (see \ref{example of bad stratification}) and show that $\shF=\Coker(a)$ (pointwise) satisfies the requirements. In particular if follows that $\Crys_k(L)\arr \Strat_k(L)$ is not faithful.
 Let $\chi \in (\Spec L)_{\iinf}$ given by a map $L\arr B/I$, where $B$ is a $k$-algebra and $I$ a nilpotent ideal. Using the Frobenius it is easy to show that there exists a unique $\F_p$-linear map $\phi_\chi\colon L\arr B$ lifting the given $k$-map $L\arr B/I$. The map $a$ we are looking for is given by $a(\chi)=\phi_\chi(z)-z$. We have $a(\id_L,\id_L)=0$ so that $a=0$ in $\Strat_k(L)$. Consider $B=L[x]/(x^2)$ with the $k$-structure $\lambda\colon k\arr B$, $\lambda(z)=z-x$, $I=(x)$. If $\chi\in (\Spec L)_{\iinf}$ is the corresponding object, by construction $a(\chi)=x\neq 0$ in $B$ and therefore $\shF(\chi)=B/x$ which is not locally free.
\end{ex}

\begin{lem}\label{twist as intersection}
 Let $K$ be a purely transcendental field extension over a field $k$ of positive characteristic and  $L/K$ be a finite separable extension. Then the intersection of all fields $E$ such that $L^{(i)}\subseteq E \subseteq L$ and $L/E$ is finite coincides with the image of the relative Frobenius $L^{(i)}\arr L$
\end{lem}
\begin{proof}
Notice that $L^{(i)}$ is a field by \ref{separable extension base change by purely inseparable is a field} because $L/k$ is separable. In particular we will identify $L^{(i)}$ with its image under the relative Frobenius.
 Let $\{z_s\}_{s\in S}$ be a transcendental basis of $K/k$ and let $\alpha\in L$ such that $L=K(\alpha)$. Given $T\subseteq S$ set
 \[
 K_T=k(z_s \st s\notin T\comma z_s^{p^i}\st s\in T)\subseteq K=k(z_s)_{s\in S}
 \]
 We have that $L^{(i)}=K_S(\alpha^{p^i}) \subseteq K_T(\alpha^{p^i})$ and that $L$ is finite over $K_T(\alpha^{p^i})$ if $T$ is finite. 
 Since $K/K_T$ is purely inseparable and $\alpha^{p^i}$ is separable over $K_T$, $K_T[\alpha^{p^i}]\otimes_{K_T}K$ is a field. It follows that, for all $T$, the surjective map $K_T[\alpha^{p^i}]\otimes_{K_T}K\arr K[\alpha^{p^i}]$ is an isomorphism. Thus we have the equality
 \[
 n:=[K_T(\alpha^{p^i}):K_T]=[K(\alpha^{p^i}):K]
 \] 
 Let $\beta\in K_T(\alpha^{p^i})$ for all $T$ finite. Then $\beta$ can be written uniquely as a linear combination of $1,{\alpha^{p^i}},\cdots,{\alpha^{(n-1)p^i}}$ with coefficients in $K$. 
 Since $\beta\in K_T(\alpha^{p^i})$ for all $T$ finite it follows that the coefficients of the linear combination lie in the intersection of all $K_T$ for $T$ finite. This intersection is $K_S$, so that $\beta \in K_S(\alpha^{p^i})=L^{(i)}$.
\end{proof}

\begin{lem}\label{higher diagonal for etale maps}
 Let $S$ be a base scheme and $f\colon Y\arr X$ be an \'etale map of $S$-schemes. Then the following commutative diagrams are Cartesian for all $i=1,2$.
   \[
  \begin{tikzpicture}[xscale=2.0,yscale=-1.2]
    \node (A0_0) at (0, 0) {$P^n_{Y/S}$};
    \node (A0_1) at (1, 0) {$P^n_{X/S}$};
    \node (A1_0) at (0, 1) {$Y$};
    \node (A1_1) at (1, 1) {$X$};
    \path (A0_0) edge [->]node [auto] {$\scriptstyle{f_n}$} (A0_1);
    \path (A1_0) edge [->]node [auto] {$\scriptstyle{f}$} (A1_1);
    \path (A0_1) edge [->]node [auto] {$\scriptstyle{p_i}$} (A1_1);
    \path (A0_0) edge [->]node [auto] {$\scriptstyle{p_i}$} (A1_0);
  \end{tikzpicture}
  \]
Here the $p_i$'s are the projections, while $f_n$ is the map induced by $f\times f \colon Y\times_S Y\arr X\times_S X$.
\end{lem}
\begin{proof}
 We have Cartesian diagrams
   \[
  \begin{tikzpicture}[xscale=1.6,yscale=-1.2]
    \node (A0_0) at (0, 0) {$Y$};
    \node (A0_1) at (1, 0) {$Z_i$};
    \node (A0_2) at (2, 0) {$Y$};
    \node (A1_0) at (0, 1) {$X$};
    \node (A1_1) at (1, 1) {$P^n_{X/S}$};
    \node (A1_2) at (2, 1) {$X$};
    \path (A0_0) edge [->]node [auto] {$\scriptstyle{\alpha_i}$} (A0_1);
    \path (A0_1) edge [->]node [auto] {$\scriptstyle{}$} (A0_2);
    \path (A1_0) edge [->]node [auto] {$\scriptstyle{}$} (A1_1);
    \path (A1_1) edge [->]node [auto] {$\scriptstyle{p_i}$} (A1_2);
    \path (A0_2) edge [->]node [auto] {$\scriptstyle{f}$} (A1_2);
    \path (A0_0) edge [->]node [auto] {$\scriptstyle{f}$} (A1_0);
    \path (A0_1) edge [->]node [auto] {$\scriptstyle{}$} (A1_1);
  \end{tikzpicture}
  \]
Notice that $Z_1$ is a subscheme of $Y\times_S X$, while $Z_2$ is a subscheme of $X\times_S Y$. Since $X\arr P^n_{X/S}$ is a nilpotent closed immersion and the maps $Z_i\arr P^n_{X/S}$ are \'etale, by \cite[Expos\'e VIII, Th\'eor\` eme 1.1]{SGA4}, there exists an isomorphism $\lambda\colon Z_1\arr Z_2$ over $P^n_{X/S}$ and such that $\lambda \circ \alpha_1 = \alpha_2 \colon Y\arr Z_2$. The map $Y\times_S Y \arr Y\times_S X$ induced a map $P^n_{Y/S}\arrdi a Z_1$ which is compatible with the first projection. We must show this map is an isomorphism. We have a commutative diagram
  \[
  \begin{tikzpicture}[xscale=2.4,yscale=-1.2]
    \node (A0_0) at (0, 0) {$P^n_{Y/S}$};
    \node (A0_1) at (1, 0) {$Z_1$};
    \node (A0_2) at (2, 0) {$Y\times_X Y$};
    \node (A0_3) at (3, 0) {$X\times_SY$};
    \node (A0_4) at (4, 0) {$Y$};
    \node (A1_0) at (0, 1) {$Y$};
    \node (A1_2) at (2, 1) {$Y\times_S X$};
    \node (A1_3) at (3, 1) {$X\times_S X$};
    \node (A1_4) at (4, 1) {$X$};
    \path (A1_0) edge [->,bend left=30]node [auto] {$\scriptstyle{f}$} (A1_4);
    \path (A0_1) edge [->,bend right=45]node [auto] {$\scriptstyle{\lambda}$} (A0_3);
    \path (A0_3) edge [->]node [auto] {$\scriptstyle{p_2}$} (A0_4);
    \path (A0_1) edge [->]node [auto] {$\scriptstyle{b}$} (A0_2);
    \path (A0_3) edge [->]node [auto] {$\scriptstyle{}$} (A1_3);
    \path (A0_0) edge [->]node [auto] {$\scriptstyle{a}$} (A0_1);
    \path (A0_4) edge [->]node [auto] {$\scriptstyle{f}$} (A1_4);
    \path (A0_1) edge [right hook->]node [auto] {$\scriptstyle{}$} (A1_2);
    \path (A0_2) edge [->]node [auto] {$\scriptstyle{}$} (A1_2);
    \path (A0_0) edge [->]node [auto] {$\scriptstyle{p_2}$} (A1_0);
    \path (A0_2) edge [->]node [auto] {$\scriptstyle{}$} (A0_3);
    \path (A1_2) edge [->]node [auto] {$\scriptstyle{}$} (A1_3);
    \path (A1_3) edge [->]node [auto] {$\scriptstyle{p_2}$} (A1_4);
  \end{tikzpicture}
  \]
where the square diagrams are Cartesian and the map $b$ is induced by the universal property of fiber product. Notice that $b$ is a monomorphism because $\lambda$ is a monomorphism. The equality $\lambda \circ \alpha_1 = \alpha_2 \colon Y\arr Z_2$ implies that $Y\arr Z_1 \arrdi b Y\times_S Y$ is the diagonal. Since $Y\arr Z_1$ is, by construction, a nilpotent closed immersion whose sheaf of ideal to the power $n$ vanishes, it follows that $b$ factors through $P^n_{Y/S}$. Thus it is enough to show that $P^n_{Y/S}\arrdi a Z_1 \arrdi b Y\times_S Y$ is the inclusion. By construction $p_1=p_1\circ (ba)$ and we must prove that $p_2=p_2 \circ (ba)$. By the commutativity of diagram above we obtain a map
\[
\gamma\colon P^n_{Y/S}\arrdi{(p_2,p_2\circ(ba))} Y\times_X Y
\]
whose composition along $Y\arr P^n_{Y/S}$ is the diagonal. Since $Y\arr X$ is \'etale, it follows that the diagonal is an open immersion. Since $Y\arr P^n_{Y/S}$ is an homeomorphism, it follows that $\gamma$ factors through the diagonal, that is $p_2=p_2 \circ (ba)$ as required.
\end{proof}

\begin{prop}\label{endo of trivial object for a field}
 Let $L/k$ be a field extension separably generated up to a finite extension. Then $\End_{\Strat_k(L)}(1)=L_{\et,k}$.
\end{prop}
\begin{proof}
Applying \ref{when TX is Tannakian} we can conclude that $\End_{\Strat_k(L)}(1)$ is a subfield of $L$. If $E/k$ is a separable and finite field extension then $(\Spec E)_{\Strat/k}=(\Spec E)_{\Strat/E}$ and therefore $\Strat_k(E)=\Vect(E)$. By functoriality this implies $L_{\et,k}\subseteq \End_{\Strat_k(L)}(1)$. So we concentrate on the other inclusion. We first deal with a particular case.

\emph{The case $L/k$ finite and purely inseparable when $\car k > 0$.} 
We have to prove that $\End_{\Strat_k(L)}(1)=k$.
 Set $A=L\otimes_k \overline k$, which is a local and finite $\overline k$-algebra with residue field $\overline k$. Since the maximal ideal of $A$ is nilpotent, by \ref{Str and Crys along nilpotent closed immersions} we see that $\End_{\Strat_{\overline k}(A)}(1)=\overline k$. Using the functor $\Strat_k(L)\arr \Strat_{\overline k}(A)$ we can conclude that $\End_{\Strat_k(L)}(1)$ is contained in the intersection of $L$ and $\overline k$ inside $A=L\otimes_k \overline k$, which coincides with $k$.

Coming back to the general statement,  we proceed by making some reductions. Consider $K\subseteq F\subseteq L$ where $K$ is purely transcendental, $F/K$ is algebraic and separable and $L/F$ is finite and purely inseparable. In what follows we will use \ref{separable extension base change by purely inseparable is a field} several times.

\emph{Reduction to the perfect case when $\car k > 0$.} Let $k^\perf$ be the perfect closure of $k$ and assume to know that the statement of the theorem holds for perfect fields. Consider the map
\[
\Strat_k(L)\arr \Strat_{k^\perf}(L\otimes_k k^\perf)
\]
The ring $L\otimes_k k^\perf$ is a finite extension of the field $F\otimes_k k^\perf$ and therefore it is a local $k^\perf$-algebra with a nilpotent maximal ideal. In particular, by \ref{Str and Crys along nilpotent closed immersions}, $\Strat_{k^\perf}(L\otimes_k k^\perf)=\Strat_{k^\perf}(E)$, where $E$ is the residue field of $L\otimes_k k^\perf$. Moreover $(L\otimes_k k^\perf)_{\et,k^\perf}=E_{\et,k^\perf}$. We can therefore conclude that $\End_{\Strat_k(L)}(1)$ lies in the intersection of $L$ and $(L\otimes_k k^\perf)_{\et,k^\perf}$ inside $L\otimes_k k^\perf$. This intersection is $L_{\et,k}$ because, using \ref{etale part and geometric connectdness}, we have
\[
L \otimes_{L_{\et,k}} (L\otimes_k k^\perf)_{\et,k^\perf} \simeq L\otimes_{L_{\et,k}} (L_{\et,k}\otimes_k k^\perf) \simeq L \otimes_k k^\perf
\]

\emph{Reduction to the separably generated case when $\car k > 0$.} Assume to know that the statement of the theorem holds for separably generated field extensions.
The ring $L^{(i,K)}$ is a finite and local algebra over the field $F^{(i,K)}$. Moreover its residue field $E_i$ is contained in $K(L^{p^i})$. Since $L/F$ is finite and purely inseparable we can choose $i$ such that $E_i\subseteq F$. There are functors 
$$
\Strat_k(L) \arr \Strat_k(L^{(i,K)}) \arr \Strat_k(F)
$$
which implies that $\End_{\Strat_k(L)}(1)$ lies in $F_{\et,k}$. Notice that here, to be precise, the $k$-structure of $F$ is the one given by $k\arr k \arr F$, where the first map is the $i$-th power of the Frobenius. Since $k$ is perfect, $F$, with this new structure, is still a separably generated extension of $k$.

\emph{Reduction to the case $L/K$ finite.} Let $\beta\in \End_{\Strat_k(L)}(1)$. We claim that $\beta\in \End_{\Strat_k(K(\beta))}(1)$. Given a field extension $Q/k$ denotes by $J_Q\subseteq Q\otimes_k Q$ the ideal of the diagonal. It is enough to prove that, if $Q/E$ is an algebraic and separable extension, then the map 
\[
\gamma\colon (E\otimes_k E)/J_E^n\arr (Q\otimes_k Q)/J_Q^n
\]
is injective for all $n$. Since the ideal $J_Q$ is generated by elements of the form $q\otimes 1 - 1\otimes q$ for $q\in Q$, it is easy to show that the functor $Q\mapsto (Q\otimes_k Q)/J_Q^n$ commutes with filtered direct limits. In particular, for the injectivity of $\gamma$, one can assume that $Q/E$ is finite. In this case the result follows from \ref{higher diagonal for etale maps}.

We can therefore assume that $L/K$ is a finite and separable extension.

 \emph{Computation via relative Frobenius when $\car k > 0$.}
 We show that 
 \begin{equation}\label{end ring and frobenius}
  \End_{\Strat_k(L)}(1) = \bigcap_j L^{(j)}
 \end{equation}
 Here we are identifying $L^{(j)}$ with the image of the relative Frobenius $\phi_{j,L}\colon L^{(j)}\arr L$.
 
 $\supseteq$ By \ref{classical definition of stratifications} we have that $\End_{\Strat_k(L)}(1)$ is the intersection of all 
 $$
 L_n = \Ker (L\arrdi{\id \otimes 1 - 1\otimes \id} (L\otimes_kL)/J^n)
 $$
 where $J$ is the ideal of the diagonal. Since
 \[
 \phi_{j,L}(\sum_q z_q\otimes \lambda_q)\otimes 1 - 1\otimes \phi_{j,L}(\sum_q z_q\otimes \lambda_q)=\sum_q \lambda_q(z_q\otimes 1-1\otimes z_q)^{p^j} \in J^{p^j}
 \]
 we get $\Imm\phi_{j,L}\subseteq L_{p^j}$ for all $j$.
 
 $\subseteq$ If $k\subseteq E \subseteq L$ is an intermediate field extension with $L/E$ finite and purely inseparable then  $\End_{\Strat_E(L)}(1)=E$. Using the functor $\Strat_k(L)\arr \Strat_E(L)$ we see that $\End_{\Strat_k(L)}(1)\subseteq E$. By \ref{twist as intersection} we can conclude that $\End_{\Strat_k(L)}(1)\subseteq L^{(j)}$ for all $j\in \N$.
 
 \emph{Conclusion.}
Let $x\in \End_{\Strat_k(L)}(1)$. We are going to show that $x\in L$ is algebraic over $k$. Since $L/k$ is separably generated, this will imply $x\in L_{\et,k}$. Write $K=k(z_s)_{s\in S}$.

Assume by contradiction that $x$ is transcendental and let $f(X)=X^n+a_{1}X^{n-1}+\cdots+a_n$ be the minimal polynomial of $x$ over $K$. If $\car k > 0$ we make the following simplification. Since $f\notin k[X]$, there exists a maximum $r\in \N$ such that $f\in K^{(r)}[X]$. On the other hand $x\in \End_{\Strat_k(L)}(1)=\End_{\Strat_k(L^{(r)})}(1)$ by (\ref{end ring and frobenius}), $L^{(r)}/K^{(r)}$ is finite and separable, $K^{(r)}$ is purely transcendental and $f$ is also the minimal polynomial of $x$ over $K^{(r)}$. Thus if $\car k > 0$ we can further assume that $f\notin K^{(1)}[X]$, that is $r=0$.

 Since $x\in \End_{\Strat_k(L)}(1)$ and thanks to \ref{classical definition of stratifications}, we have $0=d(x)=x\otimes 1 - 1\otimes x\in I/I^2\simeq \Omega_{L/k}$ where $I$ is the ideal of the diagonal in $L\otimes_k L$. Thus $0=d(f(x))=d(a_1)x^{n-1}+\cdots+d(a_n)$.  Since $L/K$ is finite and separable, $\{d(z_s)\}_{s\in S}$ is a free basis of $\Omega_{L/k}$. Since $f$ is the minimal polynomial we can conclude that $\partial a_i/\partial z_s=0$ for all $i$ and $s$. If $\car k = 0$ this implies $f\in k[X]$ contradicting the assumption. If $\car k > 0$ this tells us that $f\in K^{(1)}[X]$, which is again a contradiction.
 \end{proof}

\begin{proof}[Proof of Theorem \ref{main thm for str and inf},3), second sentence.] Since axioms A and B holds and $\stZ$ is connected, by \ref{when TX is Tannakian} we know that $\Hl^0(\odi{\stZ_{\str}})=F\subseteq \Hl^0(\odi\stZ)$ is a field. By pulling back via $\Spec L\arr \stZ$ we get a map $F\arr \Hl^0(\odi{(\Spec L)_{\str}})=\End_{\Strat(L/k)}(1)=L_{\et,k}$, where we have used \ref{endo of trivial object for a field}. So $F\subseteq \Hl^0(\odi\stZ)_{\et,k}$. The other inclusion follows pulling back along $\stZ\arr \Spec \Hl^0(\odi\stZ)_{\et,k}$ and using again \ref{endo of trivial object for a field}.
\end{proof}

\begin{proof}[Proof of Theorem \ref{main thm for str and inf},4), second sentence.]
We can assume $\stZ$ connected and set $$F:=\Hl^0(\odi{\stZ_{\iinf}})\subseteq \Hl^0(\odi\stZ).$$ Using the map $\Crys(\stZ)\arr \Strat(\stZ)$ we can conclude that $F\subseteq \Hl^0(\odi\stZ)_{\et}$. The other inclusion follows pulling back along $\stZ\arr \Spec \Hl^0(\odi\stZ)_{\et,k}$: If $Q/k$ is a separable and finite field extension then $(\Spec Q)_{\iinf/k}=(\Spec Q)_{\iinf/Q}$ so that $\Crys(Q/k)=\Vect(Q)$.
\end{proof}

\subsection{F-divided structures}

In this section we fix a base field $k$ with positive characteristic $p$.
\begin{defn}\label{definition of Fdivided sheaves}
 Let $\stZ$ be a category fibered in groupoids over $k$. The chain of relative Frobenius of $\stZ$
 \[
 \stZ\arr \stZ^{(1,k)}\arr \stZ^{(2,k)}\arr \cdots
 \]
 defines a direct system of fibered categories over $k$ indexed by $\N$ and we will denote by $\stZ^{(\infty,k)}$ its limit, which is a category fibered in groupoids over $k$ (see \ref{limit of fibered categories}).
 Let $\stY$ be a fibered category over $k$. Following notations and definitions from \ref{the T construction} we define the following objects: if $\stX=\stZ$, $\stX_\shT=\stZ^{(\infty,k)}$ and the map $\stX\arr \stX_\shT$ is the one induced by the limit, then $\shT_k$ will be replaced by $\Fdiv_k$ and an object of $\Fdiv_k(\stZ,\stY)=\Hom_k^c(\stZ^{(\infty,k)},\stY)$ will be called an F-divided map, while an object of $\Fdiv_k(\stZ)=\Vect(\stZ^{(\infty,k)})$ an F-divided sheaf.
 When $k$ is clear from the context it will be omitted.
\end{defn}

Using \ref{limit over natural numbers} and  \ref{limit of fibered categories} we have a more concrete description, which will be the one used in this paper.
\begin{prop}
 Let $\stY$ be a fibered category and $\stZ$ be a category fibered in groupoids over $k$. Then $\Fdiv(\stZ,\stY)$ is equivalent to the category of objects $(Q_n,\sigma_n)_{n\geq 0}$ where $Q_n\colon \stZ^{(n)}\arr \stY$ is a $k$-map and $\sigma_n\colon Q_{n+1}\circ R_n \arr Q_n$ are isomorphisms, where $R_n\colon \stZ^{(n)}\arr \stZ^{(n+1)}$ is the relative Frobenius. Under this equivalence the functor $\Fdiv(\stZ,\stY)\arr \Hom(\stZ,\stY)$ is given by $(Q_n,\sigma_n)_{n\geq 0}\longmapsto Q_0$.
\end{prop}

\begin{rmk}\label{Fdiv when k is perfect}
 If $k$ is perfect there is a even more concrete description of $F$-divided sheaves:  $\Fdiv(\stZ)$ is the category of tuples $(Q_n,\sigma_n)_{n\geq 0}$ where $Q_n$ is a vector bundle over $\stZ$ and $\sigma_n\colon F_\stZ^*Q_{n+1}\arr Q_n$ is an isomorphism. This is because the projections $\stZ^{(n)}\arr \stZ$ are equivalences
\end{rmk}

The main result of this section is the following Theorem:
\begin{thm}\label{main thm for FDiv}
Let $\stZ$ be a category fibered in groupoids over $k$. Then:
\begin{enumerate}
 \item if $\stZ$ is connected then axiom A implies axioms B,C and D for $\stZ\arr \stZ^{(\infty)}$ and that $\Pi_{\Fdiv(\stZ)}$ is a pro-smooth banded gerbe (see \ref{pro stuff for gerbes 2});
 \item axioms A and B holds for $\stZ\arr \stZ^{(\infty)}$ if $\stZ$ admits an fpqc covering $U\arr \stZ$ from a scheme $U$ such that all its nonempty closed subsets contains an adically separated point $q$ (see \ref{defn: adically separated schemes}) with $k(q)/k$ separable up to a finite extension (see \ref{separably generated}); if moreover $\stZ$ is connected and there exists a map $\Spec L\arr \stZ$ where $L/k$ is a field extension which is separably generated up to a finite extension (see \ref{separably generated}) then $\Hl^0(\odi{\stZ^{(\infty)}})=\Hl^0(\odi\stZ)_{\et,k}$.
\end{enumerate}
\end{thm}


\begin{lem}\label{Fdiv is insensible to thickening} 
Let $f:\stZ\arr \stZ'$ be a nilpotent closed immersion. Then the induced functor $\stZ^{(\infty)}\arr \stZ'^{(\infty)}$ is an equivalence. In particular for any fiber category $\stY$ the restriction $\Fdiv(\stZ,\stY) \to \Fdiv(\stZ',\stY)$ is an equivalence.
\end{lem}
\begin{proof}
 Let $N\in \N$ such that the $N$-th power of the ideals defining $\stZ\arr \stZ'$ are all $0$. For a given $i\in\N$ set $\stX=\stZ^{(i)}$ and $\stX'=\stZ'^{(i)}$. In particular this $N$ works also for the nilpotent closed immersion $\stX\arr \stX'$. Let $i_0\in \N$ with $p^{i_0}\geq N$. Notice that $F_{\stX'}^{i_0}\colon \stX'\arr \stX'$ factors through $\stX\subseteq \stX'$. This yield a $k$-map $\stX'\arr \stX^{(i_0)}$ making the following diagram commutative
   \[
  \begin{tikzpicture}[xscale=3.1,yscale=-1.2]
    \node (A0_0) at (0, 0) {$\stX$};
    \node (A0_1) at (1, 0) {$\stX^{(i_0)}$};
    \node (A1_0) at (0, 1) {$\stX'$};
    \node (A1_1) at (1, 1) {$\stX'^{(i_0)}$};
    \path (A0_0) edge [->]node [auto] {$\scriptstyle{}$} (A0_1);
    \path (A1_0) edge [->]node [auto] {$\scriptstyle{}$} (A0_1);
    \path (A0_0) edge [->]node [auto] {$\scriptstyle{}$} (A1_0);
    \path (A0_1) edge [->]node [auto] {$\scriptstyle{}$} (A1_1);
    \path (A1_0) edge [->]node [auto] {$\scriptstyle{}$} (A1_1);
  \end{tikzpicture}
  \]
Thus we get $k$-maps $\stZ'^{(i)}\arr \stZ^{(i+i_0)}$ with the above property. This yields a $k$-map $\stZ'^{(\infty)}\arr \stZ^{(\infty)}$ which is easily seen to be a quasi-inverse of $\stZ^{(\infty)}\arr \stZ'^{(\infty)}$.
\end{proof}

\begin{cor}\label{from Fdiv to crystals}
Let $\stZ$ be a category fibered in groupoids. Then there exists a natural $k$-functor $\psi\colon \stZ_{\iinf}\arr \stZ^{(\infty)}$ making the following diagram commutative
  \[
  \begin{tikzpicture}[xscale=1.0,yscale=-1.2]
    \node (A0_1) at (1, 0) {$\stZ$};
    \node (A1_0) at (0, 1) {$\stZ_{\iinf}$};
    \node (A1_2) at (2, 1) {$\stZ^{(\infty)}$};
    \path (A1_0) edge [->]node [auto] {$\scriptstyle{\psi}$} (A1_2);
    \path (A0_1) edge [->]node [auto] {$\scriptstyle{}$} (A1_2);
    \path (A0_1) edge [->]node [auto] {$\scriptstyle{}$} (A1_0);
  \end{tikzpicture}
  \]
In particular if $\stY$ is a fibered category we obtain a restriction functor $\Fdiv(\stZ,\stY)\arr \Crys(\stZ,\stY)$.
\end{cor}
\begin{proof}
 Given $(\xi,U\arrdi j T)\in \stZ_{\iinf}$ we obtain an arrow 
 $$
 T\arr T^{(\infty)}\xrightarrow{(j^{(\infty)})^{-1}}U^{(\infty)} \arrdi {\xi^{(\infty)}}\stZ^{(\infty)}
 $$
 In a similar way an arrow in $\stZ_{\iinf}$ can be mapped to an arrow in $\stZ^{(\infty)}$. 
\end{proof}

%
%


\begin{lem}\label{A implies B for Fdiv}
 If $\stZ$ is a category fibered in groupoids then, for $\stZ\arr\stZ^{(\infty)}$, axiom A implies axiom B.
\end{lem}
\begin{proof}
 By \ref{TX abelian} the category $\Fdiv(\stZ)$ is abelian, thus one has to show that if $\shF=(\shF_n,\sigma_n)\in \Fdiv(\stZ)$ and $\shF_0=0$ then $\shF=0$. If $C$ is an $\F_p$-algebra and $\xi\colon \Spec(C)\arr \stZ^{(n)}$ a map, there exists $\eta\colon \Spec(C)\arr \stZ$, namely the composition $\Spec(C)\arr \stZ^{(n)}\arr \stZ$ and a factorization of $\xi$ as $\Spec(C)\arr V=\stZ^{(n)}\times_\stZ \Spec C\arr \stZ^{(n)}$. If $C$ has the $k$-structure induced by $\eta\colon \Spec C\arr \stZ$, then $V=(\Spec C)^{(n)}$, so that the pullback of $(\shF_n)_{|V}$ along the relative Frobenius of $C$ coincides with $\eta^*\shF_0=0$ on $\Spec C$. Since the relative Frobenius for affine schemes is a homeomorphism, we can conclude that $(\shF_n)_{|V}=0$, so that $\xi^*\shF_n=0$.
\end{proof}

\begin{proof}[Proof of Theorem \ref{main thm for FDiv}, 2), first sentence.]
Let $(\shE_n,\sigma_n)_{n\in \N}\in \Fdiv(\stZ,\QCoh_{\textup{fp}})$ and $U\arr \stZ$ be the atlas of the statement. We have to show that all $\shE_i$ are locally free. Since all $U^{(i)}\to \stZ^{(i)}$ are fpqc coverings we can assume $\stZ=U$. Moreover, since the relative Frobenius is a homeomorphism, we can moreover assume $\stZ=\Spec R$, where $(R,m)$ is a local ring which is $m$-adically separated and whose residue field $L$ is separable up to a finite extension over $k$. As for each $i\in \N$, $(\shE_n,\sigma_n)_{n\geq i}$ is in $\Fdiv(\Spec(R)^{(i)},\QCoh_{\textup{fp}})$ and $R^{(i)}$ is again adically separated with respect to its maximal ideal and has a residue field separable up to a finite extension over $k$ by \ref{Bon Ring}, we see that we can always replace $R$ by $R^{(i)}$ and, using \ref{m-adically sepated rings and free modules}, that we can assume $m$ nilpotent. Since $L/k$ is separable up to a finite extension, we have a decomposition $k\subseteq E\subseteq L$, where $E/k$ is separable and $L/E$ is finite. It follows that, for $i\gg 0$, $R^{(i)}$ has separable residue field. On the other hand  $(R/m)^{(i)}=L^{(i)}$ is finite over the field $E^{(i)}$, so it Noetherian and therefore the maximal ideal of $R^{(i)}$ is nilpotent. Thus we can assume $L/k$ separable. By  \ref{Fdiv is insensible to thickening} applied on the nilpotent closed immersion $\Spec L\arr \Spec R$ we obtain $(\Spec L)^{(\infty)}\simeq (\Spec R)^{(\infty)}$. Thus we may assume $R=L$ a field. Since $L/k$ is separable all $L^{(i)}$ are fields by \ref{separable extension base change by purely inseparable is a field}. Thus all $\E_n$ are vector spaces and thus locally free.
\end{proof}


\begin{ex}
 Without the hypothesis on the residue fields in \ref{main thm for FDiv} the conclusion is false. Indeed if $k=\F_p(z)$ and $L=k^{\perf}$ then $\Fdiv_k(L)\neq \Fdiv_k(L,\QCoh_{\textup{fp}})$. Let $\phi_i \colon L^{(i+1)}\arr L^{(i)}$ the relative Frobenius, that is $\phi_i(a\otimes \lambda)=a^p\otimes \lambda$, and consider $x_i=z^{1/p^i}\otimes 1 - 1\otimes z\in L^{(i)}$. A direct computation shows that $\phi_i(x_{i+1})=x_i$ and $x_0=0$. The collection $x=(x_i)_{i\in \N}$ defines a morphism $\odi{(\Spec L)^{(\infty)}}\arr \odi{(\Spec L)^{(\infty)}}$. Its cokernel is not in $\Fdiv_k(L)$ because $x_0=0$ but $x_1\neq 0$.
\end{ex}

\begin{proof}[Proof of Theorem \ref{main thm for FDiv}, 2), second sentence.] Since axioms A and B holds and $\stZ$ is connected, by \ref{when TX is Tannakian} we know that $\Hl^0(\odi{\stZ^{(\infty)}})=F\subseteq \Hl^0(\odi\stZ)$ is a field. The inclusion $\Hl^0(\odi\stZ)_{\et,k}\subseteq F$ follows pulling back along $\stZ\arr \Spec \Hl^0(\odi\stZ)_{\et,k}$: if $Q/k$ is a separable and finite extension of $k$ then $\Spec Q=(\Spec Q)^{(\infty,k)}$ so that $\Fdiv_k(Q)=\Vect(Q)$. For the other inclusion, pulling back via $\Spec L\arr \stZ$ we get a map $F\arr \Hl^0(\odi{(\Spec L)^{(\infty)}})=\End_{\Fdiv_k(L)}(1)=L'$. Using the map $\Fdiv_k(L)\arr \Strat_k(L)$ and \ref{main thm for str and inf}, 2) we see that $L'\subseteq L_{\et,k}$ as desired.
\end{proof}

\begin{proof}[Proof of Theorem \ref{main thm for FDiv}, 1), A$\then$C]  By \ref{when TX is Tannakian} and \ref{A implies B for Fdiv} $L=\Hl^0(\odi{\stZ^{(\infty)}})$ is a field. 
In what follows we will use the following notation. If $\stW$ is a category fibered in groupoids over $L$ we will use $\stW^{(i)}$ for $\stW^{(i,k)}$ for $i\in \N\cup\{\infty\}$ and denote by $\stW_{(i,L)}$, for $i\in \N$, the fibered category $\stW$ with $L$-structure $\stW\arrdi{F^i_\stW}\stW \xrightarrow{\pi} \Spec L$, where $F_\stW$ is the absolute Frobenius, $\pi$ is the structure map.

We need to show that the pullback functor $\Hom_L(\sZ^{(\infty)}, \Gamma)\to\Hom_L(\sZ,\Gamma)$ is an equivalence for a finite and \'etale stack $\Gamma$ over $L$. By \ref{limit over natural numbers} it is enough to prove that $\phi_{(i,\sZ)}^*:\Hom_L(\sZ^{(i)}, \Gamma)\to\Hom_L(\sZ,\Gamma)$ is an equivalence for all $i$, where $\sZ^{(i)}$ is equipped with the $L$-structure via $\sZ^{(\infty)}$ and $\phi_{(i,\stZ)}$ is the relative Frobenius. Denote by $\varphi\colon \stZ^{(i)}\arr \stZ$ the projection and consider the following $2$-commutative diagram.
  \[
  \begin{tikzpicture}[xscale=3.5,yscale=-1.2]
    \node (A0_1) at (1, 0) {$(\stZ)_{(i,L)}$};
    \node (A0_2) at (2, 0) {$(\stZ^{(i)})_{(i,L)}$};
    \node (A0_3) at (3, 0) {$\stZ$};
    \node (A1_0) at (0, 1) {$(\stZ^{(i)})_{(i,L)}$};
    \node (A1_1) at (1, 1) {$\stZ$};
    \node (A1_2) at (2, 1) {$\stZ^{(i)}$};
    \path (A0_1) edge [->]node [auto] {$\scriptstyle{F^i_\stZ}$} (A1_1);
    \path (A1_0) edge [->]node [below] {$\scriptstyle{\varphi}$} (A1_1);
    \path (A0_1) edge [->]node [auto] {$\scriptstyle{\phi_{(i,\stZ)}}$} (A0_2);
    \path (A0_2) edge [->]node [auto] {$\scriptstyle{F^i_{\stZ^{(i)}}}$} (A1_2);
    \path (A1_1) edge [->]node [below] {$\scriptstyle{\phi_{(i,\stZ)}}$} (A1_2);
    \path (A0_2) edge [->]node [auto] {$\scriptstyle{\varphi}$} (A0_3);
  \end{tikzpicture}
  \]
We have $\varphi\circ \phi_{(i,\stZ)}=F^i_\stZ$,  $\phi_{(i,\stZ)}\circ \varphi = F^i_{\stZ^{(i)}}$, and that the morphisms in the above diagram are $L$-linear. The result then follows upon applying $\Hom_L(-,\Gamma)$ to the diagram,  provided that the following is true: If $\stW$ is a category fibered in groupoids over $L$ then $\Hom_L(\stW,\Gamma)\arrdi{F^i_\stW\circ -} \Hom_L(\stW_{(i,L)},\Gamma)$ is an equivalence. By construction $\Hom_L(\stW_{(i,L)},\Gamma) \simeq \Hom_L(\stW,\Gamma^{(i,L)})$ and a direct check shows that $F^i_\stW\circ -$ corresponds to the composition along the relative Frobenius $\Gamma\arr\Gamma^{(i,L)}$, which is an equivalence because $\Gamma$ is \'etale over $L$. 
%
%
\end{proof}

\begin{proof}[Proof of Theorem \ref{main thm for FDiv},1), A$\then$D and last sentence]
For quotient gerbes, please refer to definition \ref{quotient groups and gerbes}.
 By \ref{when TX is Tannakian} and \ref{A implies B for Fdiv} we have that $L=\Hl^0(\odi{\stZ^{(\infty,k)}})$ is a field and $\Pi_{\Fdiv(\stZ)}$ an $L$-gerbe. We must prove that, if $\Gamma$ is a quotient $L$-gerbe of $\Pi_{\Fdiv(\stZ)}$ of finite type, then $\Gamma$ is smooth banded. We have an $L$-map $\phi\colon \stZ^{(\infty,k)}\arr\Gamma$ such that $\phi^*\colon \Rep\Gamma\arr \Fdiv(\stZ)$ is fully faithful. Set $\overline\stZ=\stZ\times_k\overline{k}$ and $\overline\Gamma=\Gamma\times_k\overline k\arrdi\pi\Gamma$. Since $\overline\stZ^{(i,\overline k)}\simeq \stZ^{(i,k)}\times_k \overline k$, using the definition of limit it is easy to see that $\overline\stZ^{(\infty,\overline k)}\simeq \stZ^{(\infty,k)}\times_k \overline k$. Denote by $\overline\phi\colon\overline\stZ^{(\infty,\overline k)}\arr\overline\Gamma$ the base change of $\phi$. We claim that
 \[
 \overline\phi^*\colon \Vect(\overline\Gamma)\arr\Vect(\overline\stZ^{(\infty,\overline k)})=\Fdiv_{\overline k}(\overline\stZ)
 \]
 is fully faithful. Let $\overline V,\overline W\in\Vect(\overline{\Gamma})$. Since $\phi^*$ is faithful, it is enough to prove that
 \[
 \Hom_{\Gamma}(\pi_*\overline V,\pi_*\overline W)\arr \Hom_{\stZ^{(\infty,k)}}(\phi^*\pi_*\overline V,\phi^*\pi_*\overline W)
 \]
 is bijective. The pushforward $\pi_*\overline V$ can be written as a direct sum of vector bundles on $\Gamma$. Indeed let $k'/k$ be a finite extension for which there exists $V'\in\Vect(\Gamma\times_k k')$ inducing $\overline V$ and consider $\overline\Gamma\arrdi\alpha\Gamma\times_kk'\arrdi\beta\Gamma$. We have that $\pi_*\overline{V}=\beta_*(V'\otimes_{k'}\overline k)$, which is a direct sum of copies of $\beta_*V'$, and $\beta_*V'$ is a vector bundle because it is a coherent sheaf on $\Gamma$, which is an $L$-gerbe. Writing $\pi_*V=\bigoplus_i V_i$ and $\pi_*W=\bigoplus_j W_j$ and using that $\phi^*$ is fully faithful on vector bundles, the proof of the bijectivity of the above map translates into the following statement: given a collection of maps $\lambda_{i,j} \colon V_i\arr W_j$ for all $i,j$ such that $\phi^*\lambda_{i,j}$ induces a map $\mu \colon \bigoplus_i \phi^*V_i\arr \bigoplus_j \phi^*W_j$, then it also induces a map $\bigoplus_i V_i\arr \bigoplus_j W_j$. 
 If $\xi\colon \Spec B\arr \stZ^{(\infty,k)}$ is any object, since $\xi^*\mu$ is defined and $\Spec B$ is quasi-compact, we can conclude that for all $i$ the set $\{j\st \xi^*\phi^*\lambda_{i,j}\neq 0\}$ is finite. Since $\Spec B\arrdi{\phi\xi}\Gamma$ is faithfully flat, the same holds over $\Gamma$ and therefore the map $\bigoplus_i V_i\arr \bigoplus_j W_j$ is well defined.
 
As $\overline{k}$ is perfect, the absolute Frobenius of $\overline\stZ^{(\infty,\overline k)}$ is also an equivalence. By the discussion above, we conclude that $F^*\colon \Vect(\overline \Gamma)\arr \Vect(\overline \Gamma)$ is fully faithful, where $F$ is the absolute Frobenius of $\overline\Gamma$. We show that $u\colon \odi{\overline\Gamma}\arr F_*\odi{\overline\Gamma}$ is surjective. For all $V\in\Vect(\overline\Gamma)$ we have a bijection
 \[
 \Hom_{\overline\Gamma}(V,\odi{\overline\Gamma})\arr \Hom_{\overline\Gamma}(F^*V,F^*\odi{\overline\Gamma})\simeq \Hom_{\overline\Gamma}(V,F_*\odi{\overline\Gamma}) 
 \]
 which is induced by $u$. By \cite[Cor 3.9, pp. 132]{De3} and \ref{generation for affine and finite maps} the sheaf $F_*\odi{\overline\Gamma}$ is a quotient of a direct sum of vector bundles. This easily implies that $u$ is surjective.
 
 Recall that if $\stX$ is a category fibered in groupoids over a scheme $S$ and we set $\stX^{(1,S)}$ for the base change of $\stX\arr S$ along the absolute Frobenius of $S$, then the absolute Frobenius factors as $\stX\arr\stX^{(1,S)}\arr \stX$. Moreover if $T\arr S$ is a map and we apply $-\times_S T$ to the map $\stX\arr \stX^{(1,S)}$ we get $\stX\times_S T\arr (\stX\times_S T)^{(1,T)}$. The stack $\overline \Gamma=\Gamma\times_k \overline k$ is a stack over $S=\Spec(L\otimes_k \overline k)$. Thus the absolute Frobenius of $\overline\Gamma$ factors as $\overline\Gamma\arrdi\alpha \overline\Gamma^{(1,S)}\arrdi\beta \overline\Gamma$. Since $\beta$ is affine and $\odi{\overline \Gamma}\arr \beta_*\odi{\overline\Gamma^{(1,S)}}\arr \beta_*\alpha_*\odi{\overline\Gamma}=F_*\odi{\overline\Gamma}$ is surjective, we can conclude that $\odi{\overline\Gamma^{(1,S)}}\arr \alpha_*\odi{\overline\Gamma}$ is surjective. Since $\alpha$ is the base change of $\Gamma\arrdi \delta \Gamma^{(1,L)}$ along the flat map $S\arr \Spec L$, it also follow that $\odi{\Gamma^{(1,L)}}\arr \delta_*\odi\Gamma$ is surjective. In particular $\delta_*\odi\Gamma$ is of finite type and thus locally free, which implies that $\odi{\Gamma^{(1,L)}}\arr \delta_*\odi\Gamma$ is an isomorphism.
 Using \ref{quotient and faithful maps for gerbes}, (1) and \ref{when the factorization is affine} it follows that $\Gamma\arr \Gamma^{(1,L)}$ is a quotient.  
 We claim that this implies that $\Gamma$ is smooth banded. For this we can assume $L=k$ algebraically closed and $\Gamma=\Bi G$, for an affine group scheme $G$ of finite type over $k$. The relative Frobenius is a quotient means that $G\arr G^{(1)}$  is faithfully flat, which implies that $G$ is reduced and thus smooth.
\end{proof}

\begin{rmk}\label{prosmoothness when k is perfect}
If $k$ is perfect and $\stZ$ is any category fibered in groupoids over $k$ such that $\Fdiv(\stZ)$ is a $k$-Tannakian category then the relative Frobenius of $\Pi_{\Fdiv(\stZ)}$ is an equivalence and this implies that $\Pi_{\Fdiv(\stZ)}$ is pro-smooth banded.

Indeed we have commutative diagrams
  \[
  \begin{tikzpicture}[xscale=2.9,yscale=-1.2]
    \node (A0_0) at (0, 0) {$\stZ^{(\infty)}$};
    \node (A0_1) at (1, 0) {$\Pi_{\Fdiv(\stZ)}$};
    \node (A0_2) at (2, 0) {$\Vect(\Pi_{\Fdiv(\stZ)})$};
    \node (A0_3) at (3, 0) {$\Fdiv(\stZ)$};
    \node (A1_0) at (0, 1) {$\stZ^{(\infty)}$};
    \node (A1_1) at (1, 1) {$\Pi_{\Fdiv(\stZ)}$};
    \node (A1_2) at (2, 1) {$\Vect(\Pi_{\Fdiv(\stZ)})$};
    \node (A1_3) at (3, 1) {$\Fdiv(\stZ)$};
    \path (A0_0) edge [->]node [auto] {$\scriptstyle{}$} (A0_1);
    \path (A0_2) edge [->]node [auto] {$\scriptstyle{\simeq}$} (A0_3);
    \path (A1_0) edge [->]node [auto] {$\scriptstyle{}$} (A1_1);
    \path (A0_3) edge [->]node [auto] {$\scriptstyle{F_{\stZ^{(\infty)}}^*}$} (A1_3);
    \path (A0_2) edge [->]node [auto] {$\scriptstyle{F_{\Pi_{\Fdiv(\stZ)}}^*}$} (A1_2);
    \path (A0_0) edge [->]node [auto] {$\scriptstyle{F_{\stZ^{(\infty)}}}$} (A1_0);
    \path (A0_1) edge [->]node [auto] {$\scriptstyle{F_{\Pi_{\Fdiv(\stZ)}}}$} (A1_1);
    \path (A1_2) edge [->]node [auto] {$\scriptstyle{\simeq}$} (A1_3);
  \end{tikzpicture}
  \]
The absolute Frobenius of $\stZ^{(\infty)}$ is the limit of the absolute Frobenius of the $\stZ^{(i)}$. Using the description in \ref{Fdiv when k is perfect} we can interpret  $F_{\stZ^{(\infty)}}^*\colon \Fdiv(\stZ)\arr \Fdiv(\stZ)$ as a shift and thus conclude that it is an equivalence. Since $\Pi_{\Fdiv(\stZ)}$ is a $k$-gerbe, it follows that its absolute and relative Frobenius are equivalences. This implies that $\Pi_{\Fdiv(\stZ)}$ is pro-smooth banded. Indeed if $\Gamma$ is a quotient of finite type of $\Pi_{\Fdiv(\stZ)}$, its relative Frobenius $\Gamma\arr\Gamma^{(1)}$ is a quotient. It follows that $\Gamma$ is smooth banded arguing as in the end of the above proof.

If $k$ is not perfect we do not have the description of \ref{Fdiv when k is perfect} and it is unclear whether the relative Frobenius of $\Pi_{\Fdiv(\stZ)}$ is an equivalence or not. When $k$ is algebraically closed and $\stZ$ is a connected, locally Noetherian and regular scheme the above argument has already been used by dos Santos in \cite[Theorem 11]{dS}.
\end{rmk}

\section{The Local Quotient of the Nori Fundamental Gerbe} \label{local gerbe}

Let $k$ be a field of characteristic $p>0$, $\sX$ be a category fibered in groupoids over $k$ and denote by $F\colon\stX\arr\stX$ the absolute Frobenius. For $i\in \N$ denote by $\shD_i$ the category of triples $(\shF,V,\lambda)$ where $\shF\in\Vect(\stX)$, $V\in\Vect(k)$ and $\lambda\colon V\otimes_k\odi\stX \arr F^{i*}\shF$ is an isomorphism. The category $\shD_i$ is monoidal, rigid and $k$-linear via $k\arr \End_{\shD_i}(\odi\stX,k,\id)$, $x\longmapsto (x,x^{p^i})$. Moreover the association
\[
\shD_i\arr\shD_{i+1}\comma (\shF,V,\lambda)\longmapsto (\shF,F_k^*V,F^*\lambda)
\]
where $F_k$ is the absolute Frobenius of $k$, is $k$-linear and monoidal. We can therefore define
\[
\shD_\infty = \varinjlim_{i\in\N}\shD_i
\]

\begin{thm}\label{main theorem local gerbes}
Let $\sX$ be a reduced category fibered in groupoids over $k$. Then $\stX$ admits a Nori local fundamental gerbe over $k$ if and only if $\Hl^0(\odi\stX)$ does not contain non trivial purely inseparable field extensions of $k$. In this case $\shD_\infty$ is a $k$-Tannakian category and the map $\stX\arr\Pi_{\shD_\infty}$, induced by the forgetful functor $\shD_\infty\arr\Vect(\stX)$, is the pro-local Nori fundamental gerbe of $\stX$.

If $\Hl^0(\odi\stX)=k$ then $\Rep(\Pi^{\NN,\LL}_{\stX/k})\arr \Vect(\stX)$ is an equivalence onto the full subcategory of $\Vect(\stX)$ of sheaves $\shF$ such that $F_{\sX}^{i^*}\shF$ is free for some $i\in \N$.
\end{thm}
\begin{proof} 
The only if part in the first claim is very similar to the proof in \ref{existence of Nori etale gerbe}, taking into account that a finite and purely inseparable field extension is a finite and local stack. For the if part it is enough to show the remaining claims in the statement. 
We apply \ref{main theorem formal II} on the map $\stX\arr\stX_\shT=\Spec k$, which satisfies axiom $A$ and $L=\Hl^0(\odi{\stX_\shT})=k$ is a field. We have $\shT_i(\stX)=\shD_i$ for all $i\in \N\cup\{\infty\}$ and that $\stX\arr (\Pi_{\shD_\infty})_{\LL}$ is the local Nori fundamental gerbe of $\stX/L_\infty$. Since $L_0=k$ and $L_0\subseteq L_i$ are purely inseparable inside $\Hl^0(\odi\stX)$ we also have $L_i=k$ for all $i\in \N\cup\{\infty\}$. Thus it remains to show $\Pi_{\shD_\infty}$ is pro-local. Thanks to \ref{Frobenius for TiX}, for any $V\in\sD_\infty=\Vect(\Pi_{\shD_\infty})$ there exists an index $i\in \N$ such that $F_{\Pi_{\shD_\infty}}^{i^*}V$ is free,  where $F_{\Pi_{\shD_\infty}}$ is the absolute Frobenius of $\Pi_{\shD_\infty}$, and by \ref{main theorem formal II} plus the fact that $\shT(\stX)=\Vect(k)$ is made of finite objects, $V$ is essentially finite. Let $\Gamma$ be the monodromy gerbe of $V\in\shD_\infty$ (see \ref{pro stuff for gerbes}). Then the absolute Frobenius $F_\Gamma^i$ factors as $\Gamma\xrightarrow{\pi}\Spec(k)\to\Gamma$, where $\pi$ is the structure map of $\Gamma/k$. This implies immediately that $\Gamma$ is local.

In the last claim we have to show that $\shD_\infty\arr \Vect(\stX)$ is full. Actually one can easily check that $\shD_i\arr\Vect(\stX)$ is fully faithful for all $i\in \N$.
\end{proof}

\begin{rmk} In \cite{EH} Esnault and Hogadi did not go into the study of the local quotient of Nori's fundamental group. However, using their main theorem it is easily seen (under their assumptions) that the finite representations of the local quotient of Nori's fundamental group is the full Tannakian subcategory of $\sD_\infty$ consisting of the essentially finite objects. Now our Theorem \ref{main theorem local gerbes} grantees that any object in $\sD_\infty$ is essentially finite. 
\end{rmk}

\appendix
 \section{Limit of categories and fibered categories}
 \begin{defn}\label{direct limit of categories}
Let $I$ be  a filtered category. A directed system of categories indexed by $I$ is a pseudo-functor $\shD_* \colon I\arr (\text{Cat})$ \cite[Def. 3.10]{Vis}. Concretely this is the assignment of data $(\shD_i,\shD_\alpha,\lambda_{\alpha,\beta},\lambda_i)$: categories $\shD_i$ for all $i\in I$, functors $\shD_\alpha\colon \shD_i\arr\shD_j$ for all $i\arrdi\alpha j$ in $I$ and natural isomorphisms $\lambda_{\alpha,\beta}\colon \shD_\beta\circ \shD_\alpha \arr \shD_{\beta\alpha}$ for all composable arrows $i\arrdi \alpha j \arrdi \beta k$ and $\lambda_i \colon \shD_{\id_i}\arr \id_{\shD_i}$ for all $i\in I$. This data is subject to compatibility conditions (see  \cite[Def. 3.10]{Vis}).

We define the limit of $\shD_*$, written $\lim_{i\in I} \shD_i$ or $\shD_\infty$, in the following way. The category $\shD_\infty$ has pairs $(i,x)$, where $i\in I$ and $x\in \shD_i$, as objects. Given $(i,x),(j,y)\in\shD_\infty$ the set $\Hom_{\shD_\infty}((i,x)),(j,y))$ is the limit on the category of pairs $(i\arrdi \alpha k, j\arrdi \beta k)$ (which is a filtered category) of the sets $\Hom_{\shD_k}(\shD_\alpha(x),\shD_\beta(y))$. Composition is defined in the obvious way. For all $i\in I$ there are functors $F_i\colon\shD_i\arr \shD_\infty$, $F_i(x)=(i,x)$ and, for all $i\arrdi \alpha j$ in $I$, there are canonical isomorphisms $ \mu_\alpha \colon F_j\circ \shD_\alpha \arr F_i$.

Given a category $\shC$ we define the category $\shC^\shD$ in the following way. The objects are collections $(H_i,\delta_\alpha)_{i,i\arrdi \alpha j}$ where: $H_i \colon \shD_i\arr \shC$ are functors for all $i\in I$, $\delta_\alpha \colon H_j\circ \shD_\alpha \arr H_i$ are natural isomorphisms for all arrows $i\arrdi \alpha j$ in $I$. This data is subject to the following compatibilities. For all $i\in I$ we have $\delta_{\id_i}=H_i\circ \lambda_i\colon H_i\circ \shD_{\id_i}\arr H_i$. For all composable arrows $i\arrdi \alpha j\arrdi \beta k$ the following diagram commutes
  \[
  \begin{tikzpicture}[xscale=3.1,yscale=-1.2]
    \node (A0_0) at (0, 0) {$H_k\circ \shD_{\beta\alpha}$};
    \node (A0_1) at (1, 0) {$H_i$};
    \node (A1_0) at (0, 1) {$H_k\circ\shD_\beta\circ\shD_\alpha$};
    \node (A1_1) at (1, 1) {$H_j\circ\shD_\alpha$};
    \path (A0_0) edge [->]node [auto] {$\scriptstyle{\delta_{\beta\alpha}}$} (A0_1);
    \path (A0_0) edge [<-]node [left] {$\scriptstyle{H_k\circ \lambda_{\alpha,\beta}}$} (A1_0);
    \path (A1_0) edge [->]node [auto] {$\scriptstyle{\delta_\beta\circ \shD_\alpha}$} (A1_1);
    \path (A1_1) edge [->]node [auto] {$\scriptstyle{\delta_\alpha}$} (A0_1);
  \end{tikzpicture}
  \]
The arrows in $\shC^\shD$ are the obvious ones.
  
Given a functor $G\colon \shC\arr \shC'$ one can easily define a functor $G^\shD\colon \shC^\shD \arr \shC'^\shD$. Moreover the data $(F_i,\mu_\alpha)$ defined above is an object of $\shD_\infty^\shD$. In particular we obtain a functor
\[
\chi_\shC\colon \Hom(\shD_\infty,\shC)\arr \shC^\shD\comma (\shD_\infty \arrdi G \shC) \longmapsto G^\shD(F_i,\mu_\alpha)
\]
\end{defn}

\begin{prop}\label{hom from the limit}
 The functor $\chi_\shC$ in \ref{direct limit of categories} is an isomorphism of categories.
\end{prop}
\begin{proof}
Let's define a functor $\iota\colon \shC^\shD\arr \Hom(\shD_\infty,\shC)$. Given $a=(H_i,\delta_\alpha)\in \shC^\shD$ defines $\iota(a)\colon \shD_\infty\arr \shC$ as follows. For $(i,x)\in \shD_\infty$ set $\iota(a)(i,x)=H_i(x)$. For $\phi\colon (i,x)\arr (j,y)$ in $\shD_\infty$ choose $i\arrdi f k,j\arrdi g k$ such that $\phi$ is induced by the arrow $v\colon \shD_f(x)\arr \shD_g(y)$ in $\shD_k$. Set $\iota(a)(\phi)$ as the only dashed arrow making the following diagram commutative
  \[
  \begin{tikzpicture}[xscale=3.1,yscale=-1.2]
    \node (A0_0) at (0, 0) {$H_k\circ D_f(x)$};
    \node (A0_1) at (1, 0) {$H_k \circ D_g(y)$};
    \node (A1_0) at (0, 1) {$H_i(x)$};
    \node (A1_1) at (1, 1) {$H_j(y)$};
    \path (A0_0) edge [->]node [auto] {$\scriptstyle{H_k(v)}$} (A0_1);
    \path (A0_0) edge [->]node [auto] {$\scriptstyle{\delta_f}$} (A1_0);
    \path (A0_1) edge [->]node [auto] {$\scriptstyle{\delta_g}$} (A1_1);
    \path (A1_0) edge [->,dashed]node [auto] {$\scriptstyle{}$} (A1_1);
  \end{tikzpicture}
  \]
A direct check shows that this arrows does not depend on the choices of $f,g,v$. In particular $\iota(a)$ is easily seen to be a functor $\shD_\infty\arr \shC$. The action of $\iota$ on arrows is the obvious one: the required compatibilities follows from the compatibilities of arrows in $\shC^\shD$. In conclusion one get a functor $\iota\colon \shC^\shD\arr \Hom(\shD_\infty,\shC)$. The equality $\chi_\shC \circ \iota = \id$ can be checked directly.

For the converse let $G\colon \shD_\infty\arr \shC$ be a functor. We have $\chi_C(G)=(G\circ F_i,G\circ \mu_\alpha)$ and set $\widetilde G=\iota(\chi_C(G))$. We must show that $G=\widetilde G$. For $i\in I$ and $x\in \shD_i$ we have $G(i,x)=G(F_i(x))=\widetilde G(x)$. Let now $\phi\colon (i,x)\arr (j,y)$ be an arrow in $\shD_\infty$ and $i\arrdi \alpha k,j\arrdi \beta k$ arrows, $v\colon \shD_\alpha(x)\arr \shD_\beta(y)$ inducing $\phi$. This can be expressed in the following commutative diagram
  \[
  \begin{tikzpicture}[xscale=3.1,yscale=-1.2]
    \node (A0_0) at (0, 0) {$(k,\shD_\alpha(x))$};
    \node (A0_1) at (1, 0) {$(i,x)$};
    \node (A1_0) at (0, 1) {$(k,\shD_\beta(y))$};
    \node (A1_1) at (1, 1) {$(j,y)$};
    \path (A0_0) edge [->]node [auto] {$\scriptstyle{\mu_\alpha(x)}$} (A0_1);
    \path (A1_0) edge [->]node [auto] {$\scriptstyle{\mu_\beta(y)}$} (A1_1);
    \path (A0_1) edge [->]node [auto] {$\scriptstyle{\phi}$} (A1_1);
    \path (A0_0) edge [->]node [auto] {$\scriptstyle{F_k(v)}$} (A1_0);
  \end{tikzpicture}
  \]
We have $G(F_k(v))=\widetilde G(F_k(v))$, $G\circ \mu_\alpha =\widetilde G \circ \mu_\alpha$ and $G\circ \mu_\beta=\widetilde G \circ \mu_\beta$ by construction. It follows that $G(\phi)=\widetilde G(\phi)$.
\end{proof}

\begin{rmk}\label{limit over natural numbers}
 When $I=\N$ with the usual order a directed system $\shD_*$ of categories indexed by $\N$ is just an infinite sequence of categories and functors:
 \[
 \shD_0 \arrdi{G_0} \shD_1 \arrdi{G_1} \shD_2 \arrdi{G_2} \cdots 
 \]
 Moreover if $\shC$ is a category then $\shC^\shD$ is equivalent to the category whose objects are tuples $(H_n,\sigma_n)$ where: $H_n\colon \shD_n\arr \shC$ is a functor, $\sigma_n \colon H_{n+1}\circ G_n \arr H_n$ a natural isomorphism.
\end{rmk}

Let $\shD_*$ be a direct system of categories indexed by $I$. We have the following fact which are easy to check:
\begin{itemize}
 \item If for all arrows $\alpha$ in $I$ the functor $\shD_\alpha$ is faithful (resp. fully faithful, equivalence) then for all $i\in I$ the functor $F_i$ is faithful (resp. fully faithful, equivalence);
 \item If for all $i\in I$ the category $\shD_i$ is a groupoid then $\shD_\infty$ is a groupoid;
 \item If $R$ is a ring, for all $i\in I$ the category $\shD_i$ is $R$-linear and for all arrows $\alpha$ in $I$ the functor $\shD_\alpha$ is $R$-linear then $\shD_\infty$ is naturally an $R$-linear category and for all $i\in I$ the functor $F_i$ is $R$-linear;
 \item If for all $i\in I$ the category $\shD_i$ is abelian and for all arrows $\alpha$ the functor $\shD_\alpha$ is additive and exact, then $\shD_\infty$ is an abelian category and for all $j\in I$ the functor $F_j$ is also additive and exact.
 \item If for all $i\in I$ the category $\shD_i$ is monoidal and for all arrows $\alpha,\beta$ in $I$ the functor $D_\alpha$ has a monoidal structure and the $\lambda_{\alpha,\beta}$ are monoidal then we can endow $\shD_\infty$ and, for all $i\in I$, $F_i$ with a monoidal structure in the following way. Given $i,j\in I$ choose $k_{i,j}\in I$, maps $i\arrdi{ \alpha_{i,j}} k_{i,j}$, $j\arrdi{ \beta_{i,j}} k_{i,j}$ and define
 \[
 (i,x)\otimes (j,y) = (k_{i,j}, \shD_{\alpha_{i,j}}(x)\otimes_{\shD_{k_{i,j}}} \shD_{\beta_{i,j}}(y))
 \]
 and $(i_0,1_{\shD_{i_0}})$ as unit for a chosen $i_0\in I$. All the maps required in order to have a monoidal structure are easy to define.
\end{itemize}

\begin{prop}\label{limit of fibered categories}
 Let $\shC$ be a category with fiber products, $I$ be a filtered category and $\stX_*$ be a directed system of fibered categories over $\shC$, that is a direct system of categories $\stX_*$ given by data $(\stX_i,\stX_\alpha,\lambda_{\alpha,\beta},\lambda_i)$ such all $\pi_i\colon \stX_i\arr \shC$ are fibered categories, all $\stX_\alpha\colon \stX_i\arr \stX_j$ are maps of fibered categories and all $\lambda_{\alpha,\beta}$, $\lambda_i$ are base preserving natural transformations. Then the induced functor $\stX_\infty\arr \shC$ makes $\stX_\infty$ into a fibered category, the functor $F_i\colon \stX_i\arr \stX_\infty$ are maps of fibered categories and $\mu_\alpha$ are base preserving natural transformations.
 Moreover if all $\stX_i$ are fibered in groupoids (resp. sets) then so is $\stX_\infty$.
 
 If $c\in \shC$ then the direct system $\stX_*$ induces a direct system of categories $\stX(c)_*\colon I\arr (\text{cat})$ and the $F_i\colon \stX_i \arr \stX_\infty$ and the natural transformations $\mu_\alpha$ induces an equivalence
 \[
 \stX(c)_\infty \simeq \stX_\infty (c)
 \]
 
 If $\stY$ is another fiber category over $\shC$ then $\chi_\stX$ restricts to an isomorphism between $\Hom_\shC(\stX_\infty,\stY)$ and the full subcategory of $\stY^\stX$ of objects $(H_i,\delta_\alpha)$ such that $H_i$ are maps of fibered categories and the $\delta_\alpha$ are base preserving natural transformations.
\end{prop}
\begin{proof}
 We have that $(\pi_i,\omega_\alpha)\in \shC^\stX$, where we set $\omega_\alpha=\id$ for all $\alpha$, because the $\pi_i$ strictly commutes with the $\stX_\alpha$. We therefore get a functor $\pi_\infty\colon \stX_\infty\arr \shC$ such that $\pi_i=\pi_\infty\circ F_i$ and $\pi_\infty( \mu_\alpha)=\id$. The first equation assures that the $F_i$ strictly commutes over $\shC$, the second assures that the $\mu_\alpha$ are base preserving natural transformations. Moreover it is easy to see that the $F_i$ map Cartesian arrows to Cartesian arrows, which in particular implies that $\stX_\infty$ is a fibered category.
 
 The system $\stX_*$ together with the structure morphisms $\pi_i$ can be seen as a pseudo-functor from $I$ to the $2$-category $\Fib(\shC)$ of fibered categories over $\shC$. Given $c\in \shC$ the evaluation in $c$ yields a functor $\Fib(\shC)\arr (\text{cat})$ and, composing, we obtain the direct system $\stX(c)_*$. It is easy to see that $\stX_\infty(c)$ and $\stX(c)_\infty$ are the same categories. In particular if all $\stX_i$ are fibered in groupoids (resp. sets) then so is $\stX_\infty$.
 
 Let $G\colon \stX_\infty\arr \stY$ any functor and $\chi_\stX(G)=(G\circ F_i,G\circ \mu_\alpha)\in \stY^\stX$. It is easy to see that $G$ is base preserving if and only if the $G\circ F_i$ and $G( \mu_\alpha)$ are base preserving. In this case, assuming that the $G\circ F_i$ preserve Cartesian arrows, we have to show that $G$ does the same. This follows from the fact that a Cartesian arrow $\gamma$ in $\stX_\infty$ is, up to isomorphism, determined by the target of $\gamma$ and $\pi_\infty(\gamma)$, which implies that $\gamma$ is image of a Cartesian arrow in some $\stX_i$. 
\end{proof}

\section{Affine gerbes and Tannakian categories}

Let $k$ be a field. In this appendix we collect useful results about affine gerbes and Tannakian categories.
Recall that an affine gerbe $\Gamma$ over $k$ is a gerbe for the fpqc topology $\Gamma\arr \Aff/k$ with affine diagonal. If $L/k$ is a field extension and $\xi\in \Gamma(L)$ then $\Gamma$ is affine if and only if $\Autsh_\Gamma(\xi)$ is an affine scheme. Moreover any map from a scheme $X\arr \Gamma$ is an fpqc covering which is affine if $X$ is affine. (See \cite[Prop 3.1]{BV} for details.)

A $k$-Tannakian category is a $k$-linear, monoidal, rigid and abelian category $\shC$ such that $\End_\shC(1_\shC)=k$ (where $1_\shC$ is the unit) and there exists a field extension $L/k$ and a $k$-linear, exact and monoidal functor $\shC\arr \Vect L$. 

Classical Tannaka's duality states that the functors $\Vect(-)$ and $\Pi_*$ between the $2$-categories of affine gerbes over $k$ and $k$-Tannakian categories are ``quasi-inverses'' of each other. See Section \ref{general Tannaka duality} for the definition of $\Pi_*$ and of the natural functors $\shC \arr \Vect(\Pi_\shC)$ and $\Gamma\arr \Pi_{\Vect(\Gamma)}$.

Given an affine gerbe $\Gamma$ we will often use the notation $\Rep\Gamma$ instead of $\Vect(\Gamma)$.

\begin{defn}\label{quotient groups and gerbes}
A map of affine group schemes $G\arr G'$ over $k$ is a quotient if it is faithfully flat or equivalently if $\Hl^0(\odi{G'})\arr \Hl^0(\odi G)$ is injective (see \cite[Chapter 14]{Wat}).

A map of affine gerbes $\Gamma\arrdi\phi \Gamma'$ over $k$ is a quotient (resp. faithful) if there exists a field $L$ and $\xi\in \Gamma(L)$ such that the map of affine group schemes $\Autsh_\Gamma(\xi)\arr \Autsh_{\Gamma'}(\phi(\xi))$ is a quotient (a monomorphism or equivalently a closed immersion by \cite[Section 15.3]{Wat}). This notion does not depend on the choice of $\xi$ and $L$. Moreover $\phi$ is faithful if and only if it is faithful as a functor.
\end{defn}

\begin{prop}\label{quotient and faithful maps for gerbes}
 Let $\phi\colon \Gamma\arr\Gamma'$ be a map of affine gerbes. Then
 \begin{enumerate}
  \item the map $\odi{\Gamma'}\arr\phi_*\odi\Gamma$ is an isomorphism if and only if $\phi^*\colon \Rep \Gamma'\arr \Rep \Gamma$ is fully faithful;
  \item the following are equivalent: a) $\phi$ is a quotient; b) $\phi$ is a relative gerbe; c) the functor $\phi^*\colon \Rep\Gamma'\arr\Rep\Gamma$ is fully faithful and its image is stable under quotients;
  \item  the functor $\phi$ is faithful if and only if all $V\in\Rep\Gamma$ is a subquotient of $\phi^*W$ for some $W\in \Rep\Gamma'$.
 \end{enumerate}
\end{prop}
\begin{proof}
 For $(1)$, the map $\rho\colon \odi{\Gamma'}\arr\phi_*\odi\Gamma$ induces maps
 \[
 \Hom_{\Gamma'}(V,W)\arr \Hom_{\Gamma'}(V,W\otimes \phi_*\odi\Gamma)\simeq \Hom_\Gamma(\phi^*V,\phi^*W)\text{ for }V,W\in \Rep \Gamma'
 \]
 So if $\rho$ is an isomorphism then $\phi^*$ is fully faithful. Conversely assume the above map bijective for all $V,W$ and choose $W=\odi{\Gamma'}$. The map $\rho$ is injective since $\phi$ is faithfully flat. The surjectivity follows using that $\Rep\Gamma$ generates $\QCoh(\Gamma)$ by \cite[Cor 3.9, pp. 132]{De3}.
 
 For $(2), a)\iff c)$ and $(3)$ see \cite[3.3.3 c), pp. 205]{Saa}. For $(2), a)\iff b)$ we can assume $\Gamma=\Bi G$, $\Gamma'=\Bi G'$ and $\phi$ induced by $G\arr G'$. If $\phi$ is a quotient then $\Bi G \times_{\Bi G'} \Spec k \simeq \Bi K$, where $K$ is the kernel of $G\arr G'$, and thus $\phi$ is a relative gerbe. For the converse, one can replace $\Gamma$ by the image of $G\arr G'$ and assume $G\subseteq G'$ a closed subgroup. In this case $\Bi G \times_{\Bi G'}\Spec k \simeq G'/G$ and $G'/G\arr \Spec k$ is a gerbe if and only if it is an isomorphism, that is $G'=G$.
\end{proof}
\begin{defn}
 Given a Tannakian category $\shC$ a full Tannakian subcategory of $\shC$ is a sub-abelian, sub-monoidal and rigid full subcategory $\shD \subseteq \shC$ which is stable under quotients (in other words is the image of a functor $\Rep\Gamma'\arr \shC$ induced by a quotient map $\Pi_\shC\arr \Gamma'$.
 
 Given a subset $T$ of objects of $\shC$ we denote by $\langle T \rangle$ the full subcategory of $\shC$ whose objects are subquotients of objects of the form $P(X)$ or $P(X^\vee)$ for $X\in T$ and $P\in \N[t]$. It is easy to see that $\langle T \rangle$ is the smallest full Tannakian subcategory of $\shC$ containing $T$. For this reason we call $\langle T \rangle$ the sub Tannakian category spanned by $T$. 
\end{defn}

\begin{defn}\label{universal factorization for map of gerbes}
If $\phi\colon \Gamma\arr \Gamma'$ is a map of affine gerbe there exists a unique (up to a unique isomorphism) factorization of $\phi$ as $\Gamma\arrdi\alpha \Delta\arrdi\beta \Gamma'$, where $\alpha$ is a quotient and $\beta$ is faithful. We call $\Delta$ the image of $\phi$.
\end{defn}


\begin{defn}
A finite gerbe over $k$ is an affine gerbe over $k$ which is a finite stack. An affine gerbe $\Gamma$ over $k$ is finite and \'etale (resp. local) if it is finite and \'etale (resp. local) in the sense of \ref{finite and etale stacks} (resp. \ref{finite and local stacks}). 
\end{defn}

\begin{prop} \label{gerbe finite and of finite type}
 Let $\Gamma$ be an affine gerbe over $k$, $L/k$ be a field extension and $\xi\in \Gamma(L)$.
 \begin{enumerate}
 \item The following conditions are equivalent: a) $\Gamma$ is an algebraic stack; b) $\Autsh_\Gamma(\xi)/L$ is of finite type; c) there exists $V\in \Rep \Gamma$ such that $\langle V \rangle = \Rep \Gamma$.
 \item The gerbe $\Gamma$ is finite if and only if there exists $V\in \Rep\Gamma$ generating $\QCoh(\Gamma)$ (see \ref{sheaves generating categories});
 \item The gerbe $\Gamma$ is finite (resp. finite and \'etale, finite and local) if and only if $\Autsh_\Gamma(\xi)/L$ is finite (resp. finite and \'etale, finite and local).
 \end{enumerate}
\end{prop}
\begin{proof}
 Implications $1)$, b)$\iff$c)$\then$a) follows from \cite[Chapter III, 3.3.1.1]{Saa} and fpqc descent. For a)$\then$b), we choose an fppf atlas $X\to\Gamma$ with $X$ a  $k$-scheme.  Since $X\times_{\Gamma}X$ is an fppf $X$-algebraic space, the map $X\times_{\Gamma}X\to X\times_kX$ is also  fppf. This implies that the diagonal of $\Gamma$ is fppf, whence the result.

 
 Item 2) is proved in \cite[Chapter III, 3.3.3 a)]{Saa}, while 3) follows from 1) and \ref{base change of etale for finite gerbes}.
\end{proof}

\begin{rmk}\label{when the factorization is affine}
 Let $\phi$ be a map of gerbes factorizing as $\Gamma\arrdi\alpha \Delta\arrdi\beta \Gamma'$, where $\alpha$ is a quotient and $\beta$ is faithful. If $\beta$ is affine then $\phi$ is a quotient if and only if $\phi^*\colon\Rep\Gamma'\arr \Rep\Gamma$ is fully faithful. Indeed in this last case also $\beta^*\colon \Rep \Gamma'\arr \Rep\Delta$ would be fully faithful, that is $\odi{\Gamma'}\simeq \beta_*\odi\Delta$ thanks to \ref{quotient and faithful maps for gerbes}: if $\beta$ is affine than it is an isomorphism. 
 
 The map $\beta$ is affine in the following cases: $\Delta$ is finite, for instance if $\Gamma$ or $\Gamma'$ is finite; $\Gamma$ is of finite type and $\phi$ is a relative Frobenius. Moreover, if $L/k$ is a field extension, $\xi\in\Gamma(L)$,  $v\colon G=\Autsh_\Gamma(\xi)\arr \Autsh_{\Gamma'}(\phi(\xi))=G'$ and $H$ its image, then $\beta$ is affine if and only if $G'/H$ is affine, which is true in the following cases: $H$ is normal in $G'$, for instance if $\Gamma'$ is abelian; $G'$ is of finite type and the closed immersion $H\arr G'$ is nilpotent.
 
 This can be proved when $L=k$ is algebraically closed, so that $\Gamma=\Bi G$, $\Gamma'=\Bi G'$ and $\phi$ is induced by $v\colon G\arr G'$. The map $\beta$ is $\Bi H\arr \Bi G'$ and we have a $2$-Cartesian diagram
   \[
  \begin{tikzpicture}[xscale=1.8,yscale=-1.2]
    \node (A0_0) at (0, 0) {$G'/H$};
    \node (A0_1) at (1, 0) {$\Spec k$};
    \node (A1_0) at (0, 1) {$\Bi H$};
    \node (A1_1) at (1, 1) {$\Bi G'$};
    \path (A0_0) edge [->]node [auto] {$\scriptstyle{}$} (A0_1);
    \path (A1_0) edge [->]node [auto] {$\scriptstyle{\beta}$} (A1_1);
    \path (A0_1) edge [->]node [auto] {$\scriptstyle{}$} (A1_1);
    \path (A0_0) edge [->]node [auto] {$\scriptstyle{}$} (A1_0);
  \end{tikzpicture}
  \]
  So $\beta$ is affine if and only if $G'/H$ is affine. This is the case if $H$ is finite (see \cite[\href{http://stacks.math.columbia.edu/tag/03BM}{03BM}]{SP}) or if $H$ is normal (see \cite[Section 16.3]{Wat}). If $H(k)=G'(k)$, as for the relative Frobenius, we have that $G'/H$ is an algebraic space of finite type and with only one rational section $p\in G'/H$. The complement of $p$ is an algebraic space of finite type without rational points and thus empty. Since quasi-separated algebraic spaces are generically schemes, we can conclude that $G'/H$ is a scheme of finite type over $k$ with just one point, thus a finite $k$-scheme.
\end{rmk}

\begin{defn}\label{pro stuff for gerbes}
 Given an affine gerbe $\Gamma$ over $k$ and $E\in\Vect(\Gamma)$, the \textit{monodromy gerbe} of $E$, denoted by $\Gamma_E$, is the gerbe corresponding to $\langle E \rangle$, or, equivalently, the image of the map $\Gamma\arr \Bi \GL_n$ induced by $E$ (where $n=\rk E$). By \ref{gerbe finite and of finite type} $\Gamma_E$ is of finite type over $k$.
 
 Let $\sC$ be a Tannakian category, we denote $\Ess(\sC)$ (resp. $\Et(\sC)$, ${\rm Loc}(\sC)$) the full sub category of $\sC$ consisting of objects with finite (resp. finite and \'etale, finite and local) monodromy gerbe.
 \end{defn}

 \begin{rmk}
 If $\shC$ is a Tannakian category then $\shD=\Ess(\sC)$ (resp. $\shD=\Et(\sC)$, $\shD={\rm Loc}(\sC)$) is a full Tannakian subcategory of $\shC$. Indeed $\shD$ is additive because, given $E,F\in \shC$, the monodromy gerbe of $E\oplus F$ is the image of $\Pi_\shC\arr (\Pi_\shC)_E \times_k (\Pi_\shC)_F$. Moreover notice that if $E,F\in\shC$ and $F\in \langle E \rangle$ then $(\Pi_\shC)_F$ is a quotient of $(\Pi_\shC)_E$. We conclude that $\shD$ is a full Tannakian subcategory of $\shC$ observing that: $\shD$ is monoidal because $E\otimes F\in \langle E\oplus F\rangle$; $\shD$ is abelian and stable under quotients because if $F$ is a quotient or a subobject of $E$ then $F\in \langle E\rangle$; $\shD$ is stable under duals because $E^\vee\in \langle E\rangle$.
 \end{rmk}
 
 \begin{defn}\label{essentially finite}\cite[Definition 7.7, pp. 21]{BV}
Let $\shC$ be an additive and monoidal category. An object $E\in\shC$ is called finite if there exist $f\neq g \in \N[X]$ polynomials with natural coefficients and an isomorphism $f(E)\simeq g(E)$, it is called essentially finite if it is a kernel of a map of finite objects of $\shC$. We denote by $\EF(\shC)$ the full subcategory of $\shC$ consisting of essentially finite objects. When $\shC$ is $k$-Tannakian the two definitions of $\EF(\shC)$ introduced agree thanks to \cite[Thm 7.9]{BV}, that is an object of $\shC$ is essentially finite if and only if it has finite monodromy gerbe.
\end{defn}

 \begin{defn}\label{pro stuff for gerbes 2}
 Let $\Gamma$ be an affine gerbe. We say that $\Gamma$ is profinite (resp. pro-\'etale, pro-local) if it is a filtered projective limit (in the sense of \cite[Section 3]{BV}) of finite (resp. finite and \'etale, finite and local) gerbes. We denote by $\widehat \Gamma$ (resp. $\Gamma_{\et}$, $\Gamma_{\LL}$) the quotient gerbe $\Pi_{\Ess(\Rep\Gamma)}$ (resp. $\Pi_{\Et(\Rep\Gamma)}$, $\Pi_{{\rm Loc}(\Rep\Gamma)}$) and call it the profinite (resp. pro-\'etale, pro-local) quotient of $\Gamma$. Notice that $\Gamma$ is profinite (resp. pro-\'etale, pro-local) if and only if $\Gamma=\widehat\Gamma$ (resp. $\Gamma=\Gamma_{\et}$, $\Gamma=\Gamma_{\LL}$) and, if $\Gamma$ is an affine gerbe over $k$, then $\widehat\Gamma=\Pi_{\Gamma/k}^\NN$,  $\Gamma_{\et}=\Pi_{\Gamma/k}^{\NN,\et}$ and $\Gamma_\LL = \Pi_{\Gamma/k}^{\NN,\LL}$.
 
 We say that $\Gamma$ is smooth (pro-smooth) banded if there exists $L/k$ field extension and $\xi\in\Gamma(L)$ such that $\Autsh_\Gamma(\xi)$ is a smooth group scheme over $L$ (a projective limit of smooth group schemes over $L$).  
\end{defn}

\begin{rmk}
 An affine gerbe $\Gamma$ is pro-smooth banded if and only if any finite type quotient of $\Gamma$ is smooth banded. The implication "$\Leftarrow$" follows from the fact that affine gerbes are projective limit of gerbes of finite type. For the other, we can reduce to the neutral case, so that one has to prove that if $v\colon G_\infty=\varprojlim_j G_j\arr G$ is a quotient, $G$ is of finite type and the $G_j$ are smooth then $G$ is smooth. But $v$ factors through a quotient map $G_j\arr G$. Since $G_j\arr G$ is faithfully flat and $G_j$ is smooth it follows that $G$ is smooth.
\end{rmk}

\end{document}

%% file: packages_and_functions.for_preamble.tex
\usepackage{extarrow}
\usepackage{enumerate}
\usepackage{fontenc}
\usepackage[T1]{fontenc}
\usepackage{graphicx}
\usepackage[colorlinks=true, linkcolor=blue]{hyperref}
\usepackage{latexsym}
\usepackage{mathrsfs}
\usepackage{mathtext}
\usepackage{tikz}
\usepackage{textcomp}
\usepackage{upgreek}
\usepackage{xy}

\usetikzlibrary{arrows}
\usetikzlibrary{matrix}
\usetikzlibrary{shapes}
\usetikzlibrary{snakes}
\usetikzlibrary{matrix}

\DeclareMathOperator{\Autsh}{\underline{\textup{Aut}}}
\DeclareMathOperator{\Bi}{\textup{B}}

\DeclareMathOperator{\Coker}{\textup{Coker}}
\DeclareMathOperator{\End}{\textup{End}}

\DeclareMathOperator{\GL}{\textup{GL}}

\DeclareMathOperator{\Hl}{\textup{H}}
\DeclareMathOperator{\Homsh}{\underline{\textup{Hom}}}

\DeclareMathOperator{\Imm}{\textup{Im}}

\DeclareMathOperator{\QCoh}{\textup{QCoh}}

\DeclareMathOperator{\car}{\textup{char}}

\DeclareMathOperator{\et}{\textup{et}}

\DeclareMathOperator{\pr}{\textup{pr}}
\DeclareMathOperator{\pt}{\textup{pt}}
\DeclareMathOperator{\rk}{\textup{rk}}

%% file: packages_and_functions.tex
\global\long\def\A{\mathbb{A}}

\global\long\def\Ab{(\textup{Ab})}

\global\long\def\C{\mathbb{C}}

\global\long\def\Cat{(\textup{cat})}

\global\long\def\Di#1{\textup{D}(#1)}

\global\long\def\E{\mathcal{E}}

\global\long\def\F{\mathbb{F}}

\global\long\def\GCov{G\textup{-Cov}}

\global\long\def\Gcat{(\textup{Galois cat})}

\global\long\def\Gfsets#1{#1\textup{-fsets}}

\global\long\def\Gm{\mathbb{G}_{m}}

\global\long\def\GrCov#1{\textup{D}(#1)\textup{-Cov}}

\global\long\def\Grp{(\textup{Grps})}

\global\long\def\Gsets#1{(#1\textup{-sets})}

\global\long\def\HCov{H\textup{-Cov}}

\global\long\def\MCov{\textup{D}(M)\textup{-Cov}}

\global\long\def\MHilb{M\textup{-Hilb}}

\global\long\def\N{\mathbb{N}}

\global\long\def\PGor{\textup{PGor}}

\global\long\def\PGrp{(\textup{Profinite Grp})}

\global\long\def\PP{\mathbb{P}}

\global\long\def\Pj{\mathbb{P}}

\global\long\def\Q{\mathbb{Q}}

\global\long\def\RCov#1{#1\textup{-Cov}}

\global\long\def\RR{\mathbb{R}}

\global\long\def\Sch{\textup{Sch}}

\global\long\def\WW{\textup{W}}

\global\long\def\Z{\mathbb{Z}}

\global\long\def\acts{\curvearrowright}

\global\long\def\alA{\mathscr{A}}

\global\long\def\alB{\mathscr{B}}

\global\long\def\arr{\longrightarrow}

\global\long\def\arrdi#1{\xlongrightarrow{#1}}

\global\long\def\catC{\mathscr{C}}

\global\long\def\catD{\mathscr{D}}

\global\long\def\catF{\mathscr{F}}

\global\long\def\catG{\mathscr{G}}

\global\long\def\comma{,\ }

\global\long\def\covU{\mathcal{U}}

\global\long\def\covV{\mathcal{V}}

\global\long\def\covW{\mathcal{W}}

\global\long\def\duale#1{{#1}^{\vee}}

\global\long\def\fasc#1{\widetilde{#1}}

\global\long\def\fsets{(\textup{f-sets})}

\global\long\def\iL{r\mathscr{L}}

\global\long\def\id{\textup{id}}

\global\long\def\la{\langle}

\global\long\def\odi#1{\mathcal{O}_{#1}}

\global\long\def\ra{\rangle}

\global\long\def\set{(\textup{Sets})}

\global\long\def\sets{(\textup{Sets})}

\global\long\def\shA{\mathcal{A}}

\global\long\def\shB{\mathcal{B}}

\global\long\def\shC{\mathcal{C}}

\global\long\def\shD{\mathcal{D}}

\global\long\def\shE{\mathcal{E}}

\global\long\def\shF{\mathcal{F}}

\global\long\def\shG{\mathcal{G}}

\global\long\def\shH{\mathcal{H}}

\global\long\def\shI{\mathcal{I}}

\global\long\def\shJ{\mathcal{J}}

\global\long\def\shK{\mathcal{K}}

\global\long\def\shL{\mathcal{L}}

\global\long\def\shM{\mathcal{M}}

\global\long\def\shN{\mathcal{N}}

\global\long\def\shO{\mathcal{O}}

\global\long\def\shP{\mathcal{P}}

\global\long\def\shQ{\mathcal{Q}}

\global\long\def\shR{\mathcal{R}}

\global\long\def\shS{\mathcal{S}}

\global\long\def\shT{\mathcal{T}}

\global\long\def\shU{\mathcal{U}}

\global\long\def\shV{\mathcal{V}}

\global\long\def\shW{\mathcal{W}}

\global\long\def\shX{\mathcal{X}}

\global\long\def\shY{\mathcal{Y}}

\global\long\def\shZ{\mathcal{Z}}

\global\long\def\st{\ | \ }

\global\long\def\stA{\mathcal{A}}

\global\long\def\stB{\mathcal{B}}

\global\long\def\stC{\mathcal{C}}

\global\long\def\stD{\mathcal{D}}

\global\long\def\stE{\mathcal{E}}

\global\long\def\stF{\mathcal{F}}

\global\long\def\stG{\mathcal{G}}

\global\long\def\stH{\mathcal{H}}

\global\long\def\stI{\mathcal{I}}

\global\long\def\stJ{\mathcal{J}}

\global\long\def\stK{\mathcal{K}}

\global\long\def\stL{\mathcal{L}}

\global\long\def\stM{\mathcal{M}}

\global\long\def\stN{\mathcal{N}}

\global\long\def\stO{\mathcal{O}}

\global\long\def\stP{\mathcal{P}}

\global\long\def\stQ{\mathcal{Q}}

\global\long\def\stR{\mathcal{R}}

\global\long\def\stS{\mathcal{S}}

\global\long\def\stT{\mathcal{T}}

\global\long\def\stU{\mathcal{U}}

\global\long\def\stV{\mathcal{V}}

\global\long\def\stW{\mathcal{W}}

\global\long\def\stX{\mathcal{X}}

\global\long\def\stY{\mathcal{Y}}

\global\long\def\stZ{\mathcal{Z}}

\global\long\def\then{\ \Longrightarrow\ }

\global\long\def\L{\textup{L}}

\global\long\def\l{\textup{l}}